\documentclass[10pt]{article}
\usepackage{bbm}
\usepackage{amsmath}
\usepackage{amsthm}
\usepackage{tabularx}
\usepackage{enumitem}
\usepackage{amssymb} 
\usepackage{mathrsfs}
\usepackage{caption}
\usepackage{graphicx}
\usepackage{authblk}
\usepackage[hypertexnames=false,colorlinks=true,urlcolor=blue,linkcolor=blue,citecolor=blue]{hyperref}
\usepackage[numbers,comma,square,sort&compress]{natbib}
\usepackage[letterpaper,text={6.5in,9in},centering]{geometry}

\makeatletter\@addtoreset{equation}{section}\makeatother
\renewcommand{\theequation}{\arabic{section}.\arabic{equation}}

\newtheorem{thm}{Theorem}[section] 
\newtheorem{lem}[thm]{Lemma}  
 
\newtheorem{prop}[thm]{Proposition}  
\newtheorem{hyp}{Hypothesis}

\newtheorem{rmk}[thm]{Remark}

 {\begin{trivlist}\item[]\textbf{Acknowledgments.}}{\end{trivlist}}
\newcommand{\R}{\mathbb{R}}

\newcommand{\drm}{\mathrm{d}}

\usepackage{tikz}

\graphicspath{{Figures_PDF/}}

\begin{document}

\title{Approximate Localised Dihedral Patterns Near a Turing Instability}
\author[1,3]{Dan J. Hill}
\author[2]{Jason J. Bramburger}
\author[1]{David J.B. Lloyd}

\affil[1]{\small Department of Mathematics, University of Surrey, Guildford, GU2 7XH, UK}

\affil[2]{\small Department of Mathematics and Statistics, Concordia University, Montr\'eal, QC H3G 1M8, Canada}

\affil[3]{\small Fachrichtung Mathematik, Universit\"at des Saarlandes, Postfach 151150, 66041 Saarbr\"ucken, Germany}

\date{}
\maketitle

\begin{abstract}
\noindent Fully localised patterns involving cellular hexagons or squares have been found experimentally and numerically in various continuum models. However, there is currently no mathematical theory for the emergence of these localised cellular patterns from a quiescent state. A key issue is that standard techniques for one-dimensional patterns have proven insufficient for understanding localisation in higher dimensions. In this work, we present a comprehensive approach to this problem by using techniques developed in the study of radially-symmetric patterns. Our analysis covers localised planar patterns equipped with a wide range of dihedral symmetries, thereby avoiding a restriction to solutions on a predetermined lattice. The context in this paper is a theory for the emergence of such patterns near a Turing instability for a general class of planar reaction-diffusion equations. Posing the reaction-diffusion system in polar coordinates, we carry out a finite-mode Fourier decomposition in the angular variable to yield a large system of coupled radial ordinary differential equations.  We then utilise various radial spatial dynamics methods, such as  invariant manifolds, rescaling charts, and normal form analysis, leading to an algebraic matching condition for localised patterns to exist in the finite-mode reduction. This algebraic matching condition is nontrivial, which we solve via a combination of by-hand calculations and Gr\"obner bases from polynomial algebra to reveal the existence of a plethora of localised dihedral patterns. These results capture the essence of the emergent localised hexagonal patterns witnessed in experiments. Moreover, we combine computer-assisted analysis and a Newton-Kantorovich procedure to prove the existence of localised patches with $6m$-fold symmetry for arbitrarily large Fourier decompositions. This includes the localised hexagon patches that have been elusive to analytical treatment.


\end{abstract}
%
%
%
%
%
%
%
%
%
%
\section{Introduction}\label{s:intro}
In this work, we are interested in stationary localised patches of planar patterns embedded in a quiescent state bifurcating from a Turing instability. These patterns have the fascinating property that outside of some compact region in the plane they resemble a homogeneous, or background, state, while inside the compact region they can take on intricate and striking spatial arrangements. Particular examples of such spatial arrangements that continually arise in applications are those of cellular hexagons and squares, as illustrated in Figure~\ref{fig:IntroFig}(a). These localised structures are known to occur in the quadratic-cubic Swift-Hohenberg equation  (SHE)~\cite{sakaguchi1996,sakaguchi1997stable,lloyd2008localized} given by   
\begin{align}\label{e:SH}
    u_t &= -\left(1 + \Delta\right)^{2}u - \mu u + \gamma u^{2} - u^{3},
\end{align} 
where $u = u(r,\theta)$ is a function of the polar coordinates $(r,\theta)$ in the plane, $\Delta:=\left(\partial_{rr} + \frac{1}{r}\partial_{r} + \frac{1}{r^{2}}\partial_{\theta\theta}\right)$ is the polar Laplacian operator, $0<\mu\ll1$ acts as the bifurcation parameter, and $\gamma\in\mathbb{R}\setminus\{0\}$ is fixed, as well as in two-component reaction-diffusion (RD) systems near a Turing instability of the form 
\begin{equation}\label{e:RDsys}
    \mathbf{u}_t = \mathbf{D}\Delta\mathbf{u} - \mathbf{f}(\mathbf{u},\mu),\qquad \mathbf{u}\in\mathbb{R}^2,
\end{equation}
where $\mathbf{D}$ is the diffusion matrix, $\mathbf{f}$ is a nonlinear function and $\mu$ is the bifurcation parameter. For example, in the von Hardenberg RD model for dryland vegetation \cite{vonHardenberg2001} one can find localised hexagon patches, as depicted in Figure~\ref{fig:IntroFig}(b), which represent the density of vegetation in a water-scarce environment. Such localised cellular patterns are similarly found for other models of vegetation in arid climates~\cite{JAIBI2020vegetation,alnahdi2018localized}, nonlinear optics~\cite{Vladimirov2002Clusters,mcsloy2002computationally,menesguen2006optical}, phase-field crystals~\cite{Subramanian2018localizedPFC,Ophaus2021localizedPFC}, water waves~\cite{buffoni2018variational,buffoni2021localisedinfinitedepth}, neural field equations~\cite{rankin2014}, granular dynamics~\cite{umbanhowar1996localized,Aranson1998localizedGranular}, binary fluid convection~\cite{LoJacono2013localizedBinaryFluid}, and peaks on the surface of a ferrofluid~\cite{sakaguchi1997stable,lloyd2015homoclinic}. Despite the prevalence and importance of localised planar patterns, little is known about them from a mathematical perspective.  

\begin{figure}[t!] 
    \centering
    \includegraphics[width=\linewidth]{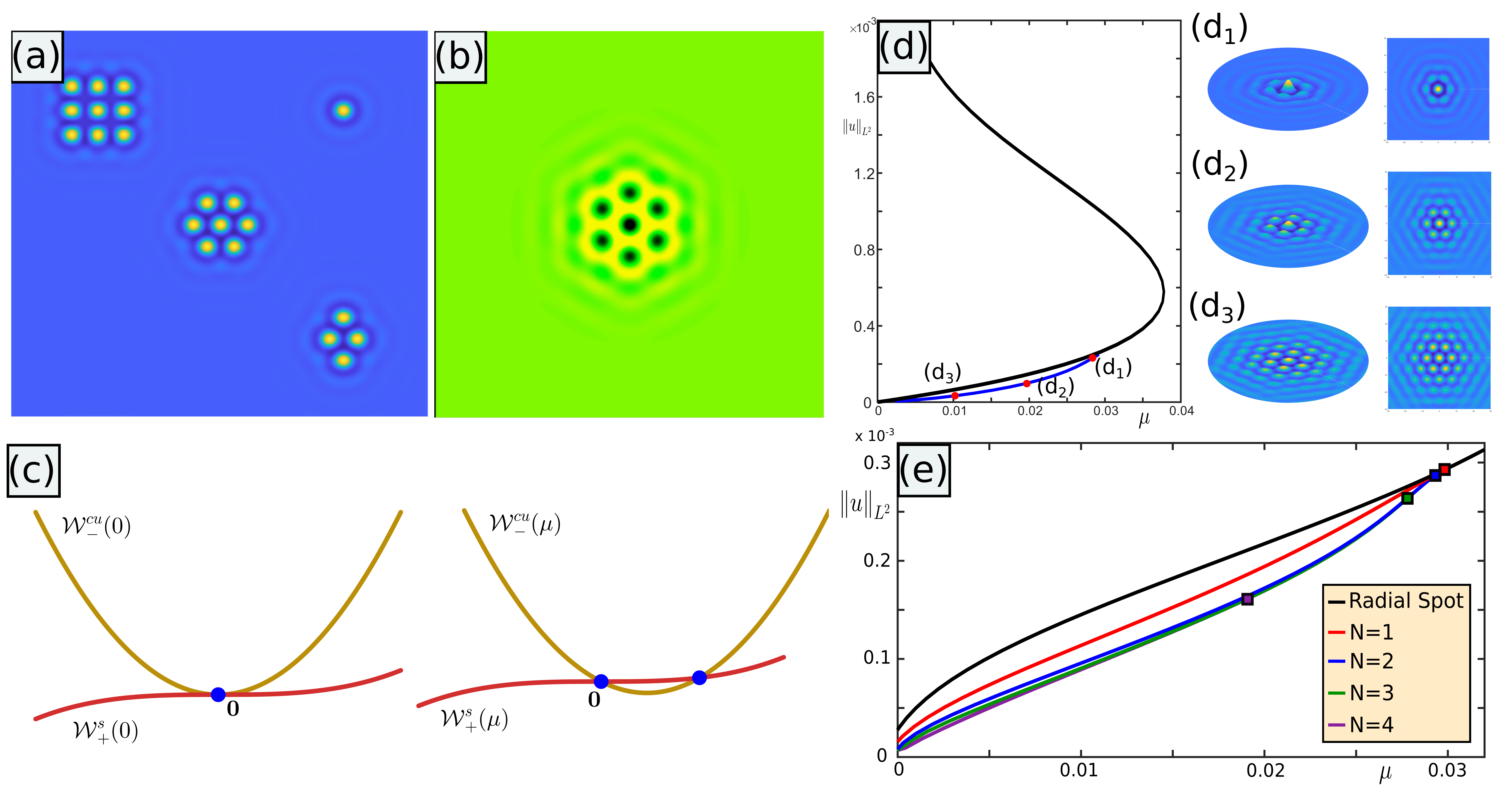}
    \caption{(a) Examples of localised planar patterns possessing different dihedral symmetries: including squares ($\mathbb{D}_{4}$, top-left), spots (radial, top-right), hexagons ($\mathbb{D}_{6}$, centre), and rhomboids ($\mathbb{D}_{2}$, bottom-right). Light yellow regions indicate peaks, while dark blue regions indicate depressions. (b) A localised $\mathbb{D}_{6}$ patch observed in the von Hardenberg model for dryland vegetation \cite{vonHardenberg2001}, where dark regions indicate peaks in vegetation density and light regions indicate depressions. Here, solutions are found in the bistability region between uniform solutions and domain-covering hexagons, for parameters seen in \cite[Figure 8]{Gowda2014}. (c) An illustration of the geometry between the core and far-field manifolds $\mathcal{W}^{cu}_{-}(\mu)$ and $\mathcal{W}^{s}_{+}(\mu)$, respectively, for (left) $\mu=0$ and (right) $0<\mu\ll1$. (d) Localised $\mathbb{D}_6$ patches (blue) are found to bifurcate off the lower branch of a localised radial spot (black) for \eqref{e:SH} when $\gamma=1$. Panels ($d_1$)-($d_3$) show surface plots of solutions at various points (red circles) on the curve. (e) For different choices of truncation order $N$, solutions bifurcate from the flat state along unique curves in parameter space which coincide as $\mu$ increases.}
    \label{fig:IntroFig}
\end{figure}

To understand why this problem remains elusive, we first explore the theory of localised patterns in one spatial dimension. In one dimension, the existence of localised structures can be explained using spatial dynamics \cite{chapman2009exponential} or symmetry arguments \cite{Coullet2000Localized}, and even more complicated behaviour can be understood from bifurcation theory and energy arguments \cite{beck2009snakes,Burke2006Localized,burke2007snakes,Woods1999}. Naturally, one might question whether the aforementioned techniques could be extended to higher dimensions. Such an approach has yielded success in the study of localised planar fronts, where we now have a two-dimensional pattern (such as stripes or hexagons) but the localisation remains restricted to a single direction \cite{Doelman2003,Kozyreff2013Hexagonal}. For example, Doelman et al. \cite{Doelman2003} proved the existence of modulated hexagon fronts connected to an unstable flat, or patterned, state. Proving this result required a centre-manifold reduction as well as finding connecting orbits to the resultant amplitude equations, utilising techniques from spatial dynamics, bifurcation theory, and geometric singular perturbation. However, this approach still requires that the localisation is in a single direction, such that the interface between states is a straight line. In fact, in \cite{Doelman2003} the authors note the difficulty in rigorously explaining the existence of fully-localised patches of hexagons, acknowledging that `the approach used in this paper certainly fails'. Alternatively, one could formally impose an ansatz on a fixed hexagon spatial lattice and derive 2D amplitude equations. However, trying to find connecting orbits in these equations is even more difficult than the planar hexagon equations. Furthermore, this formal reduction still results in trying to find fully-localised solutions to a two-dimensional equation, and so does not serve to simplify the problem.

Despite this pessimistic outlook, recent attempts at understanding fully-localised planar patterns have yielded significant progress by focusing on patterns that are radially-symmetric~\cite{scheel2003radially,lloyd2009localized,mccalla2013spots,mcquighan2014oscillons,Hill2020Localised}. In particular, the works \cite{lloyd2009localized,mccalla2013spots} rigorously establish the existence of radially-symmetric solutions in the SHE~\eqref{e:SH}. We note that, since we are exclusively interested in time-independent solutions, the SHE~\eqref{e:SH} takes the form of an RD equation as in \eqref{e:RDsys} by simply setting $\mathbf{u} = (u,(1 + \Delta)u)$. In the radially-symmetric case, i.e. $u = u(r)$, the SHE \eqref{e:SH} reduces to a fourth order ordinary differential equation (ODE) in $r$, meaning that one may interpret localised radially-symmetric patterns as ODE solutions which decay to zero as $r \to \infty$. Using radial centre manifold theory, three types of stationary radially localised patterns have been shown to exist in \eqref{e:SH} \cite{lloyd2009localized,mccalla2013spots}: spot A which has a maximum at the core, spot B which has a minimum at the core, and rings which have their maximum/minimum away from the core. These patterns are shown to bifurcate from $\mu = 0$ and for $0 < \mu \ll 1$ each pattern is constructed via asymptotic matching of solutions in a core manifold, containing all small-amplitude solutions that remain bounded as $r\to0$, and a far-field manifold, which contains all small-amplitude solutions with exponential decay to zero as $r\to\infty$. The construction of the spot A solution is the simplest to understand as a quadratic order expansion of the core manifold is matched with solutions to the linear flow in the far-field. We refer the reader to Figure~\ref{fig:IntroFig}(c) for a visualisation of the nondegenerate quadratic tangency between the core and far-field manifolds at $\mu=0$. 

In this manuscript we attempt to move beyond the radially-symmetric case of $\mathbf{u} = \mathbf{u}(r)$ by searching for steady-state solutions to two-component RD systems, including the SHE, that have a nontrivial dependence on the phase variable $\theta$. To this end, we focus on {\em localised dihedral patterns} which depend on $\theta$ through a rotational $m$-fold symmetry. Precisely, such a pattern is comprised of peaks arranged in a compact region of the plane such that it is symmetric with respect to reflection in the $x$-axis and rotations of angle $2\pi/m$ about its centre, $r = 0$. We term these solutions $\mathbb{D}_{m}$ patches, in reference to the dihedral symmetry group generated by the above rotations and reflections, and note that a localised square pattern would be a $\mathbb{D}_4$ patch whereas a localised hexagon pattern would be $\mathbb{D}_6$ patch. In \cite{lloyd2008localized}, Lloyd et al. numerically observed localised $\mathbb{D}_{6}$ solutions to \eqref{e:SH} bifurcating from the flat state, which then connects to the curve of a spot A solution, as shown in Figure~\ref{fig:IntroFig}(d). Such numerical schemes involve solving a Galerkin system coming from an $N$-term truncated Fourier expansion of the solution in the phase variable $\theta$. In Figure~\ref{fig:IntroFig}(e) we observe that different choices of $N$ possess distinct solution curves of localised $\mathbb{D}_{6}$ patches bifurcating from the flat state at $\mu=0$, which is a key motivation of this work. We emphasise that our work herein goes beyond just localised square and hexagonal patterns and accounts for both even and odd $m$.

Here we leverage the radially-symmetric analysis in ~\cite{scheel2003radially,lloyd2009localized,mccalla2013spots} by following a similar approach to find $\mathbb{D}_{m}$ patches bifurcating from the flat state at a Turing bifurcation point. We study approximations of small amplitude localised dihedral patterns by expanding a $\mathbb{D}_{m}$ patch solution as the truncated Fourier series
\begin{equation}\label{e:FS}
    \mathbf{u}(r,\theta) = \mathbf{u}_{0}(r) + 2\sum_{n=1}^{N} \mathbf{u}_{n}(r)\cos\left(m n \theta\right),
\end{equation}
where $N\in\mathbb{N}$ is the truncation order. It should of course be pointed out that by definition a $\mathbb{D}_{m}$-lattice pattern is invariant under rotations of $2\pi/m$ about its centre, meaning $\mathbf{u}(r,\theta + 2\pi/m) = \mathbf{u}(r,\theta)$ for all $(r,\theta)$, as well as under reflection in the $x$-axis, meaning $\mathbf{u}(r,-\theta) = \mathbf{u}(r,\theta)$ for all $(r,\theta)$, and so any such solution of \eqref{e:RDsys} is captured by the Fourier cosine-series \eqref{e:FS} upon letting $N \to \infty$. However, the truncation to order $N$ in the above Fourier series is necessary for our analysis, thus leading to the stipulation that we study {\em approximate} localised dihedral patterns here. Putting \eqref{e:FS} into the RD system \eqref{e:RDsys} results in a nonlinearly coupled system of non-autonomous ODEs in the radial variable $r$ in terms of the Fourier coefficients $\mathbf{u}_n(r)$. Our goal in this work is to use the radial centre-manifold theory developed in \cite{lloyd2009localized,mccalla2013spots,mcquighan2014oscillons} to demonstrate the existence of exponentially decaying solutions $\mathbf{u}_n(r)$ to our coupled ODEs, which through \eqref{e:FS} lead to approximate localised planar patterns in the planar RD equations of the form \eqref{e:RDsys}. In particular, our analysis will be restricted to parameter values $\mu$ in the neighbourhood of a Turing instability and we will show that our planar patterns bifurcate from the homogeneous state undergoing such an instability. We will exclusively look for localised $\mathbb{D}_{m}$ patch solutions that are analogous to the radially-symmetric spot A solutions from \cite{lloyd2009localized} and do not attempt to prove the existence of all types of small amplitude localised radial solutions. The major difference between our work and the work on radially-symmetric patterns is that our asymptotic matching between the core and far-field manifolds requires solving $(N + 1)  \geq 2$ nonlinearly coupled algebraic equations, while the latter has only a single nonlinear equation ($N = 0$). Our main results in the following section make this connection precise, and we show that these matching equations can be solved explicitly for $N \leq 4$, while in the case that $m$ is a multiple of $6$ we employ a computer-assisted proof to help demonstrate the existence of solutions for all finite $N \gg 1$.

Our approach possesses a number of advantages in studying the emergence of localised planar patterns. Firstly, our choice of Fourier decomposition \eqref{e:FS} allows for patterns with a wide choice of possible symmetries. 
A common approach is to carry out a weakly nonlinear analysis where one derives slowly varying amplitude equations over a pre-determined periodic lattice in Cartesian coordinates by expanding $\mathbf{u}$ as
\begin{equation}
\mathbf{u}(\mathbf{x}) = \sum_i \varepsilon A_i(X,Y)e^{i\mathbf{k_i}\cdot \mathbf{x}} + c.c. + \mathcal{O}(\varepsilon^2),\qquad   (X,Y) = \varepsilon (x,y),\qquad \varepsilon\ll1, 
\end{equation}
where $c.c.$ denotes complex conjugate, and $\mathbf{k_i}$ are dual lattice vectors (e.g. a hexagon lattice would have 3 vectors $\mathbf{k}_1=(-1,0),\mathbf{k}_2=\frac12(1,\sqrt{3}),\mathbf{k}_3=\frac12(1,-\sqrt{3})$). One can then formally derive a set of 3 complex Ginzburg-Landau type equations for the amplitudes $A_i$; see for instance~\cite[Chapter 9]{hoyle2006pattern} for the hexagonal case. However, this approach has a number of significant hurdles to overcome, including rigorous justification for the reduction and then the proof of a localised solution to the amplitude equations.  Rather than considering the localised version of a predetermined domain covering solution, our approach provides a local bifurcation theory for a multitude of dihedral patterns that extend beyond the well-studied stripes, squares and hexagons.
Moreover, through this approach we are able to reduce a planar PDE problem to an algebraic matching condition, which can be solved numerically with relative ease. Even at small truncation orders, which we demonstrate can be solved without numerical assistance, our approximate solutions provide excellent initial conditions for numerical continuation and exhibit a strong likeness to examples of localised patterns observed in experiments. Finally, this approach has value not just in its results, but also in its limitations. We derive an upper bound for the bifurcation parameter in terms of the truncation order $N$, providing useful intuition for the difficulties encountered as $N\to\infty$. We observe infinitely many localised patterns bifurcating from the trivial state, which highlights an issue in understanding fully-localised planar patterns and helps to motivate further study in this area.

The key limitation of this approach is that we cannot hope to explain the emergence of $\mathbb{D}_{m}$ localised patterns in full planar RD systems. However, our solutions closely resemble those that have been documented in the literature, thus leading to the belief that they closely resemble true solutions of RD systems. Furthermore, we are able to say when localised patches of cellular patterns exist and bifurcate from the trivial state in a numerical scheme based on the finite mode Fourier decomposition \eqref{e:FS}. Through our analysis we are able to uncover new approximate localised solutions which can be used as initial conditions for numerical path-following routines, allowing one to continue localised solutions further into $\mu > 0$ where the finite-mode decomposition becomes a good approximation for the fully localised numerical solutions, and we observe exponential decay in the maximum of the amplitudes of $\mathbf{u}_n$. Numerically, we find these patches then undergo a process similar to homoclinic snaking, as discussed in~\cite{lloyd2008localized,Avitabile2010Snake}. See also the special issue on Homoclinic snaking~\cite{Champneys2021Editorial} for a proper introduction to the topic and review of the literature.

In this work, we demonstrate that a local change of variables brings distinct RD equations into a normal form in the neighbourhood of a Turing bifurcation. Notably, the SHE \eqref{e:SH} is a specific instance of the general RD equations \eqref{e:RDsys} which can be written in this canonical form without any change of variables, where the only difference is that now the nonlinearity in \eqref{e:SH} should be considered as a truncated Taylor expansion about $u = 0$, representing the deviation from the homogeneous state that undergoes a Turing instability. This instance in which the SHE arises allows it to be considered as a truncated normal form for systems undergoing a Turing bifurcation, and is exactly why so many pattern formation investigations have employed the SHE in the first place. Turing instabilities have been documented in spatially-extended models from chemistry \cite{castets90,ouyang91}, biology \cite{gierer72,kondo10,turing1952chemical}, ecology \cite{klausmeier99,mimura78}, and fluid dynamics \cite{cross,newell69}, to name a few, and so our results may be applicable to more complicated systems if a suitable normal form reduction can be found. There has been some progress in this direction for radially-symmetric localised solutions, such as in neural-field equations \cite{faye2013localized} and on the surface of a  ferrofluid \cite{Hill2020Localised}, which suggests the approach presented here may also be extended in these cases. Hence, our work in this manuscript provides evidence for the emergence of localised dihedral patterns resulting from Turing bifurcations on planar domains, while also giving explicit forms for finding such patterns numerically.

This paper is organised as follows. In Section~\ref{s:Results} we present our main results. We begin with the necessary hypotheses to assume a non-degenerate Turing bifurcation is taking place in system \eqref{e:RDsys} at some parameter value and then proceed to state our results for the Galerkin truncated system arising from assuming the form \eqref{e:FS}. In Section~\ref{s:Numerics} we present our numerical findings, beginning with numerical verification of our analysis in \S\ref{s:Verification} and then in \S\ref{s:Continuation} we provide numerical continuations of the localised dihedral patterns far into $\mu > 0$ where they develop larger amplitudes and greater localisation. In Section~\ref{s:Matching} we define and quantify the core and far-field manifolds of our coupled radial ODE resulting from introducing \eqref{e:FS} into \eqref{e:RDsys}. Furthermore, we reduce the problem of asymptotically matching these manifolds when $0 < \mu \ll 1$ to solving a system of nonlinear matching equations in $N + 1$ variables. Section~\ref{s:patch;N} is entirely dedicated to solving these matching equations. We begin by explicitly solving them for small values of $N$ and then we demonstrate the existence of solutions to these matching equations for $N \gg 1$. The latter is achieved by first employing a computer-assisted proof (whose details are left to the appendix) to solve a limiting nonlocal integral equation and then using this solution to demonstrate the existence of solutions for large, but finite, $N$. The results of Sections~\ref{s:Matching} and \ref{s:patch;N} together are the proofs of our main results for the RD systems near a Turing instability, Theorems~\ref{thm:SmallPatch} and \ref{thm:BigPatch} below. Finally, we conclude in Section~\ref{s:Conclusion} with a discussion of our findings and some future areas of work.  

%
%
%
%
%
%
%
%
%
%
\section{Main Results}\label{s:Results}

We begin with steady planar two-component RD equations, which we express as
\begin{equation}\label{R-DEqn}
	\mathbf{0} = \mathbf{D}\Delta\mathbf{u} - \mathbf{f}(\mathbf{u},\mu), \qquad \qquad \Delta:=\bigg(\partial_{rr} + \frac{1}{r} \partial_{r} + \frac{1}{r^{2}}\partial_{\theta\theta}\bigg), 
\end{equation}
where $\mathbf{u}=\mathbf{u}(r,\theta)\in\mathbb{R}^{2}$ for the standard planar polar coordinates $r\in\mathbb{R}^{+}$ and $\theta\in(-\pi,\pi]$. Throughout we will assume that $\mathbf{D}\in\mathbb{R}^{2\times2}$ is an invertible matrix and $\mathbf{f}\in C^{k}\big(\mathbb{R}^{2}\times\mathbb{R},\mathbb{R}^{2}\big)$, $k \geq 3$, describes the reaction kinetics of the system. The parameter $\mu$ plays the role of the bifurcation parameter. We assume that $\mathbf{f}(\mathbf{0}, \mu) = 0$ for all $\mu\in\mathbb{R}$ such that $\mathbf{u}(r,\theta)\equiv\mathbf{0}$ is a homogeneous equilibrium for all $\mu \in \mathbb{R}$. Since $\mathbf{D}$ is invertible, we can apply $\mathbf{D}^{-1}$ to \eqref{R-DEqn} to normalise the diffusive term $\Delta\mathbf{u}$. 

With the following hypothesis we make the assumption that \eqref{R-DEqn} undergoes a Turing instability from the homogeneous equilibrium $\mathbf{u} = \mathbf{0}$ with non-zero wave number. In what follows $\mathbbm{1}_{n}$ will denote the $n\times n$ identity matrix and for simplicity we will assume that this bifurcation takes place at $\mu=0$. 

\begin{hyp}[\em Turing Instability]\label{R-D:hyp;1} 
	We assume that $\mathbf{f}(\mathbf{u},\mu)$ satisfies the following condition:
	\begin{equation}
    		\det\,\big(\mathbf{D}^{-1}D_{\mathbf{u}}\big[\mathbf{f}(\mathbf{0},0)\big] - \lambda\mathbbm{1}_2\big) = 0 \qquad \iff \qquad \lambda = -k_{c}^{2},\nonumber
	\end{equation}
	for some fixed $k_{c}\in\mathbb{R}^{+}$. Furthermore, the eigenvalue $\lambda=-k_{c}^{2}$ of $\mathbf{D}^{-1}D_{\mathbf{u}}\big[\mathbf{f}(\mathbf{0},0)\big]$ is algebraically double and geometrically simple with generalised eigenvectors $\hat U_0, \hat U_1 \in \mathbb{R}^2$, defined such that
	\begin{equation}
	    \bigg(\mathbf{D}^{-1}D_{\mathbf{u}}\big[\mathbf{f}(\mathbf{0},0)\big] + k_{c}^{2} \mathbbm{1}_{2}\bigg)\hat U_{0} = \mathbf{0}, \qquad \bigg(\mathbf{D}^{-1}D_{\mathbf{u}}\big[\mathbf{f}(\mathbf{0},0)\big] + k_{c}^{2} \mathbbm{1}_{2}\bigg)\hat U_{1} = k_{c}^{2} \hat U_0, \qquad \langle \hat U_i^*, \hat U_j\rangle_2 = \delta_{i,j},\nonumber
	\end{equation}
	for each $i,j\in\{0,1\}$, where $\hat U_0^*$, $\hat U_1^*$ are the respective adjoint vectors for $\hat U_0$, $\hat U_1$.
\end{hyp}

We now proceed by expanding $\mathbf{D}^{-1}\big[\mathbf{f}(\mathbf{u},\mu)\big]$ as a Taylor expansion about $(\mathbf{u},\mu) = (\mathbf{0},0)$ to write \eqref{R-DEqn} as
\begin{equation}\label{eqn:R-D}
    \mathbf{0} = \Delta\mathbf{u} - \mathbf{M}_{1}\mathbf{u} - \mu \mathbf{M}_{2}\mathbf{u} - \mathbf{Q}(\mathbf{u},\mathbf{u}) - \mathbf{C}(\mathbf{u},\mathbf{u},\mathbf{u}) + h.o.t.
\end{equation}
In the above we have $\mathbf{M}_1 = \mathbf{D}^{-1}D_{\mathbf{u}}\big[\mathbf{f}(\mathbf{0},0)\big]$, $\mathbf{M}_2 \in\mathbb{R}^{2\times 2}$, $\mathbf{Q}:\mathbb{R}^{2}\times\mathbb{R}^{2}\to\mathbb{R}^{2}$, and $\mathbf{C}:\mathbb{R}^{2}\times\mathbb{R}^{2}\times\mathbb{R}^{2}\to\mathbb{R}^{2}$, where $\mathbf{Q}$ and $\mathbf{C}$ are symmetric bilinear and trilinear maps, respectively. We note that any remainder terms do not affect the subsequent analysis since we are in the region where $|\mathbf{u}|$ and $|\mu|$ are small, and so we have neglected them in {\em h.o.t.}, representing the higher order terms. Notice that Hypothesis~\ref{R-D:hyp;1} gives that the linearization of the right-hand-side of \eqref{eqn:R-D} about $(\mathbf{u},\mu) = (\mathbf{0},0)$ has a zero eigenvalue, giving way to the Turing bifurcation. Let us now make the following non-degeneracy assumptions regarding this bifurcation. 

\begin{hyp}[\em Non-degeneracy condition]\label{R-D:hyp;2} 
	We assume that $\mathbf{M}_2$ and $\mathbf{Q}$ satisfy the following conditions:
	\begin{equation}\label{c0gamma}
    		c_{0}:= \frac{1}{4}\big\langle \hat U_1^*, -\mathbf{M}_2 \hat U_0\big\rangle_2>0, \qquad  \gamma:=\big\langle \hat U_1^*, \mathbf{Q}(\hat U_0, \hat U_0)\big\rangle_2 > 0, \nonumber
	\end{equation}
	where $\hat U_0, \hat U_1 \in \mathbb{R}^2$ are generalised eigenvectors of $\mathbf{M}_1$ defined in Hypothesis~\ref{R-D:hyp;1} with respective adjoint vectors $\hat U_0^*$, $\hat U_1^*$.
\end{hyp}

The non-degeneracy conditions in Hypothesis~\ref{R-D:hyp;2} are important for our proofs in Section~\ref{s:Matching} and so we will formally describe how they arise. We suppose $\mathbf{u}=A \hat U_0 + B \hat U_1$ and project onto each eigenvector; then, \eqref{eqn:R-D} can be transformed into the following normal form 
\begin{equation}\label{eq:Norm}
    \mathbf{0} = \Delta\mathbf{A} - k_{c}^{2}\hat{\mathbf{M}}_{1}\mathbf{A} - \mu \hat{\mathbf{M}}_{2}\mathbf{A} - \hat{\mathbf{Q}}(\mathbf{A},\mathbf{A})  - \hat{\mathbf{C}}(\mathbf{A},\mathbf{A},\mathbf{A}) + h.o.t.
\end{equation}
where $\mathbf{A}:=(A,B)^T$, and
\[
    \hat{\mathbf{M}}_{1} := \begin{pmatrix}
    -1 & 1 \\ 0 & -1
    \end{pmatrix}, \qquad \hat{\mathbf{M}}_{2} :=\begin{pmatrix}
     \langle\hat{U}^{*}_{0}, \mathbf{M}_{2}\hat{U}_{0}\rangle_2 & \langle\hat{U}^{*}_{0}, \mathbf{M}_{2}\hat{U}_{1}\rangle_2\\ \langle\hat{U}^{*}_{1}, \mathbf{M}_{2}\hat{U}_{0}\rangle_2& \langle\hat{U}^{*}_{1}, \mathbf{M}_{2}\hat{U}_{1}\rangle_2
    \end{pmatrix}.
\]
Assuming $|B|\ll |A| \ll 1$, \eqref{eq:Norm} can be reduced to the leading-order PDE
\begin{equation}\label{R-D:Formal;SH}
    0 = -(k_{c}^{2} + \Delta)^{2} A - k_{c}^{2}4 c_{0} \mu A + k_c^2 \gamma A^2 - k_c^2 \kappa A^3 +  h.o.t.,  
\end{equation}
which is analogous to the SHE \eqref{e:SH}. Here we have defined $\kappa:=\big\langle \hat U_1^*, -\mathbf{C}(\hat U_0, \hat U_0, \hat U_0)\big\rangle_2$ with $c_0$ and $\gamma$ as in Hypothesis~\ref{R-D:hyp;2}. In \cite{lloyd2009localized} spot A solutions were found to exist in the SHE if and only if $c_0>0$ and $\gamma\neq0$. We note that \eqref{eqn:R-D} is invariant under the transformation $\hat{T}:(\mathbf{u},\mathbf{Q})\mapsto(-\mathbf{u},-\mathbf{Q})$, and so the assumption $\gamma\neq0$ is equivalent to assuming $\gamma>0$, up to an application of $\hat{T}$. Hence,
Hypothesis~\ref{R-D:hyp;2} provides non-degeneracy assumptions for the existence of localised spots in the SHE \eqref{R-D:Formal;SH}, which we relate to \eqref{eqn:R-D} in the neighbourhood of a Turing instability.

\begin{rmk}
    The SHE \eqref{e:SH} fits into our RD framework and satisfies Hypothesises~\ref{R-D:hyp;1} and \ref{R-D:hyp;2} by setting $\mathbf{u}:=(u, (1+\Delta)u)^{T}$, $\hat{U}_{0} = \hat{U}^*_{0} = (1,0)^T$, $\hat{U}_{1} = \hat{U}^*_{1} = (0,1)^T$, $\langle \mathbf{x}, \mathbf{y}\rangle_2 = \mathbf{x}^{T}\mathbf{y}$,
\begin{equation}
    \mathbf{Q}(\mathbf{x},\mathbf{y}) = \begin{pmatrix}
    0 \\ \gamma x_1 y_1
    \end{pmatrix}, \qquad \mathbf{C}(\mathbf{x},\mathbf{y},\mathbf{z}) = \begin{pmatrix}
    0 \\ -x_1 y_1 z_1
    \end{pmatrix}, \qquad \mathbf{M}_{1} = \begin{pmatrix}
    -1 & 1 \\ 0 & -1
    \end{pmatrix}, \qquad \textnormal{and}\quad \mathbf{M}_{2} = \begin{pmatrix}
    0 & 0 \\ -1 & 0
    \end{pmatrix}.   \nonumber
\end{equation}  
We see that $\mathbf{M}_{1}$ is identical to $\hat{\mathbf{M}}_{1}$ in the RD normal form \eqref{eq:Norm} for patterns close to a Turing instability, and so we will focus exclusively on the SHE \eqref{e:SH} for our numerical demonstrations.
\end{rmk}

Under the above hypotheses, our goal in this work is to obtain approximate dihedral solutions to \eqref{R-DEqn} for $\mu \approx 0$ that are bounded as $r \to 0$, have $\lim_{r \to \infty} u(r,\theta) = 0$, and satisfy $u(r,\theta + 2\pi/m) = u(r,\theta)$ for all $(r,\theta)$ and some $m \in \mathbb{N}$. Previous works on the SHE have proven the existence of such solutions which are independent of the azimuthal component \cite{lloyd2009localized,mccalla2013spots} and having seen \eqref{R-D:Formal;SH} one should be convinced that these results can be extended to the more general system \eqref{R-DEqn} with little issue. We aim to initiate the study of solutions of \eqref{R-DEqn} that exhibit a nontrivial dependence on the azimuthal component, resulting in the property $u(r,\theta + 2\pi/m) = u(r,\theta)$. Such patterns are invariant with respect to the symmetry group $\mathbb{D}_{m}$, generated by rotations of $2\pi/m$ in the plane about the origin and reflections over the horizontal axis. To this end, we introduce a $\mathbb{D}_{m}$, $N$-truncated Fourier approximation of solutions to \eqref{eqn:R-D} by
\begin{equation}\label{FourierExp:R-D}
	\mathbf{u}(r,\theta) = \sum_{n = -N}^{N} \mathbf{u}_{|n|}(r) \cos\left(m n\theta\right).
\end{equation}
Upon projecting onto each Fourier mode $\cos(mn\theta)$, \eqref{eqn:R-D} is reduced to the (finite-dimensional) Galerkin system
\begin{equation} \label{eqn:R-D;Galerk}
    0 = \Delta_{n}\,\mathbf{u}_{n} - \mathbf{M}_{1}\mathbf{u}_{n} - \mu \mathbf{M}_{2}\mathbf{u}_{n} - \sum_{i+j=n} \mathbf{Q}(\mathbf{u}_{|i|},\mathbf{u}_{|j|}) - \sum_{i+j+k=n} \mathbf{C}(\mathbf{u}_{|i|},\mathbf{u}_{|j|},\mathbf{u}_{|k|}) + h.o.t.,
\end{equation}
for all $n\in[0,N]$ and 
\begin{equation}\label{Delta_n}
    \Delta_{n} := \left(\partial_{rr} + \frac{1}{r}\partial_{r} - \frac{(mn)^{2}}{r^{2}}\right).
\end{equation} 
Since the case when $N=0$ has already been thoroughly studied, we therefore restrict our analysis to the case when $N\geq 1$ for the remainder of this manuscript and consider only nontrivial solutions of \eqref{eqn:R-D;Galerk}. We present our first main result of the manuscript detailing the existence of solutions to \eqref{eqn:R-D;Galerk} for $N = 1,2,3,4$, which represent approximate small patch solutions of the steady planar RD system \eqref{R-DEqn}.

\begin{thm}\label{thm:SmallPatch} 
	Assume Hypotheses~\ref{R-D:hyp;1} and \ref{R-D:hyp;2}. Fix $m,N \in \mathbb{N}$ and assume the constants $\{a_{n}\}_{n=0}^{N}$ are nondegenerate solutions of the nonlinear matching condition 
		\begin{equation}\label{MatchEq}
    a_{n} = 2\sum_{j=1}^{N-n} \cos\left(\frac{m\pi(n-j)}{3}\right) a_{j} a_{n+j} + \sum_{j=0}^{n} \cos\left(\frac{m\pi(n-2j)}{3}\right) a_{j}a_{n-j},
\end{equation}
    for each $n = 0,1,\dots, N$. Then, there exist constants $\mu_0,r_0,r_1 > 0$ such that the Galerkin system \eqref{eqn:R-D;Galerk} has a radially localised solution of the form
\begin{equation}\label{RadialProfile}
    \mathbf{u}_{n}(r) = \frac{\sqrt{24c_{0}}\,k_{c} a_{n}\mu^{\frac{1}{2}}}{\gamma \sqrt{k_{c}\pi}}(-1)^{m n}
    \begin{cases}
         \sqrt{\frac{k_{c}\pi}{2}}\,J_{mn}(k_{c} r)\hat{U}_{0} + \mathcal{O}(\mu^{\frac{1}{2}}) & 0\leq r\leq r_{0},  \\
         r^{-\frac{1}{2}}\,\cos( \psi_n(r))\hat{U}_{0} + \mathcal{O}(\mu^{\frac{1}{2}}), & r_{0} \leq r \leq r_1\mu^{-\frac{1}{2}},\\
         r^{-\frac{1}{2}}\,\textnormal{e}^{\sqrt{c_{0}}(r_1- \mu^{\frac{1}{2}}r)}\,\cos( \psi_n(r))\hat{U}_{0} + \mathcal{O}(\mu^{\frac{1}{2}}), & r_1\mu^{-\frac{1}{2}} \leq r,
    \end{cases}
\end{equation} 
for each $\mu\in(0,\mu_{0})$, $n\in[0,N]$, where $\psi_n(r):= k_c r - \frac{mn \pi}{2} - \frac{\pi}{4}$ and $J_{mn}$ is the $(mn)^\mathrm{th}$ order Bessel function of the first kind. In particular, such localised solutions exist for $N=1,2,3$ with the $a_n$ given in Proposition~\ref{prop:N123}. Furthermore, in the case that $N = 4$ and $6 \mid m$, these results again hold with the $a_n$ given in Proposition~\ref{prop:N4}. 
\end{thm}

\begin{rmk}\label{rem:cubic}
    We note that, much like in the axisymmetric case \cite{lloyd2009localized}, the bifurcation of these localised patterns is entirely determined by the quadratic and linear coefficients $\gamma,c_0$. Hence, the cubic nonlinearity of \eqref{e:RDsys}, which is essential in determining the bifurcation of localised one-dimensional patterns, has no effect on the emergence of these solutions.
\end{rmk}

The proof of Theorem~\ref{thm:SmallPatch} is broken down over Sections~\ref{s:Matching} and \ref{s:patch;N}. In particular, Section~\ref{s:Matching} decomposes the Galerkin system on $r > 0$ into three disjoint intervals, making up the three components of the solution presented in \eqref{RadialProfile}. In each of these regions the dynamics can be captured by geometric blow-up methods. Section~\ref{s:Matching} presents the analysis for general $N \geq 1$ and concludes by showing that solutions in the three distinct regions of the Galerkin system can be patched together by finding nondegenerate solutions of \eqref{MatchEq}. The resulting values of $a_n$ that satisfy these matching equations are exactly the $a_n$ presented in Theorem~\ref{thm:SmallPatch}. Section~\ref{s:patch;N} is entirely dedicated to further understanding and solving these matching equations. We present a number of symmetry results that help to eliminate redundant solutions of \eqref{MatchEq} and then proceed to solve these equations for small $N$ by hand. The result of this work is the proof of Theorem~\ref{thm:SmallPatch}.

As one can see, the matching condition \eqref{MatchEq} depends on the value of $\cos(\frac{m\pi}{3})$ and so there are four distinct cases depending on whether $m$ is divisible by 2 and 3. In the case of $N=1$, Proposition~\ref{prop:N123} below shows that there is a unique localised $\mathbb{D}_{m}$ solution to \eqref{eqn:R-D;Galerk} so long as $3 \mid m$ or $2 \mid m$, up to a half-period rotation, with a peak at its centre if $6\mid m$ and a depression otherwise. We emphasise that, although the $a_n$ are equivalent for any two choices of $m$ leading to the same value of $\cos(\frac{m\pi}{3})$, the full solution $u_n(r)$ defined in \eqref{RadialProfile} has a distinct $m$-dependent radial profile for each $n=0,\dots,N$. This is evident in Figure~\ref{fig:N-1Solns} where we present the $N=1$ solutions for rhombic ($\mathbb{D}_{2}$), triangular ($\mathbb{D}_3$), square ($\mathbb{D}_{4}$) and hexagonal ($\mathbb{D}_{6}$) patterns; although the rhombic and square patterns possess the same $a_n$ since $\cos(\frac{2\pi}{3}) = \cos(\frac{4\pi}{3}) = -\frac 1 2$, they produce fundamentally different solutions as a result of the Bessel functions $J_{mn}(r)$ in \eqref{RadialProfile}.

\begin{figure}[t!] 
    \centering
    \includegraphics[width=\linewidth]{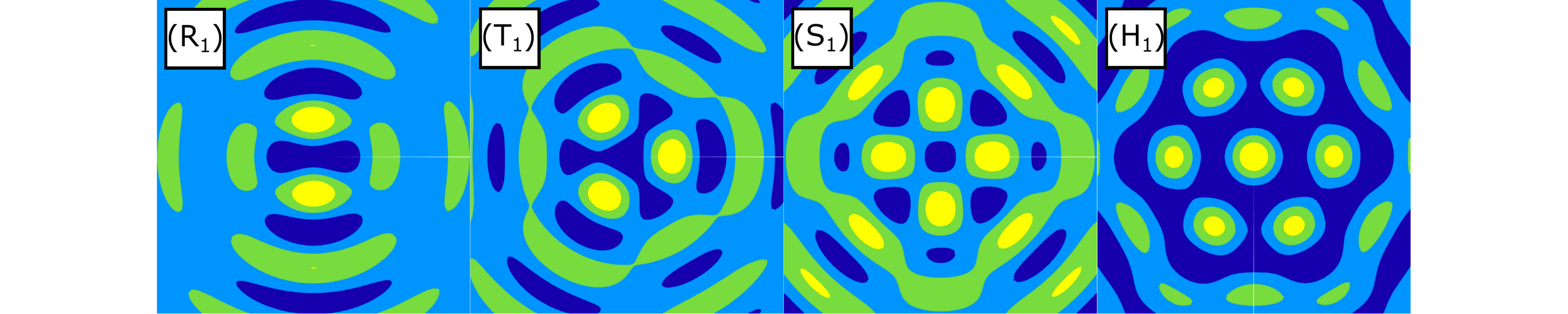}
    \caption{Solutions when $N=1$, given by Theorem \ref{thm:SmallPatch}, for $(R)$hombic $\sim \mathbb{D}_{2}$, $(T)$riangular $\sim \mathbb{D}_{3}$, $(S)$quare $\sim \mathbb{D}_{4}$, and $(H)$exagonal $\sim \mathbb{D}_{6}$ patte  rns are plotted in a circular disc of radius $20$ circumscribing the square region. Each panel presents a contour plot of the asymptotic solution $u(r,\theta)$ of \eqref{eqn:R-D;Galerk} as $\mu\to0$. Light regions indicate peaks, while dark regions indicate depressions.}
    \label{fig:N-1Solns}
\end{figure}

For $N>1$, we note that matching condition \eqref{MatchEq} possesses an extra symmetry when $6\mid m$, which is analogous to associating `bright' solutions $u_{*}(r,\theta)$ with their `dark' counterparts $v_{*}(r,\theta):=U_{0}(r)-u_{*}(r,\theta)$, where $U_{0}(r)$ is the localised axisymmetric solution found in \cite[Theorem 2]{lloyd2009localized}. This means that the full set of solutions to \eqref{MatchEq} for $6\mid m$ are described by a smaller subset of solutions than in the case when $6\nmid m$ because more solutions can be recovered from their respective symmetries. We emphasise that our work in Section~\ref{s:patch;N} shows that the matching condition \eqref{MatchEq} does not have a unique solution for $2 \leq N \leq 4$, even after we quotient by those related by symmetry. 

Solutions given by Theorem~\ref{thm:SmallPatch} for $N=2$ are presented in Figure~\ref{fig:N-2Solns}. As in the case $N = 1$, there are again no $\mathbb{D}_m$ solutions for $m = 1,5,7,11,13,...$, i.e. satisfying $2 \nmid m$ and $3 \nmid m$. In contrast, there are exactly two solutions for each other value of $m$, with $m = 2,3,4,6$ presented in the figure. Both $\mathbb{D}_6$ solutions have an elevation at the centre, where $(H_2)$ is the `dark' solution associated to the standard cellular lattice solution $(H_1)$. Notice that for $6 \nmid m$ the pair of $\mathbb{D}_m$ solutions are such that one has a depression at its centre while the other has an elevation. Moreover, those solutions for $6 \nmid m$ with a depression at the centre can be seen as a natural extension of the $N=1$ solution, whereas the latter is fundamentally different from its lower-dimension counterpart. This suggests that the second solution is not related to the `standard' localised $\mathbb{D}_{m}$ pattern. Such a solution may be a transitory state that connects to another solution, such as the radial spot, or an artefact of our Galerkin approximation which only appears for a given finite truncation $N$. We leave such questions for a follow-up investigation.

\begin{figure}[t!] 
    \centering
    \includegraphics[width=\linewidth]{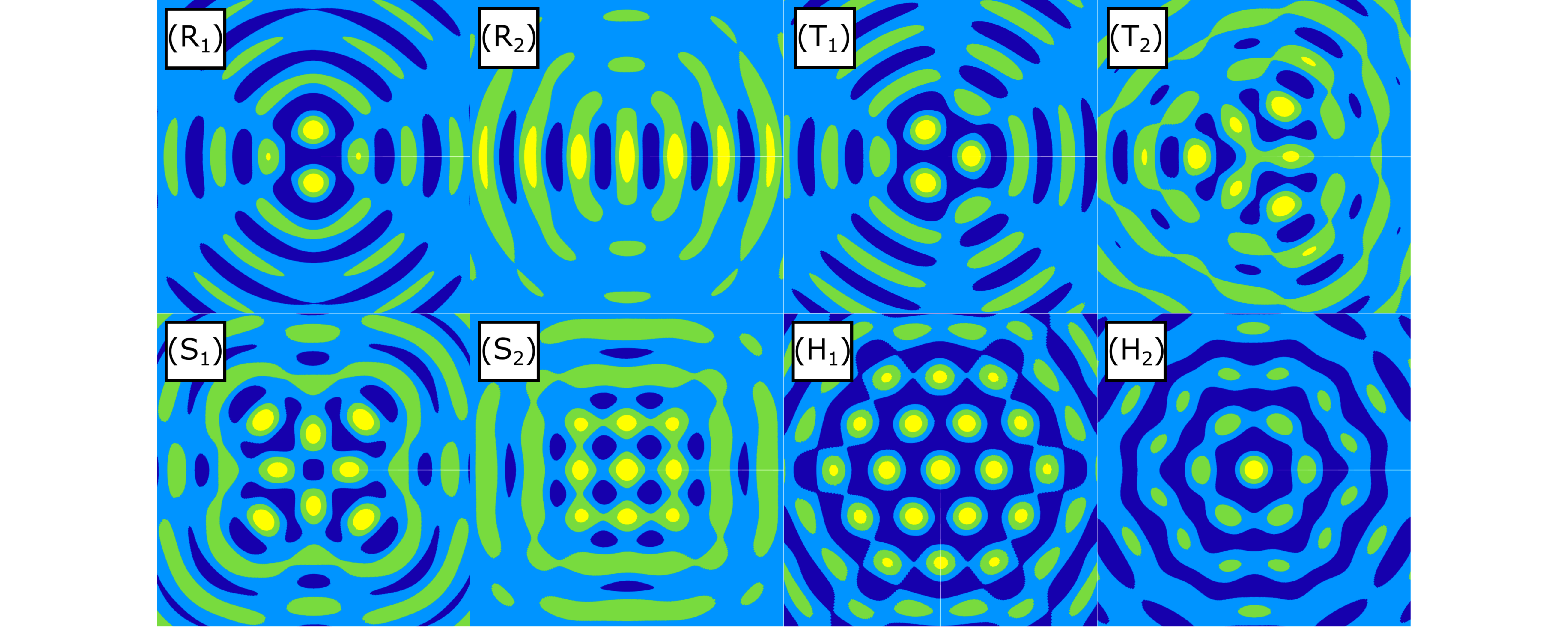}
    \caption{Solutions when $N=2$, given by Theorem \ref{thm:SmallPatch}, are plotted in a circular disc of radius $30$ circumscribing the square region. Up to a half-period rotation, there are two solutions for $(R)$hombic $\sim \mathbb{D}_{2}$, $(T)$riangular $\sim \mathbb{D}_{3}$, $(S)$quare $\sim \mathbb{D}_{4}$, and $(H)$exagonal $\sim \mathbb{D}_{6}$ patterns.}
    \label{fig:N-2Solns}
\end{figure}

For $N=3$, Theorem~\ref{thm:SmallPatch} gives five distinct solutions to \eqref{eqn:R-D;Galerk} when $m$ is even or equal to 3. These distinct solutions can be seen in Figure~\ref{fig:N-3Solns} for $m = 2,3,4,6$, as well as $m = 5$. Notice that we now find the emergence of a single $\mathbb{D}_5$ solution.  We emphasise that, much like when $2\mid m$ or $3\mid m$, we obtain more solutions for $2\nmid m$, $3\nmid m$ as $N$ increases. However, there does not appear to be any connection between solutions at different truncation orders, and so there does not seem to be a `standard' localised $\mathbb{D}_{m}$ pattern when $2\nmid m$, $3\nmid m$. We further note that super-lattice structures, caused by a superposition of lattice symmetries, begin to emerge as solutions of \eqref{MatchEq}. This is most apparent in the $\mathbb{D}_{6}$ patterns of Figure~\ref{fig:N-3Solns} where, as well as the standard cellular lattice solution $(H_{1})$, we obtain solutions comprised of a hexagon patch surrounded by other hexagon patches $(H_{2})$ or faint triangular patches $(H_{3})$, respectively. 

\begin{figure}[t!] 
    \centering
    \includegraphics[width=\linewidth]{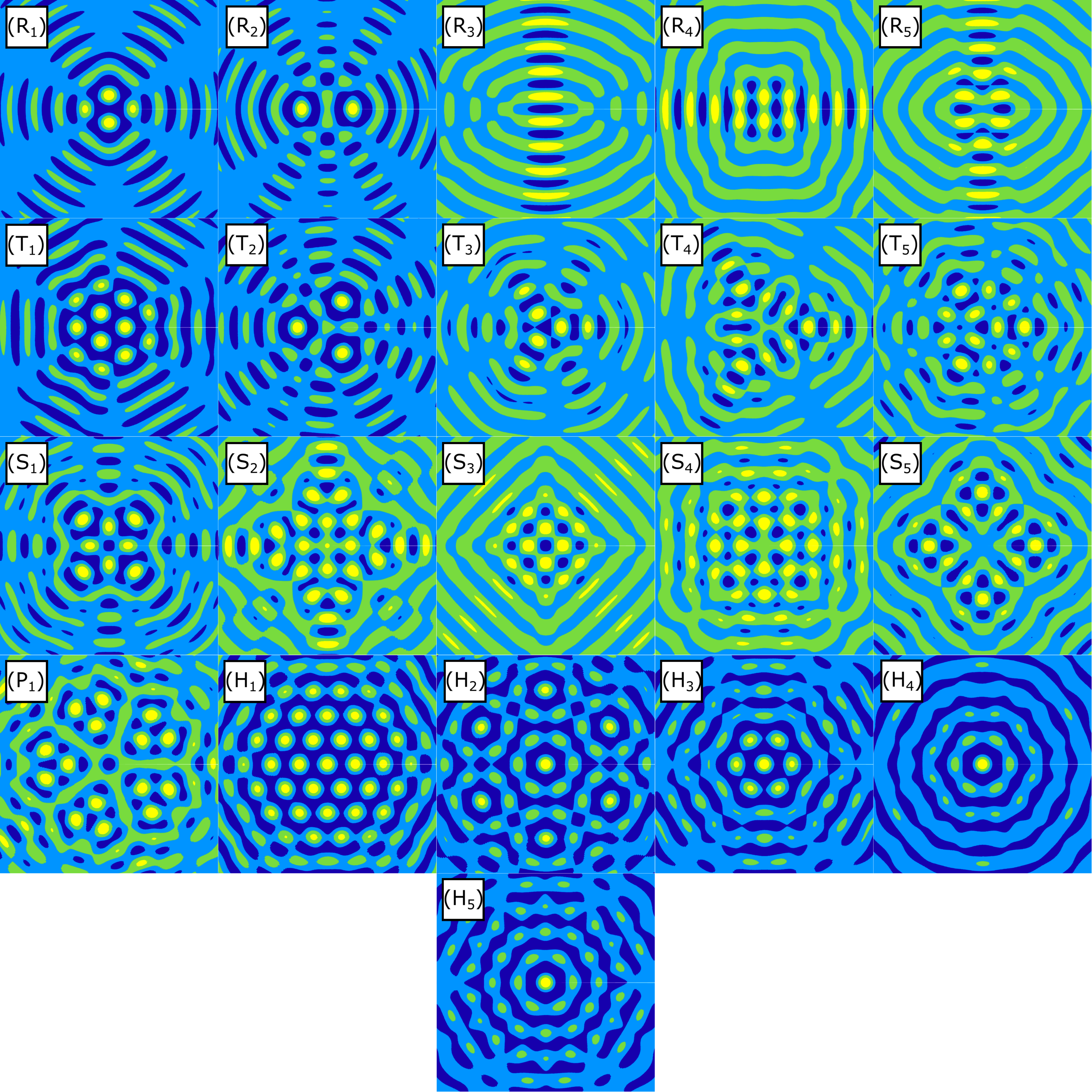}
    \caption{Solutions when $N=3$, given by Theorem \ref{thm:SmallPatch}, for $(R)$hombic $\sim \mathbb{D}_{2}$, $(T)$riangular $\sim \mathbb{D}_{3}$, $(S)$quare $\sim \mathbb{D}_{4}$, $(P)$entagonal $\sim \mathbb{D}_{5}$, and $(H)$exagonal $\sim \mathbb{D}_{6}$ patterns are plotted in a circular disc of radius $40$ circumscribing the square region.}
    \label{fig:N-3Solns}
\end{figure}

As $N$ increases, it becomes more difficult to solve \eqref{MatchEq} explicitly, and so our analysis of $N=4$ is focused on the simpler case when $6\mid m$. Theorem~\ref{thm:SmallPatch} (via Proposition~\ref{prop:N4} below) gives five distinct solutions to \eqref{eqn:R-D;Galerk} which are presented in Figure~\ref{fig:N-4Solns} for the SHE. The super-lattice structures seen in $N=3$ are more apparent, including a very striking pattern of localised triangular patches surrounding a hexagon in $(H_3)$. 

\begin{figure}[t!] 
    \centering
    \includegraphics[width=\linewidth]{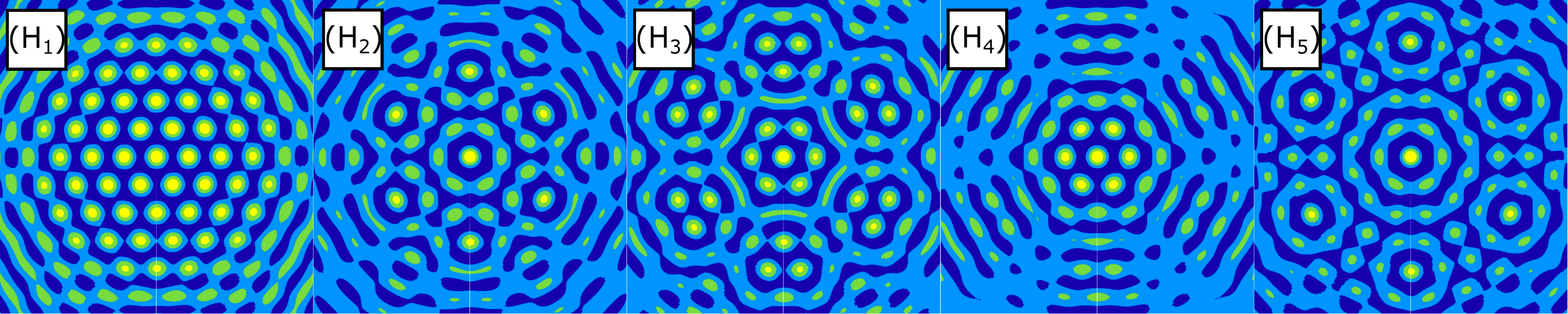}
    \caption{Solutions when $N=4$, given by Theorem \ref{thm:SmallPatch}, for $(H)$exagonal $\sim\mathbb{D}_{6}$ patterns are plotted in a circular disc of radius $50$ circumscribing the square region.}
    \label{fig:N-4Solns}
\end{figure}

Solving the matching equations \eqref{MatchEq} for general $N$ is a difficult task, especially if one seeks to identify all such solutions for a given $N$. To provide a more general result, we will restrict ourselves to the case $6 \mid m$, which includes the case of hexagonal symmetries coming from the group $\mathbb{D}_6$. In this case the matching equations are slightly simplified since $\cos(\frac{m\pi(n-j)}{3}) = \cos(\frac{m\pi(n - 2j)}{3}) = 1$ for each integer $n$ and $j$. For low truncation orders $N=1,\dots,4$ our results in Section~\ref{s:patch;N} below prove that there is a unique solution for which $a_{n}>0$ for all $n\in[0,N]$. In particular, for $m=6$ this solution corresponds to the `standard' cellular hexagon patch ($H_1$) in Figures~\ref{fig:N-1Solns}-\ref{fig:N-4Solns}, where peaks are arranged in a uniform hexagonal tiling. As $N$ increases, we numerically observe that the strictly positive solution possesses a scaling of the form $a_{n}=\mathcal{O}(N^{-1})$ for all $n\in[0,N]$, and, as can be observed in Figure~\ref{fig:LargeN}(a), we find that the rescaled positive solutions appear to converge to a continuous function as $N$ becomes very large. Interestingly, for $N\gg 1$ the matching equations (after rescaling the $a_n$) resemble a Riemann sum which in the limit as $N\rightarrow\infty$ formally yields the nonlocal continuum matching equation
\begin{equation}\label{MatchEqInt}
	\alpha(t) = 2\int_0^{1-t} \alpha(s)\alpha(s + t)\drm s + \int_0^t \alpha(s)\alpha(t-s)\drm s,   
\end{equation}
for each $t \in [0,1]$. Hence, one expects that if $\alpha^*(t)$ satisfies \eqref{MatchEqInt} then there exists a solution of \eqref{MatchEq} such that $a_n \approx \alpha^*(n/(N+1))/(N+1)$, for each $n = 0,\dots,N$. We make this precise with the following theorem, for which the details of moving between the continuum equation \eqref{MatchEqInt} and the matching equation \eqref{MatchEq} are left to Section~\ref{s:patch;N}.

\begin{figure}[t!] 
    \centering
    \includegraphics[width=\linewidth]{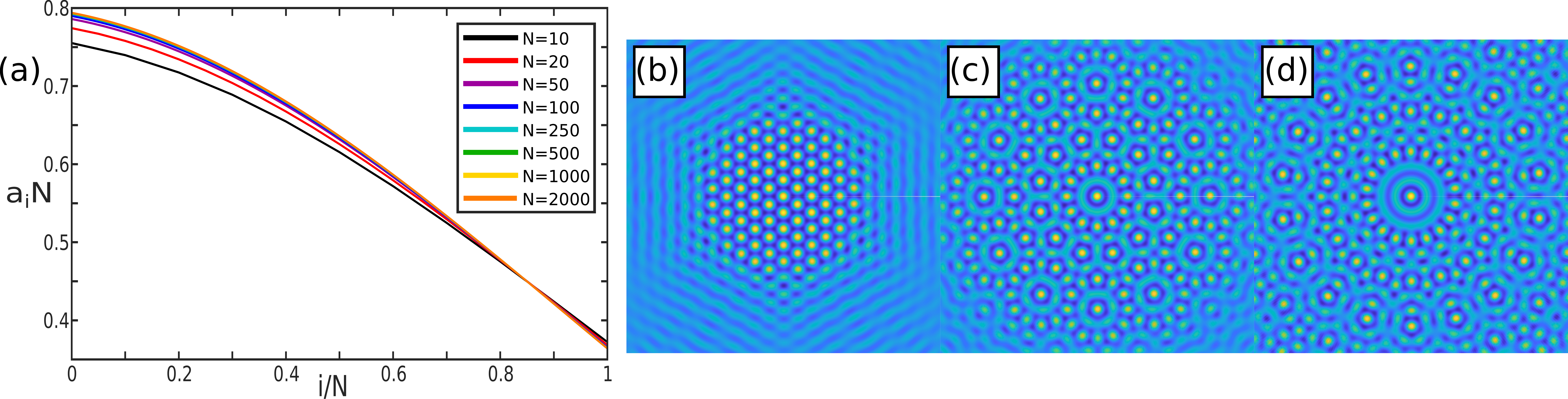}
    \caption{(a) Values of $N\cdot a_{i}$ are plotted against $i/N$ for strictly positive solutions to \eqref{MatchEq} when $6\mid m$, with increasingly large choices of $N$. Solutions given by Theorem \ref{thm:BigPatch} possess structures similar to panels (b)-(d) for (b) $\mathbb{D}_{6}$, (c) $\mathbb{D}_{12}$, and (d) $\mathbb{D}_{18}$ patterns. Here, the plotted profiles are localised solutions to \eqref{eqn:R-D;Galerk} coming from the SHE found numerically for $N=10$, $\mu=5\times10^{-12}$ and $\gamma=0.3$ in a circular disc of radius $100$ circumscribing the square region.}
    \label{fig:LargeN}
\end{figure}

\begin{thm}\label{thm:BigPatch} 
	Fix $m = 6m_0$ with $m_0 \in \mathbb{N}$. Then, there is an $N_0 \geq 1$ such that for all $N \geq N_0$ there exists $\mu_0 > 0$ so that the Galerkin system \eqref{eqn:R-D;Galerk} has a radially localised solution of amplitude $\mathcal{O}(\mu^\frac{1}{2})$ for each $\mu \in (0,\mu_0)$. Precisely, there exists $r_0,r_1 > 0$ such that 
\begin{equation}
    \mathbf{u}_{n}(r) = \frac{\sqrt{24c_{0}}\,k_{c} a_{n}\mu^{\frac{1}{2}}}{\gamma \sqrt{k_{c}\pi}}
    \begin{cases}
         \sqrt{\frac{k_{c} \pi}{2}}\,J_{6m_0n}(k_{c} r)\hat{U}_{0} + \mathcal{O}(\mu^{\frac{1}{2}}) & 0\leq r\leq r_{0},  \\
         r^{-\frac{1}{2}}\,\cos(k_{c} r-\frac{\pi}{4})\hat{U}_{0} + \mathcal{O}(\mu^{\frac{1}{2}}), & r_{0} \leq r \leq r_1\mu^{-\frac{1}{2}},\\
         r^{-\frac{1}{2}}\,\textnormal{e}^{\sqrt{c_{0}}(r_1- \mu^{\frac{1}{2}} r)}\,\cos(k_{c} r-\frac{\pi}{4})\hat{U}_{0} + \mathcal{O}(\mu^{\frac{1}{2}}), & r_1\mu^{-\frac{1}{2}} \leq r,
    \end{cases}
\end{equation} 
where $J_{6m_0n}$ is the $(6m_0n)^\mathrm{th}$ order Bessel function of the first kind. The $a_n$ are such that for all $\varepsilon > 0$, there exists $N_\varepsilon > 0$ such that for all $N \geq N_\varepsilon$ we have 
\begin{equation}\label{an:alpha;bound}
	\sup_{n \in [0,N]} \bigg|a_n - \frac{1}{N+1}\,\alpha^*\bigg(\frac{n}{N+1}\bigg)\bigg| < \varepsilon,
\end{equation}
where $\alpha^*(t)$ is a positive and continuous solution of \eqref{MatchEqInt} for all $t \in [0,1]$. 
\end{thm}

Unlike Theorem~\ref{thm:SmallPatch}, we are only able to identify a single solution of the Galerkin system \eqref{eqn:R-D;Galerk} when $N \gg 1$. This is limited by the fact that $a_n$'s are obtained from a continuous solution to the continuum matching problem \eqref{MatchEqInt}. To obtain the continuum solution we employ a computer-assisted proof, as detailed in Section~\ref{subsec:BigMatch}, whose details are primarily left to the appendix. It is possible that similar computer-assisted proofs could produce other solutions to \eqref{MatchEqInt}, in which case one can follow the work in Section~\ref{subsec:BigMatch} with relative ease to arrive at further solutions of the Galerkin system with large $N$. We emphasise that although the positive solution corresponds to standard localised cellular hexagons ($H_1$) when $m=6$, Theorem \ref{thm:BigPatch} holds for any $m=6 m_{0}$ with $m_{0}\in\mathbb{N}$. For any $m_{0}>1$, the pattern given in Theorem~\ref{thm:BigPatch} has $6 m_{0}$-fold symmetry and so corresponds to a localised quasicrystalline structure, such as those seen in \cite{Subramanian2018localizedPFC,Gokce2020NeuralCrystals}. We refer the reader to Figure~\ref{fig:LargeN} (b)-(d) for examples of such large $N$ solutions for localised $\mathbb{D}_{6}$, $\mathbb{D}_{12}$ and $\mathbb{D}_{18}$ patterns.

%
%
%
%
%
%
%
%
%
%
\section{Numerical Investigation of Localised Patterns}\label{s:Numerics}

In this section we present our numerical results for the localised patterns from the previous section. Throughout we will exclusively focus on the SHE \eqref{e:SH} and restrict our larger-$N$ system to be $N=10$. This allows us to investigate patches for $N=1,\dots,4$ embedded into a higher-dimensional system, while also maintaining computation efficiency. The choice of $N$ can be made considerably higher, however for our chosen radial domain, we have found that the choice of $N=10$ is sufficient. 

Our numerical procedure is described in Appendix \ref{s:num;impl}, which takes the radial domain to be $0 \leq r \leq r_{*}$, discretised into $T$ mesh points $\{r_{i}\}_{i=1}^{T}$ allowing us to numerically solve \eqref{e:SH} using finite difference methods. We take Neumann boundary conditions at the outer radial boundary $r=r_*$, and standard radial boundary conditions at $r=0$ such that $u(r,\theta)$ is smooth at the origin. Following this, we employ a secant continuation code similar to \cite{avitabile2020Numerical} in order to continue solutions beyond the limited parameter regions from the results of the previous section. We fix $\gamma = 1.6$ in \eqref{e:SH} as this is the same choice of parameters as seen in \cite[Figure 24]{lloyd2008localized} where localised hexagon patches are observed undergoing snaking behaviour.

In what follows we first verify our analytical results from the previous section by investigating numerical solutions to \eqref{e:SH} as $\mu\to0^+$ for $N=1,\dots,4$ and $N=10$. We then conclude this section by continuing localised dihedral patterns beyond the small $\mu > 0$ parameter regimes of our main results to observe snaking bifurcation curves.

\subsection{Verification of Analysis}\label{s:Verification}

\begin{figure}[t!]
    \centering
    \includegraphics[width=\linewidth]{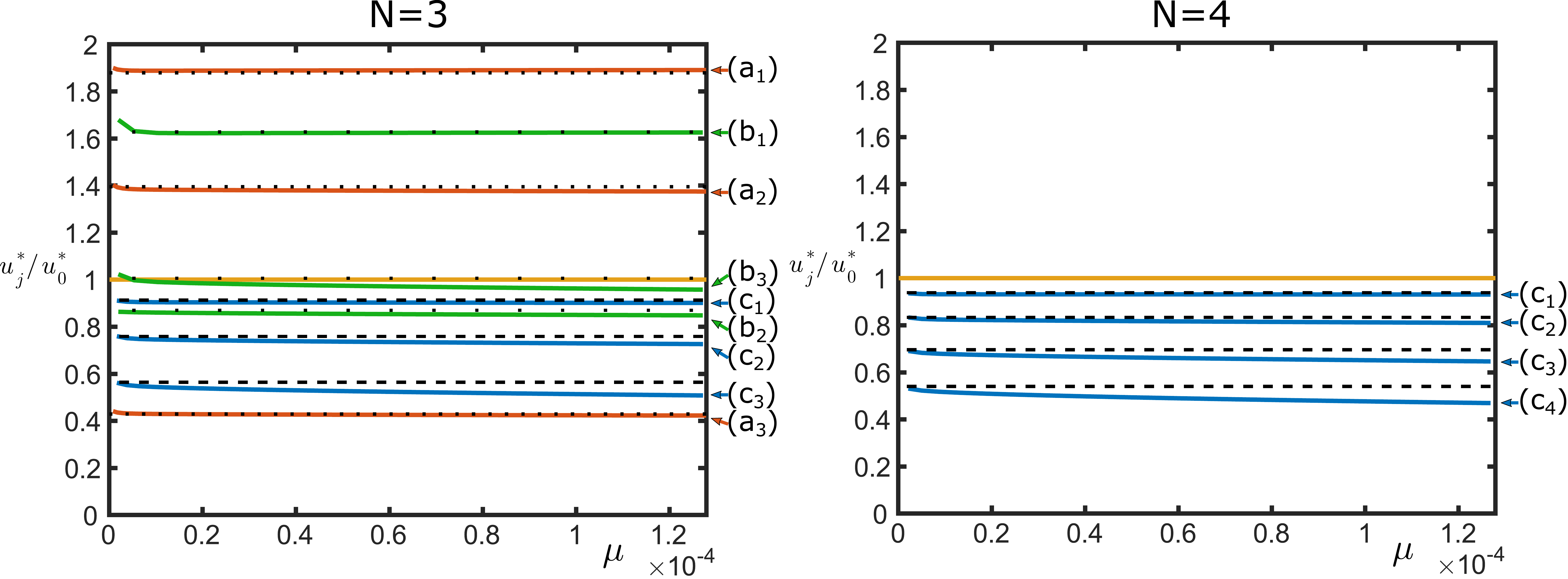}
    \caption{The amplitude ratios $u_{n}^{*}/u_{0}^{*}$ as $\mu$ tends to $0$ are plotted as the curves ($a_n$), ($b_n$) and ($c_n$), for (a) $\mathbb{D}_{2}$ (red), (b) $\mathbb{D}_{3}$ (green), and (c) $\mathbb{D}_{6}$ solutions (blue), respectively. The left plot is for $N=3$ and the right plot is for $N=4$.}
    \label{fig:ampratioN3N4}
\end{figure}

We recall that the theoretical results presented in Section~\ref{s:Results} are for the parameter region $0<\mu\ll1$, and so here we aim to support our theoretical results by numerically investigating solutions of \eqref{e:SH} in the limit as $\mu\to0$. As $\mu$ decreases towards 0, any localisation effects become weaker, meaning that solutions begin to grow in width. As a result, solutions will inevitably be affected by the width of the radial boundary $r=r_{*}$ as $\mu\to0$.  In order to minimise these boundary effects, we choose $r_{*}=2000$, with $T=6000$ mesh points, for the rest of this subsection. We investigated various values for the mesh step-size $\delta r$ with no observable differences in the results; as such, we choose $T=3\,r_*$ rather than $T=10\,r_*$ in order to improve computation speeds. We begin by introducing the `numerical amplitude' $u_{n}^{*}$ of each $\mathbf{v}_{n}$, defined by 
\begin{equation}\label{un*}
    u_{n}^{*} := \frac{\gamma}{\sqrt{3\mu}}\frac{\|\mathbf{v}_{n}\|_{T}}{\|J_{mn}(r)\|_{T}},
    \qquad \textrm{where}\quad \| \mathbf{x}\|_{T}:= \max_{j\in[1,T]}\big\{|x_{j}|\big\}, \qquad \forall \mathbf{x}:=(x_{1},\dots,x_{T})\in\mathbb{R}^{T}.
\end{equation}
We investigate the ratios of $u_{n}^{*}/u_{0}^{*}$ as $\mu\to0$ and compare with $|a^*_{n}|/|a^*_{0}|$ for each $n\in[0,N]$, where $\{a^*_{n}\}_{n=0}^{N}$ are solutions of the nonlinear matching condition \eqref{MatchEq}. We complete this investigation for (a) $\mathbb{D}_{2}$, (b) $\mathbb{D}_{3}$, and (c) $\mathbb{D}_{6}$ solutions to \eqref{e:SH} for $N=3$, as well as a $\mathbb{D}_{6}$ solution for $N=4$. In each case we solve for the `standard' pattern, indicated by ($R_{1}$), ($T_{1}$), and ($H_{1}$) in Figures~\ref{fig:N-1Solns}-\ref{fig:N-4Solns}, respectively, and we present our findings in Figure~\ref{fig:ampratioN3N4}. Notably absent are similar results for the standard $\mathbb{D}_4$ pattern ($S_1$). The reason for this is that the predicted values of $a^*_n$ for this $\mathbb{D}_4$ pattern from Proposition~\ref{prop:N123} below are the same as for the $\mathbb{D}_2$ pattern, and so their inclusion would clutter Figure~\ref{fig:ampratioN3N4}. We do however comment that we have similarly confirmed our theoretical results for the $\mathbb{D}_4$ case, despite them not being presented here. 

\begin{figure}[t!] 
    \centering
    \includegraphics[width=\linewidth]{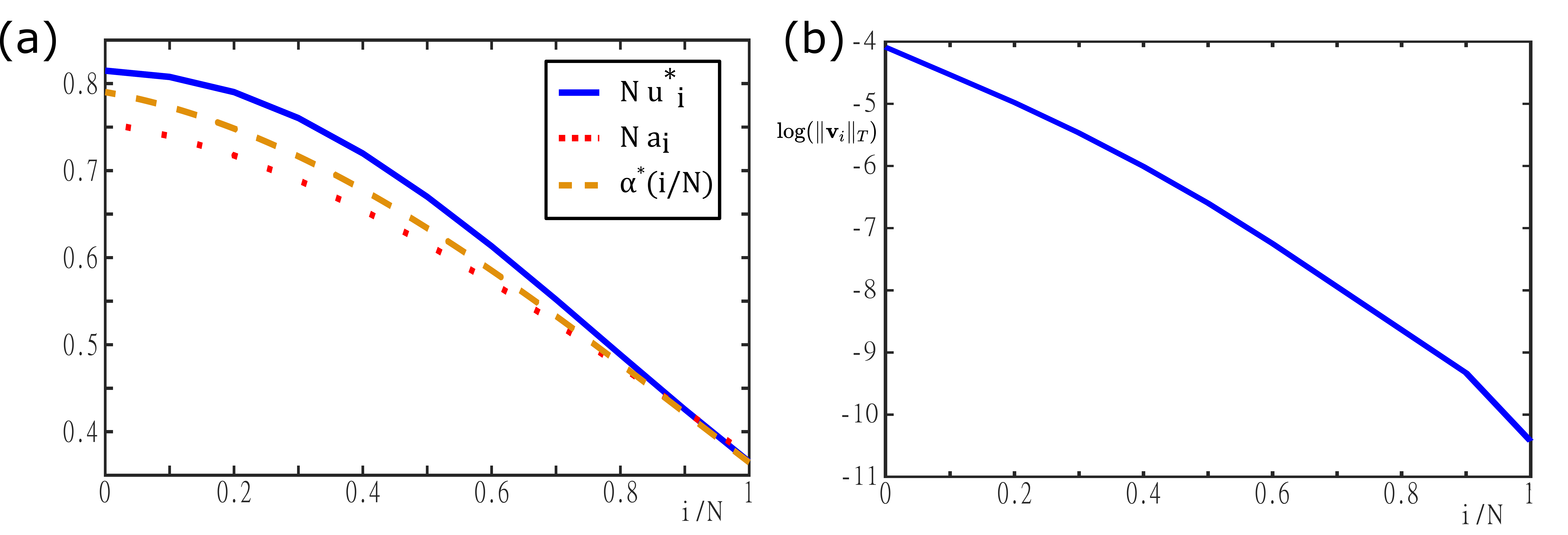}
    \caption{ $(a)$ The rescaled numerical amplitudes $N\,u_{i}^{*}$ are plotted against $i/N$ and compared with $N\,a_{i}$ and $\alpha^{*}(i/N)$, where $a_{i}$ are numerical solutions of \eqref{MatchEq} for $N=10$, $m=6$, and $\alpha^{*}(t)$ is a solution of \eqref{MatchEqInt}. $(b)$ The logarithm of the $T$-norm $\|\mathbf{v}_{i}\|_{T}$ of the numerical Fourier mode $\mathbf{v}_{n}$ is plotted against $i/N$.}
    \label{fig:N10Numerics}
\end{figure}

In Figure~\ref{fig:ampratioN3N4} we present results for $N=3$ (left) and $N=4$ (right), for (a) $\mathbb{D}_{2}$, (b) $\mathbb{D}_{3}$ and (c) $\mathbb{D}_{6}$ solutions to \eqref{e:SH}. In each case, we have a straight horizontal line valued at $1$ on the y-axis, representing the ratio of $u_{0}^{*}/u_{0}^{*}$, and so is naturally fixed for each choice of $N$ and $m$. The horizontal dotted and dashed lines represent the ratios given by solutions of the matching equation \eqref{MatchEq}; the (a) dotted, (b) sparsely-dashed, and (c) densely-dashed lines represent the cases when (a) $2\mid m$ and $3\nmid m$, (b) $2\nmid m$ and $3\mid m$, and (c) $6\mid m$, respectively. We observe that the ratio $u_{n}^{*}/u_{0}^{*}$ of a $\mathbb{D}_{2}$ solution, indicated by ($a_{n}$), tends to the associated dotted line from our theoretical results as $\mu\to0^+$, for each $n\in[0,N]$. The ratio $u_{n}^{*}/u_{0}^{*}$ of a $\mathbb{D}_{3}$ solution, indicated by ($b_{n}$), begins further away from the sparsely dotted line than ($a_{n}$) for moderate $\mu$ values. However, these solutions also appear to tend to their respective theoretical predictions as $\mu\to0^+$, for each $n\in[0,N]$. Finally, the ratio $u_{n}^{*}/u_{0}^{*}$ of a $\mathbb{D}_{6}$ solution, indicated by ($c_{n}$) in each figure, is plotted with respect to $\mu$. We note that all of the values of ($c_n$) are less than $1$, in contrast to the values of ($a_n$) and ($b_n$), and tend to their associated dashed lines from our theoretical results as $\mu\to0^+$. 
 We briefly comment on the reason why some ratios $u^*_j/u^*_0$ in Figure~\ref{fig:ampratioN3N4} begin further from their theoretical predictions than others. While this cannot be fully understood without a more thorough numerical investigation, it can be partially explained by our definition of $u^*_j$. In \eqref{un*}, we do not account for any localisation in our solutions, which will naturally create some error for moderate values of $\mu$. Furthermore, since the maximum value of the Bessel function $J_{mn}(r)$ is further from the origin for larger values of $mn$, the localisation causes a greater error than for smaller values of $mn$. Hence, in Figure~\ref{fig:ampratioN3N4} we see a larger discrepancy between numerical and theoretical results for moderate values of $\mu$ as either $m$ or $n$ is increased.

In Figure~\ref{fig:N10Numerics} we present results for the `standard' $\mathbb{D}_{6}$ solution to \eqref{e:SH} when $N=10$. In Figure~\ref{fig:N10Numerics} (a) we plot $N u_{n}^{*}$ against $n/N$, where $u_{n}^{*}$ is defined in \eqref{un*}, and compare with $N a_{n}$ and $\alpha^{*}(n/N)$, where $a_{n}$ are numerical solutions of the matching condition \eqref{MatchEq} for $N=10$, $m=6$, and $\alpha^{*}$ is a numerical solution of \eqref{MatchEqInt}. The amplitudes are in quite good agreement, which could also be improved by further increasing $r_{*}$  or $N$. In Figure~\ref{fig:N10Numerics} (b) we plot $\log(\|\mathbf{v}_{n}\|_{T})$ against $n/N$ and observe exponential decay as $n\to N$. This suggests that the coefficients of each Fourier mode $\cos(m n \theta)$ decay exponentially as $n$ increases, thus providing numerical evidence that the Fourier series \eqref{e:FS} might converge to a continuous solution as $N\to\infty$.

\subsection{Continuation of Solutions}\label{s:Continuation}

One of the major benefits of this work is its application to the numerical study of localised planar patterns. To illustrate, we can begin with an initial guess of the form
\begin{equation}
    u_{n}(r) = \beta a_{n}   J_{m n}(r) \;\textnormal{exp}(-\sqrt{\mu} \,r /2) \implies
    \left[\mathbf{v}_{n}\right]_{j} =\beta  a_{n} J_{m n}(r_{j}) \;\textnormal{exp}(-\sqrt{\mu} \,r_{j} /2),\label{num:guess}
\end{equation}
for each $n\in[0,10]$ where $m\in\mathbb{N}$. Here $\beta$ is a scaling term that accounts for our choices of $(\mu,\gamma)$, the values of $a_{n}$ are given by our theoretical solutions in Section~\ref{subsec:SmallMatch}, and the exponential term gives us an approximation for localisation in the far-field. In this parameter regime, solutions are more strongly localised and hence are smaller in width. Hence, for this subsection we choose $r_{*}=100$ with $T=1000$ mesh points for computational speed. Then, by first solving the nonlinear matching condition \eqref{MatchEq} and substituting the subsequent solution $a_n$ into \eqref{num:guess}, we are able to construct very effective initial guesses for numerical continuation of such patterns. Furthermore, for moderate values of $\mu$ it is often sufficient to solve a low-dimension algebraic system, i.e. for $N_{0}=1,2,3$, which can be embedded into an initial guess of the form \eqref{num:guess} with a higher dimension $N_{1}$, where $a_{n}=0$ for all $N_{0}<n\leq N_{1}$. We utilise this approach in order to find small-amplitude localised $\mathbb{D}_{m}$ patterns and continue them to larger amplitudes. Three examples of continued localised $\mathbb{D}_{m}$ patterns are presented in the Figures~\ref{fig:D2Snake}, \ref{fig:D3Snake}, and \ref{fig:D4Snake} for $m = 2,3,4$, respectively. 

\begin{figure}[t!] 
    \centering
    \includegraphics[width=\linewidth]{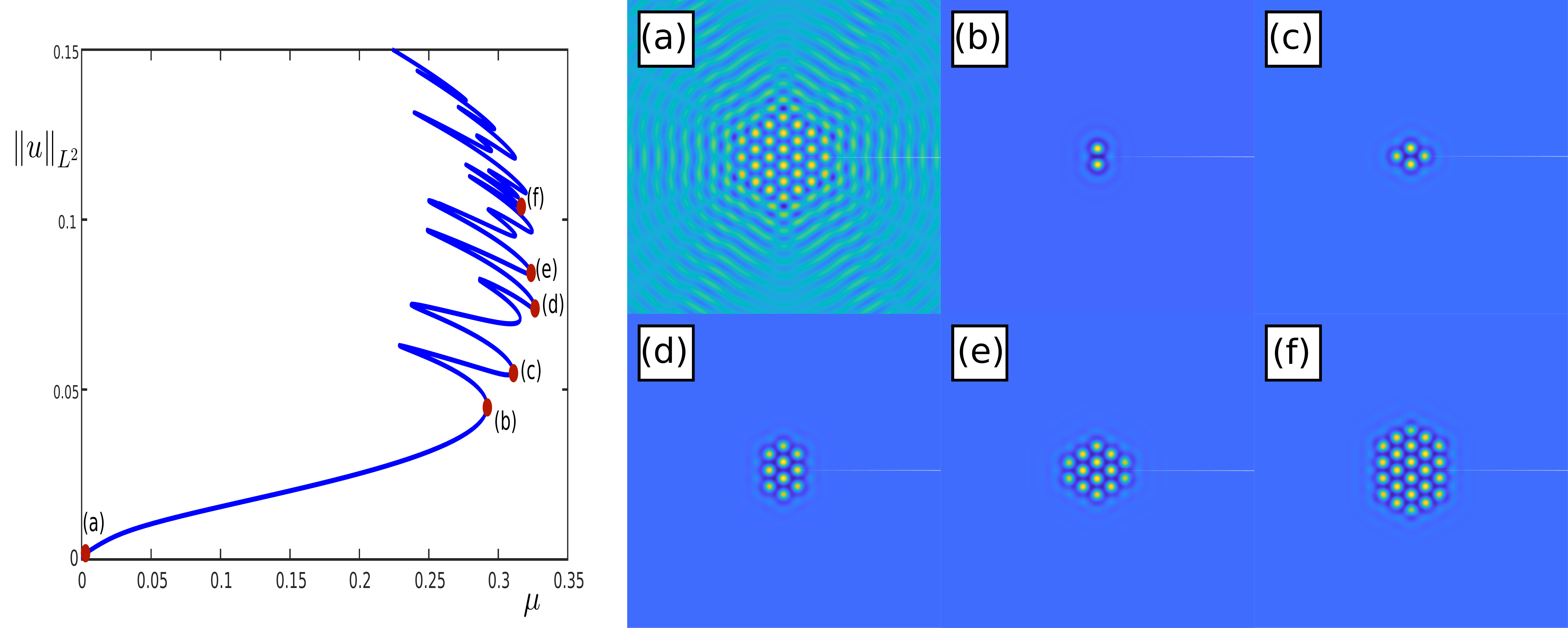}
        \caption{Localised $\mathbb{D}_2$ patterns bifurcating from the homogeneous state $u = 0$, as plotted in the left panel for $N=10$ and $\gamma = 1.6$. Panels (a)-(f) show contour plots of localised $\mathbb{D}_2$ solutions at various points (red circles) on the bifurcation curve.}
    \label{fig:D2Snake}
\end{figure}

\begin{figure}[t!] 
    \centering
    \includegraphics[width=\linewidth]{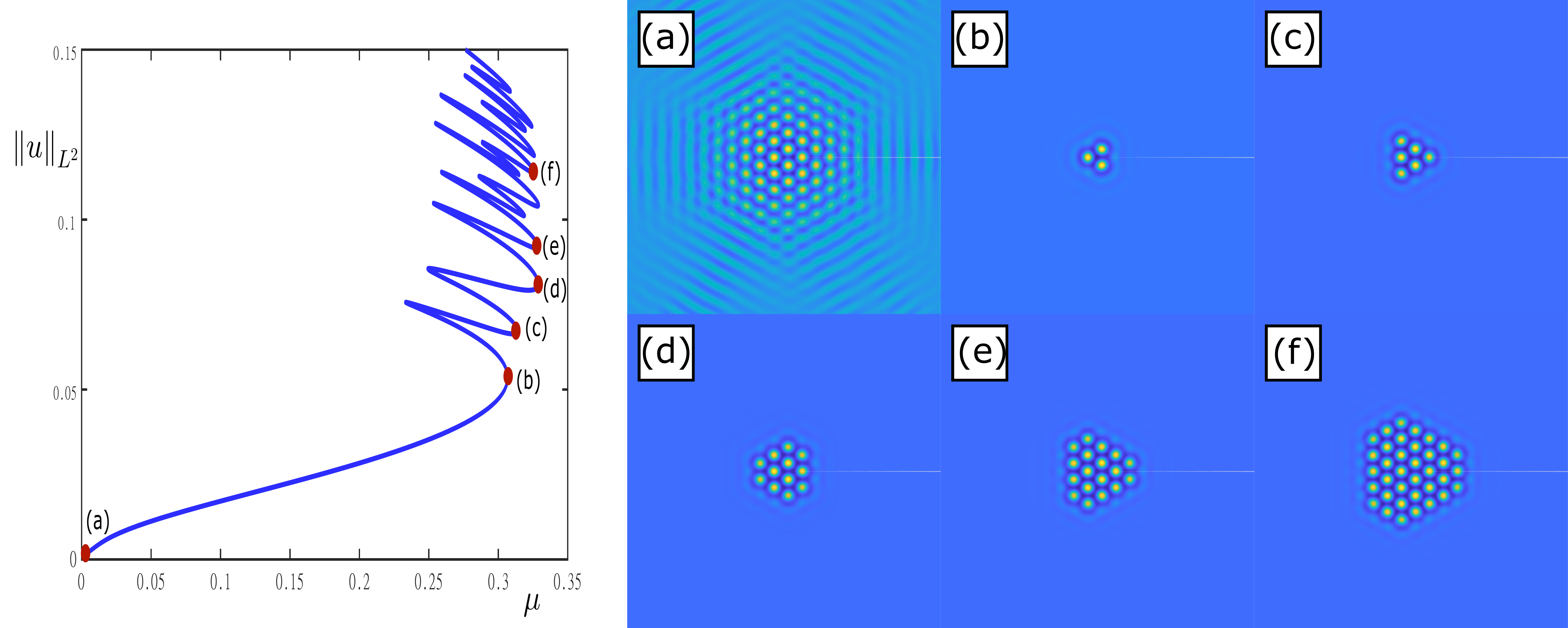}
    \caption{Localised $\mathbb{D}_3$ patterns bifurcating from the homogeneous state $u = 0$, as plotted in the left panel for $N=10$ and $\gamma = 1.6$. Panels (a)-(f) show contour plots of localised $\mathbb{D}_3$ solutions at various points (red circles) on the curve.}
    \label{fig:D3Snake}
\end{figure}

The continuation of localised $\mathbb{D}_{6}$ solutions was extensively covered in \cite{lloyd2008localized}, and so here we focus on the novel $\mathbb{D}_2$, $\mathbb{D}_3$, and $\mathbb{D}_4$ solutions. In Figure~\ref{fig:D2Snake}, we present our numerical results for a localised $\mathbb{D}_{2}$ solution to \eqref{e:SH}. We begin with an initial guess of the form \eqref{num:guess} with $a_0 = -1$, $a_1 = -\sqrt{2}$ and $a_n=0$ for $1<n\leq 10$ such that $(a_0, a_1)$ satisfies the $\mathbb{D}_{2}$ matching equation \eqref{MatchEq} when $N=1$. Then, for $\mu=0.1$, the initial guess converges to a localised solution consisting of two spikes in close proximity, corresponding to the ($R_{1}$) solution in Figure~\ref{fig:N-1Solns}. As $\mu$ increases in Figure~\ref{fig:D2Snake}, we observe that the $\mathbb{D}_{2}$ solution undergoes snaking behaviour, where variations in $\mu$ cause solutions to gain extra peaks and grow in width. 

Similarly, in Figures~\ref{fig:D3Snake} and \ref{fig:D4Snake} we present our numerical results for a localised $\mathbb{D}_3$ and $\mathbb{D}_{4}$ solution to \eqref{e:SH}, respectively. In each case we solve the matching equation \eqref{MatchEq} for $N=1$ and substitute our solution in the initial guess \eqref{num:guess}; notably, the initial guess for the $\mathbb{D}_4$ solution is identical to the $\mathbb{D}_{2}$ solution, other than a change of the value of $m$ in \eqref{num:guess}. Then, for $\mu=0.1$ the initial guess converges to a localised solution consisting to three ($m = 3$) or four ($m = 4$) spikes, which is predicted by ($T_1$) and ($S_1$), respectively, in Figure~\ref{fig:N-1Solns}. As $\mu$ increases we again observe that the $\mathbb{D}_3$ and $\mathbb{D}_4$ solutions exhibit snaking behaviour. This behaviour appears to be quite robust for larger dihedral symmetries as we have similarly observed such snaking at least for $\mathbb{D}_{6}, \mathbb{D}_8, \dots, \mathbb{D}_{14}$ solutions as well. However, these results are not presented here for brevity. In one spatial direction this snaking behaviour is known as homoclinic snaking, described by homoclinic cycles in phase space, and is now well-understood \cite{burke2007homoclinic,burke2007snakes,Avitabile2010Snake,beck2009snakes,barbay2008homoclinic,Bergeon2008Eckhaus,lloyd2015homoclinic,rankin2014,bramburgerLattice}. However, in two or more spatial directions such snaking bifurcation curves are not well understood, and there continues to be significant interest in this area \cite{Champneys2021Editorial,deWitt2019Beyond,azimi2021Modified,parrarivas2021dark,CISTERNAS2020Gapped,Uecker2020Brusselator,Aguilera-Rojas2021Localized,bramburgerLattice2}. We note that in Figures~\ref{fig:D2Snake}, \ref{fig:D3Snake}, and \ref{fig:D4Snake} we see that emerging peaks also appear to be subject to hexagonal packing, suggesting that the domain-covering hexagonal lattice may have a pivotal role in the structure of any observed localised patterns.  

\begin{figure}[t!] 
    \centering
    \includegraphics[width=\linewidth]{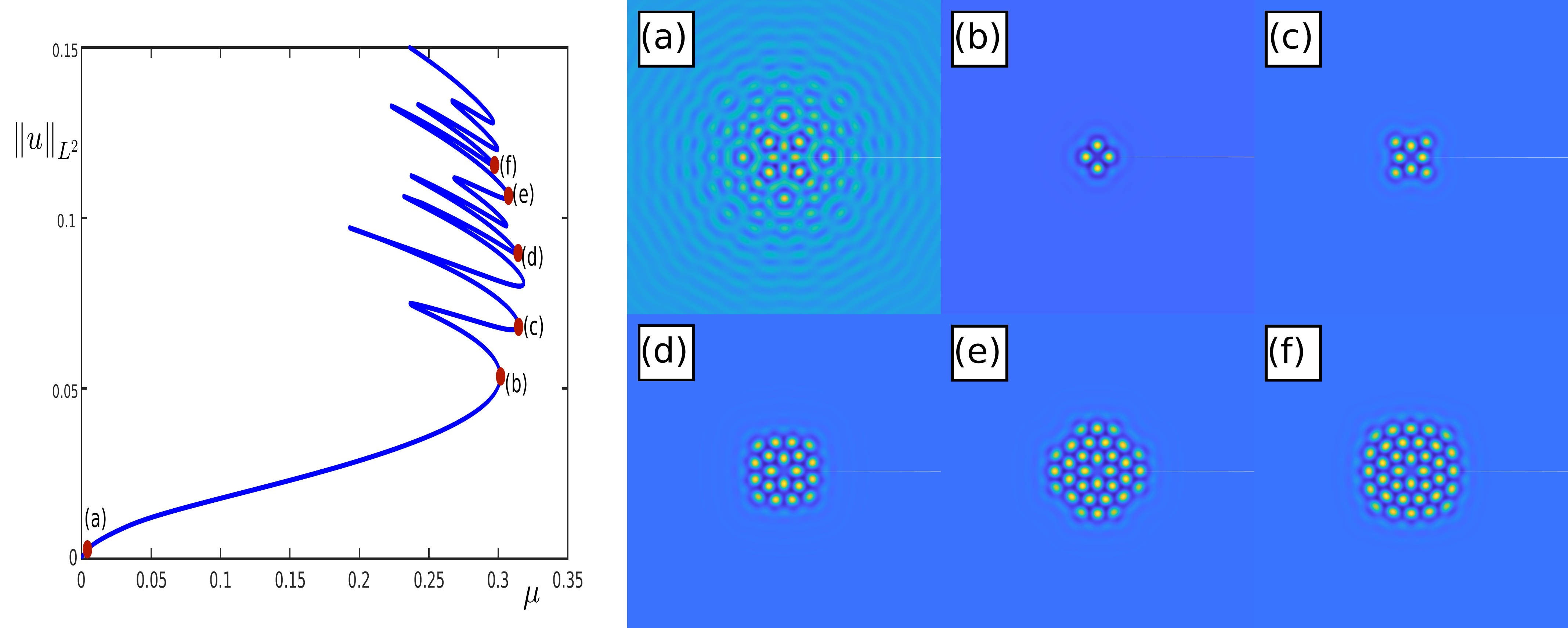}
    \caption{Localised $\mathbb{D}_4$ patterns bifurcating from the homogeneous state $u = 0$, as plotted in the left panel for $N=10$ and $\gamma = 1.6$. Panels (a)-(f) show contour plots of localised $\mathbb{D}_4$ solutions at various points (red circles) on the curve.}
    \label{fig:D4Snake}
\end{figure}

In the parameter regions of Figures~\ref{fig:D2Snake}-\ref{fig:D4Snake} solutions continue into the strongly localised regime, where $\mu$ is moderately-valued, such that localised solutions resemble the $N=1$ patterns in Figure~\ref{fig:N-1Solns}. In order to capture more complicated patterns, we numerically solve \eqref{e:SH} when $N=4$ for localised $\mathbb{D}_6$ solutions in the weakly localised regime with $\gamma=0.4$. For this parameter choice, the hysteretic region of the bistable nonlinearity in the SHE is very small and solutions never become strongly localised. We present our results for localised $\mathbb{D}_6$ patterns in Figure~\ref{fig:N4HexPatterns} where one can observe each of the distinct solutions ($H_j$) predicted in Figure~\ref{fig:N-4Solns} for $j=1,\dots,5$, and track the associated solution curves in $\mu$-parameter space. Hence, we are able to numerically observe our theoretical solutions for \eqref{e:SH} when $N=4$ in a weakly localised parameter regime.

\begin{figure}[t!]
    \centering
    \includegraphics[width=\linewidth]{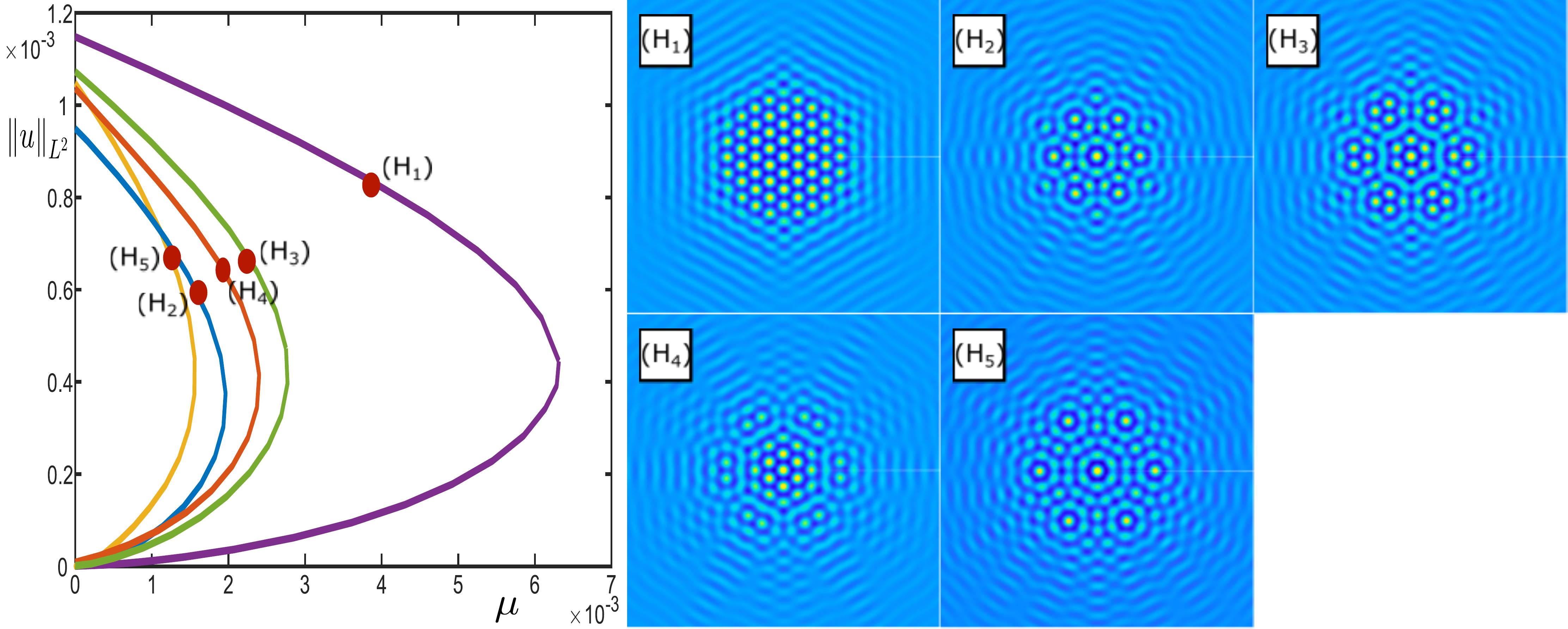}
    \caption{$\mathbb{D}_{6}$ localised patches in an $N=4$ truncated system, as seen in Figure \ref{fig:N-4Solns}, continued for $\gamma=0.4$. Contour plots are presented for particular parameter values indicated by the symbols on the bifurcation curves.}
    \label{fig:N4HexPatterns}
\end{figure}

%
%
%
%
%
%
%
%
%
%
\section{Localised Solutions to the Galerkin System}\label{s:Matching}

The goal of this section is to divide the dynamics of the Galerkin system \eqref{eqn:R-D;Galerk} into separate regions over the independent variable $r > 0$ and then provide the necessary conditions for matching the solutions in each region together. We remark that satisfying the resulting matching conditions is left to the following section. We note that there exists a rescaling of the form
\begin{equation}
    \widetilde{r} = k_{c} r, \qquad \widetilde{\mathbf{u}} = k_{c}^{-1} \mathbf{u}, \qquad \widetilde{\mathbf{Q}}(\mathbf{x},\mathbf{y}) = k_{c}^{-1} \mathbf{Q}(\mathbf{x},\mathbf{y}), \qquad \widetilde{\mathbf{M}}_{j} = k_{c}^{-2} \mathbf{M}_{j}, \label{k:rescaling}
\end{equation}
for $j=1,2$, such that $\widetilde{\mathbf{M}}_{1}$ has a repeated eigenvalue of $\lambda=-1$ and \eqref{eqn:R-D} remains unchanged. Hence, without loss of generality, we will take $\lambda=-1$ throughout since the case when $\lambda= -k_{c}^{2}$ can be recovered by inverting \eqref{k:rescaling} at the end. To formulate the problem properly, we express \eqref{eqn:R-D;Galerk} as the following first-order system,
\begin{equation}\label{R-D:U;vec}
    \frac{\textnormal{d}}{\textnormal{d} r}\mathbf{U} = \mathcal{A}(r)\mathbf{U} + \mathbf{F}(\mathbf{U},\mu), 
\end{equation}
where $\mathbf{U}:=[(\mathbf{u}_{n},\partial_{r}\mathbf{u}_{n})^{T}]_{n=0}^{N}$, $\mathcal{A}(r) = \mathrm{diag}(\mathcal{A}_0(r),\mathcal{A}_1(r),\dots,\mathcal{A}_N(r))$ and $\mathbf{F}(\mathbf{U};\mu) = [\mathbf{F}_n(\mathbf{U};\mu)]_{n = 0}^N$ with
\begin{align}
    \mathcal{A}_{n}(r) &= \begin{pmatrix}
    \mathbb{O}_{2} & \mathbbm{1}_{2} \\ \frac{(m n)^{2}}{r^{2}}\mathbbm{1}_{2} + \mathbf{M}_{1} & -\frac{1}{r}\mathbbm{1}_{2}
    \end{pmatrix},\quad  \mathbf{F}_{n}(\mathbf{U};\mu) = \begin{pmatrix}
    \mathbf{0} \\ \displaystyle \mu \mathbf{M}_{2}\mathbf{u}_{n} + \sum_{i+j=n} \mathbf{Q}(\mathbf{u}_{|i|},\mathbf{u}_{|j|}) + \sum_{i+j+k=n} \mathbf{C}(\mathbf{u}_{|i|},\mathbf{u}_{|j|},\mathbf{u}_{|k|})
    \end{pmatrix},\nonumber
\end{align}
for each $n\in[0,N]$, where we recall that $\mathbb{O}_2$ and $\mathbbm{1}_2$ are the $2\times 2$ zero and identity matrices, respectively. Recall that the nonlinear sums in $\mathbf{F}_{n}(\mathbf{U};\mu)$ can equivalently be written as
\begin{equation}
	\begin{split}
    		\sum_{i + j = n} \mathbf{Q}(\mathbf{u}_{|i|},\mathbf{u}_{|j|}) & = \sum_{i= \max\{n,0\}- N}^{\min\{n,0\} + N} \mathbf{Q}(\mathbf{u}_{|i|},\mathbf{u}_{|n-i|}),\\
    		\sum_{i + j + k = n} \mathbf{C}(\mathbf{u}_{|i|},\mathbf{u}_{|j|},\mathbf{u}_{|k|}) &= \sum_{i=-N}^{N} \left\{\sum_{j = \max\{n-i,0\} - N}^{\min\{n-i,0\} + N} \mathbf{C}(\mathbf{u}_{|i|},\mathbf{u}_{|j|},\mathbf{u}_{|n-i-j|})\right\}.
	\end{split}
\end{equation} 
Our goal is to obtain exponentially decaying solutions of \eqref{R-D:U;vec}, which give way to the $\mathbf{u}_n(r)$ in the truncated Fourier expansion \eqref{FourierExp:R-D} of the approximate solution $\mathbf{u}(r,\theta)$.

Since we are interested in solutions which decay as $r \to \infty$, we begin by noting that for each $n\in[0,N]$ we have $\lim_{r\to\infty} \mathcal{A}_{n}(r)=\mathcal{A}_{\infty}$. So, in the limit as $r\to\infty$, the linearised dynamics of \eqref{R-D:U;vec} decouple for each $n$ and can be understood through the $N+1$ identical eigenvalue problems
\begin{equation}\label{eig:lin}
    \lambda\mathbf{U}_{n} = \left[\mathcal{A}_{\infty} + D_{\mathbf{U}_{n}}\mathbf{F}_{n}(\mathbf{0};\mu)\right]\mathbf{U}_{n} = \begin{bmatrix}
    \mathbb{O}_{2} & \mathbbm{1}_{2} \\
    \mathbf{M}_{1} + \mu \mathbf{M}_{2} & \mathbb{O}_{2}
    \end{bmatrix} \mathbf{U}_{n}.
\end{equation}
The eigenvalue problem \eqref{eig:lin} reduces to solving the following equation for $\lambda$,
\begin{equation}\label{det:eig}
    \det\,\bigg(\lambda^2\mathbbm{1}_2 - \mathbf{M}_{1} - \mu\mathbf{M}_{2}\bigg) = 0
\end{equation}
which, after applying a suitable similarity transformation $\mathbf{M}\mapsto\mathbf{T}^{-1} \mathbf{M} \mathbf{T}$, where
\begin{equation}
    \mathbf{T}^{-1}\mathbf{M}\mathbf{T} = \begin{bmatrix}
     \langle \hat{U}_0^*, \mathbf{M} \hat{U}_0\rangle & \langle \hat{U}_0^*, \mathbf{M} \hat{U}_1\rangle \\
    \langle \hat{U}_1^*, \mathbf{M} \hat{U}_0\rangle & \langle \hat{U}_1^*, \mathbf{M} \hat{U}_1\rangle
    \end{bmatrix}
\end{equation}
for any matrix $\mathbf{M}\in\mathbb{R}^{2\times2}$, the determinant \eqref{det:eig} can be written as
\begin{equation}\label{eig:proj}
    \det\,\begin{pmatrix}
     (\lambda^2+1) - \mu \langle \hat{U}_0^*, \mathbf{M}_2 \hat{U}_0\rangle & -1 - \mu \langle \hat{U}_0^*, \mathbf{M}_2 \hat{U}_1\rangle \\
    -\mu \langle \hat{U}_1^*, \mathbf{M}_2 \hat{U}_0\rangle & (\lambda^2+1) - \mu \langle \hat{U}_1^*, \mathbf{M}_2 \hat{U}_1\rangle
    \end{pmatrix} = 0.
\end{equation}
Here, we have reintroduced the linearly independent vectors $\hat{U}_{0}$, $\hat{U}_{1}\in\mathbb{R}^{2}$ introduced in Hypothesis~\ref{R-D:hyp;1}, defined by
\begin{equation}
    \mathbf{M}_{1}\hat{U}_{0} = -\hat{U}_{0}, \qquad \mathbf{M}_{1}\hat{U}_{1} = -\hat{U}_{1} + \hat{U}_{0} \nonumber
\end{equation}
following the rescaling \eqref{k:rescaling}, and equipped with adjoint eigenvectors $\hat{U}_{0}^{*}$, $\hat{U}_{1}^{*}$ such that $\langle \hat{U}_{i}^{*}, \hat{U}_{j} \rangle = \delta_{i,j}$. The solutions of \eqref{eig:proj} are given by
\begin{equation}
    \lambda_{\pm}^2 = -1 + \frac{\mu}{2}\big[\langle \hat{U}_0^*, \mathbf{M}_2 \hat{U}_0\rangle + \langle \hat{U}_1^*, \mathbf{M}_2 \hat{U}_1\rangle\big] \pm \sqrt{\mu\langle \hat{U}_1^*, \mathbf{M}_2 \hat{U}_0\rangle + x\mu^2} = -1 \pm \sqrt{\mu\langle \hat{U}_1^*, \mathbf{M}_2 \hat{U}_0\rangle} + \textnormal{O}(\mu),
\end{equation}
where the coefficient $x\in\mathbb{R}$ depends on $\mathbf{M}_{2}$. Thus, for $\mu=0$ the linear system has spatial eigenvalues $\lambda=\pm\textnormal{i}$, with double algebraic multiplicity. By Hypothesis \ref{R-D:hyp;2} we have that $\langle \hat{U}_1^*, \mathbf{M}_2 \hat{U}_0\rangle<0$ and so, as $\mu$ increases off of 0, the system undergoes a Hamilton--Hopf bifurcation such that the eigenvalues split off of the imaginary axis and into the complex plane. We then expect localised solutions to bifurcate from the homogeneous state in the region $0<\mu\ll1$ since in this region each of the $N+1$ linearised equations has a two-dimensional eigenspace of solutions that exponentially decay to zero as $r \to \infty$.
 
In order to construct localised solutions to  \eqref{R-D:U;vec}, which correspond to localised solutions of \eqref{eqn:R-D;Galerk}, we utilise local invariant manifold theory for radial systems, as seen in \cite{scheel2003radially}. The general idea is as follows. In \S\ref{subsec:Core} we construct the set of all small-amplitude solutions to \eqref{R-D:U;vec} that remain bounded as $r\to0$, which we call the `core manifold'. This manifold is denoted by $\mathcal{W}_{-}^{cu}(\mu)$, where we have used the notation for a local centre-unstable manifold, and is constructed on the bounded sub-domain $r\in[0,r_{0}]$, for some fixed $r_{0}\gg1$. Then, in \S\ref{subsec:Far} we construct the set of all small-amplitude solutions to \eqref{R-D:U;vec} that decay exponentially as $r\to\infty$, which we call the `far-field manifold'. This manifold is denoted by $\mathcal{W}_{+}^{s}(\mu)$, where we have used the notation for a local stable manifold, and can be constructed on $r\geq r_{\infty}$, for some fixed $r_{\infty}>0$. The value of $r_{0}$ can be freely chosen as long as it is sufficiently large and so we choose $r_{0}\geq r_{\infty}$ such that the core and far-field manifolds overlap and can be matched at the coincidental point $r=r_{0}$; see Figure \ref{fig:manifolds}. In \S~\ref{subsec:CoreFarMatch} we show that this matching can be done under the assumption that we have a nondegenerate solution to a nonlinear system of equations, which we refer to as the 'matching equations'. In Section~\ref{s:patch;N} we provide the properties and some solutions to these matching equations; combining these results with the work in this section leads to our main results in Section~\ref{s:Results}, since any function that lies on the intersection of both $\mathcal{W}^{cu}_{-}(\mu)$ and $\mathcal{W}^{s}_{+}(\mu)$ is, by definition, a localised solution of \eqref{R-D:U;vec}.

\begin{figure}[t!] 
    \centering
    \includegraphics[width=0.75\linewidth]{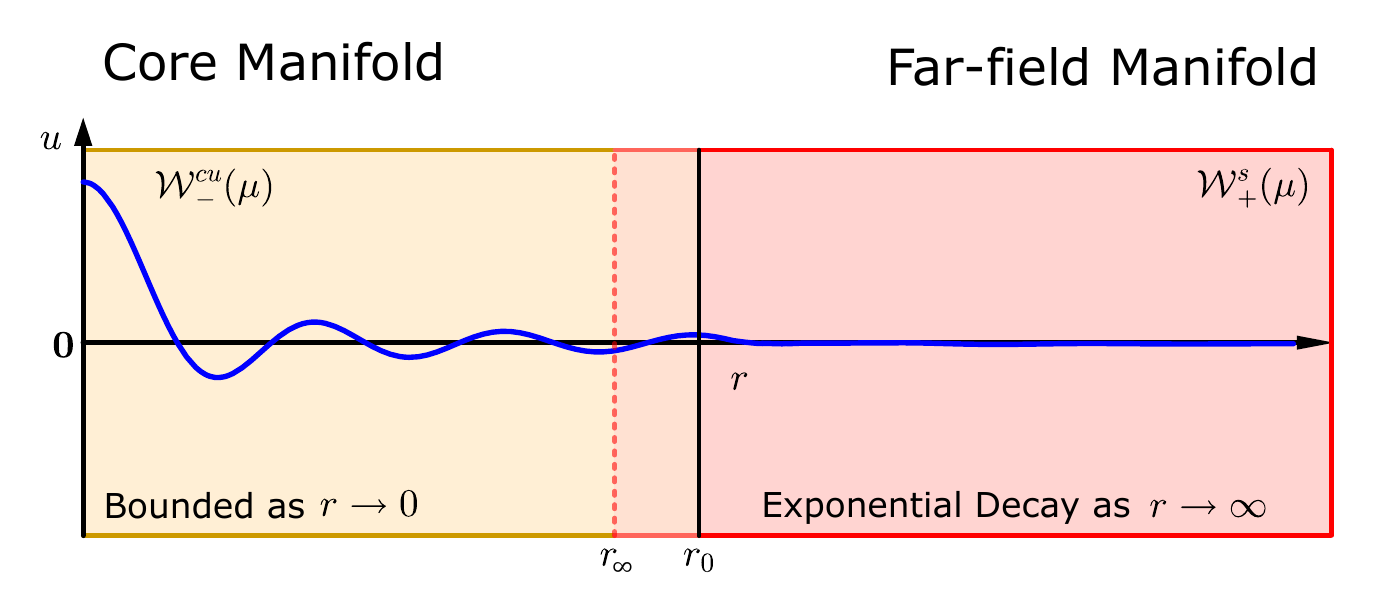}
    \caption{The local invariant manifolds $\mathcal{W}^{cu}_{-}(\mu)$ and $\mathcal{W}^{s}_{+}(\mu)$ are constructed over sub-domains $0\leq r\leq r_{0}$ and $r_{\infty}\leq r<\infty$, respectively. We choose $r_{0}>r_{\infty}$ and look for intersections between $\mathcal{W}^{cu}_{-}(\mu)$ and $\mathcal{W}^{s}_{+}(\mu)$ at the point $r=r_{0}$.}
    \label{fig:manifolds}
\end{figure}

\subsection{The Core Manifold}\label{subsec:Core}

Here we will characterise the core manifold, $\mathcal{W}_{-}^{cu}(\mu)$, for $0 < \mu \ll 1$, which contains all small-amplitude solutions to \eqref{eqn:R-D;Galerk} that remain bounded as $r\to0$. This is a local invariant manifold, and so we determine it on some bounded interval $r\in[0,r_{0}]$, for a large fixed $r_{0}>0$. We begin by noting that at the bifurcation point $\mu=0$ the linearised behaviour about $\mathbf{U} = 0$ of  \eqref{R-D:U;vec}, given by
\begin{equation}\label{Vlinear}
	\frac{\textnormal{d}}{\textnormal{d} r}\mathbf{V} = \mathcal{A}(r)\mathbf{V},
\end{equation}
decouples into $(N+1)$ distinct systems of the form 
\begin{equation}
	\frac{\textnormal{d}}{\textnormal{d} r}\mathbf{v}_{n} = \mathcal{A}_{n}(r)\mathbf{v}_{n},
\end{equation} 
for each $n\in[0,N]$. Then, the linear system $\partial_{r}\mathbf{v}_{n} = \mathcal{A}_{n}(r)\mathbf{v}_{n}$ has solutions of the form
\begin{equation}\label{Ntup:SH;lin,soln}
	\mathbf{v}_{n}(r) = \sum_{j=1}^{4}  d^{(n)}_{j} \mathbf{V}^{(n)}_{j}(r), 
\end{equation}
where   
\begin{align}
   \begin{split} \mathbf{V}_{1}^{(n)}(r) &= \sqrt{\frac{\pi}{2}}\begin{pmatrix}
    J_{mn}(r)\hat{U}_{0} \\ \frac{\textnormal{d}}{\textnormal{d} r}\big[J_{mn}(r)\big]\hat{U}_{0}
    \end{pmatrix}, \qquad \mathbf{V}_{2}^{(n)}(r) = \sqrt{\frac{\pi}{2}}\begin{pmatrix}
    r J_{mn +1 }(r) \hat{U}_{0} + 2 J_{mn}(r)\hat{U}_{1} \\ \frac{\textnormal{d}}{\textnormal{d} r}\big[r J_{mn +1}(r)\big]\hat{U}_{0} + 2 \frac{\textnormal{d}}{\textnormal{d} r}\big[J_{mn}(r)\big]\hat{U}_{1}
    \end{pmatrix}, \\
    \mathbf{V}_{3}^{(n)}(r) &= \sqrt{\frac{\pi}{2}}\begin{pmatrix}
    Y_{mn}(r)\hat{U}_{0} \\ \frac{\textnormal{d}}{\textnormal{d} r}\big[Y_{mn}(r)\big]\hat{U}_{0}
    \end{pmatrix}, \qquad \mathbf{V}_{4}^{(n)}(r) = \sqrt{\frac{\pi}{2}}\begin{pmatrix}
    r Y_{mn +1 }(r) \hat{U}_{0} + 2  Y_{mn}(r)\hat{U}_{1} \\ \frac{\textnormal{d}}{\textnormal{d} r}\big[r Y_{mn +1}(r)\big]\hat{U}_{0} + 2 \frac{\textnormal{d}}{\textnormal{d} r}\big[Y_{mn}(r)\big]\hat{U}_{1}
    \end{pmatrix},\end{split}
\end{align}
and $J_{\nu}(r)$, $Y_{\nu}(r)$ are $\nu$-th order Bessel functions of the first- and second-kind, respectively. To see how we arrive at the above solutions for $\mathbf{v}_n$, we recall that the system $\frac{\textnormal{d}}{\textnormal{d} r}\mathbf{v}_{n} = \mathcal{A}_{n}(r)\mathbf{v}_{n}$ is equivalent to 
\begin{equation}
    	\left(\frac{\textnormal{d}^{2}}{\textnormal{d} r^{2}} + \frac{1}{r}\frac{\textnormal{d}}{\textnormal{d} r} - \frac{(mn)^{2}}{r^{2}}\right)\mathbf{u}_{n} = \mathbf{M}_{1}\mathbf{u}_{n}
\end{equation}
where $\mathbf{v}_{n} =(\mathbf{u}_{n},\partial_{r}\mathbf{u}_{n})^{T}$. By decomposing $\mathbf{u}_{n}$ by the generalised eigenvectors $\{\hat{U}_{0},\hat{U}_{1}\}$ of $\mathbf{M}_{1}$ such that $\mathbf{u}_{n}(r) = v_{1}^{(n)}(r)\,\hat{U}_{0} + v_{2}^{(n)}(r)\,\hat{U}_{1}$, we arrive at the following radial ODEs
\begin{equation}\label{Ntup:SH;lin}
    	\left(\frac{\textnormal{d}^{2}}{\textnormal{d} r^{2}} + \frac{1}{r}\frac{\textnormal{d}}{\textnormal{d} r} + \left(1 - \frac{(mn)^{2}}{r^{2}}\right)\right)v_{1}^{(n)} = v_{2}^{(n)}, \qquad  \left(\frac{\textnormal{d}^{2}}{\textnormal{d} r^{2}} + \frac{1}{r}\frac{\textnormal{d}}{\textnormal{d} r} + \left(1 - \frac{(mn)^{2}}{r^{2}}\right)\right)v_{2}^{(n)} = 0.
\end{equation}
The second ODE in \eqref{Ntup:SH;lin} is just the $(mn)^\mathrm{th}$-order Bessel equation, which has solutions of the form
\begin{equation}
	v_{2}^{(n)}(r) = 2d_{2}^{(n)}  J_{mn}(r) + 2d_{4}^{(n)} Y_{mn}(r),
\end{equation}
where the factor of $2$ is included for future simplicity. Solutions to the first ODE in \eqref{Ntup:SH;lin} can be written as 
\begin{equation}
    	v_{1}^{(n)}(r) = d_{1}^{(n)} J_{mn}(r) + d_{3}^{(n)} Y_{mn}(r) + v_{n}^{p}(r),
\end{equation}
where $v_{n}^{p}(r)$ is the particular solution of \eqref{Ntup:SH;lin} for $v_{2}^{(n)}\neq0$. Defining $Z_{\nu}(r):=\eta_{1}\; J_{\nu}(r) + \eta_{2} \; Y_{\nu}(r)$ leads to the relationship 
\begin{equation}
    \left(\frac{\textnormal{d}^{2}}{\textnormal{d} r^{2}} + \frac{1}{r}\frac{\textnormal{d}}{\textnormal{d} r} + \left(1 - \frac{\alpha^{2}}{r^{2}}\right)\right)\left[r \; Z_{\alpha+1}(r)\right] = 2 \;Z_{\alpha}(r),
\end{equation}
for any $\nu\in\mathbb{N}_{0}$, $\eta_{2},\eta_2\in\mathbb{R}$. So, the particular solution $v_n^p$ must be 
\begin{equation}
	v_{n}^{p}(r) = d_{2}^{(n)}r J_{mn+1}(r) + d_{4}^{(n)}r Y_{mn+1}(r),
\end{equation}
and so, after rescaling by a factor of $\sqrt{\frac{\pi}{2}}$, we see that the general solution to $\frac{\textnormal{d}}{\textnormal{d} r}\mathbf{v}_{n} = \mathcal{A}_{n}(r)\mathbf{v}_{n}$ is as stated in \eqref{Ntup:SH;lin,soln}. Hence, the full $4(N+1)$-dimensional linear system \eqref{Vlinear} has solutions of the form
\begin{equation}
	\begin{split}
    \mathbf{V}(r) &= \sum_{n=0}^{N} \left\{\sum_{i=1}^{4} d_{i}^{(n)}\mathscr{V}_{i}^{(n)}(r)\right\},\quad \textnormal{where} \quad
    \left[\mathscr{V}^{(n)}_{i}(r)\right]_{k}=\left\{\begin{array}{cc}
        \left[\mathbf{V}_{i}^{(n)}(r)\right]_{k-4n}, & 4n < k \leq 4(n+1), \\
        0, & \textnormal{otherwise}.
    \end{array}\right. 
    \end{split}
\end{equation}
Here we have used the notation $[\mathbf{v}]_{n}$ to denote the $n^\mathrm{th}$ element of a vector $\mathbf{v}$. Furthermore, from the asymptotic forms in Table~\ref{table:Bessel}, $\mathscr{V}_{1,2}^{(n)}(r)$ remain bounded and $\mathscr{V}_{3,4}^{(n)}(r)$ blow up as $r\to0$, for each $n\in[0,N]$. Hence, we expect the set of solutions to \eqref{R-D:U;vec} that remain bounded as $r\to0$ to form a $2(N+1)$ dimensional manifold in $\mathbb{R}^{4(N+1)}$ for each fixed $r>0$. Let us denote by $\mathcal{P}^{cu}_{-}(r_{0})$ the projection onto the subspace of $\R^{4(N+1)}$ spanned by $\{\mathscr{V}_{1}^{(n)}(r_{0}),\mathscr{V}_{2}^{(n)}(r_{0})\}_{n=0}^{N}$ with null space spanned by $\{\mathscr{V}_{3}^{(n)}(r_{0}),\mathscr{V}_{4}^{(n)}(r_{0})\}_{n=0}^{N}$. 

In what follows we will use the Landau symbol $\mathcal{O}_{r_{0}}(\cdot)$ with the meaning of the standard Landau symbol $\mathcal{O}(\cdot)$ except that the bounding constants may depend on the value of $r_{0}$.  

\begin{table} 
\centering
\begin{tabular}{|c|c|c|c|}
\hline
     Function & \multicolumn{2}{c|}{as $r\to0$} & as $r\to\infty$ \\
     \hline
     & $\nu=0$ & $\nu\in\mathbb{N}$ & $\nu\in\mathbb{N}_{0}$\\
     \hline
     $J_{\nu}(r)$ & $\mathcal{O}\left(1\right)$ & $\mathcal{O}\left(r^{\nu}\right)$ & $\sqrt{\frac{2}{\pi r}} \cos\left(r - \frac{\nu \pi}{2} - \frac{\pi}{4}\right) + \mathcal{O}\left(r^{-\frac{3}{2}}\right)$\\
     $r J_{\nu+1}(r)$ & $\mathcal{O}\left(r^{2}\right)$ & $\mathcal{O}\left(r^{\nu+2}\right)$ & $\sqrt{\frac{2 r}{\pi }} \sin\left(r - \frac{\nu \pi}{2} - \frac{\pi}{4}\right) + \mathcal{O}\left(r^{-\frac{1}{2}}\right)$\\
     $\frac{\textnormal{d}}{\textnormal{d}r} J_{\nu}(r)$ & $\mathcal{O}\left( r \right)$ & $\mathcal{O}\left(r^{\nu-1}\right)$ & $-\sqrt{\frac{2}{\pi r}} \sin\left(r - \frac{\nu \pi}{2} - \frac{\pi}{4}\right) + \mathcal{O}\left(r^{-\frac{3}{2}}\right)$\\
     $\frac{\textnormal{d}}{\textnormal{d}r} [ r J_{\nu+1}(r)]$ & $\mathcal{O}\left( r \right)$ & $\mathcal{O}\left(r^{\nu+1}\right)$ & $\sqrt{\frac{2 r}{\pi }} \cos\left(r - \frac{\nu \pi}{2} - \frac{\pi}{4}\right) + \mathcal{O}\left(r^{-\frac{1}{2}}\right)$\\
     $Y_{\nu}(r)$ & $\mathcal{O}\left(\log(r)\right)$ & $\mathcal{O}\left(r^{-\nu}\right)$ & $\sqrt{\frac{2}{\pi r}} \sin\left(r - \frac{\nu \pi}{2} - \frac{\pi}{4}\right) + \mathcal{O}\left(r^{-\frac{3}{2}}\right)$\\
     $r Y_{\nu+1}(r)$ & $\mathcal{O}\left(1\right)$ & $\mathcal{O}\left(r^{-\nu}\right)$ & $-\sqrt{\frac{2 r}{\pi }} \cos\left(r - \frac{\nu \pi}{2} - \frac{\pi}{4}\right) + \mathcal{O}\left(r^{-\frac{1}{2}}\right)$\\
     $\frac{\textnormal{d}}{\textnormal{d}r} Y_{\nu}(r)$ & $\mathcal{O}\left(r^{-1}\right)$ & $\mathcal{O}\left(r^{-\nu-1}\right)$ & $\sqrt{\frac{2}{\pi r}} \cos\left(r - \frac{\nu \pi}{2} - \frac{\pi}{4}\right) + \mathcal{O}\left(r^{-\frac{3}{2}}\right)$\\
     $\frac{\textnormal{d}}{\textnormal{d}r} [ r Y_{\nu+1}(r)]$ & $\mathcal{O}\left(1\right)$ & $\mathcal{O}\left(r^{-\nu-1}\right)$ & $\sqrt{\frac{2 r}{\pi }} \sin\left(r - \frac{\nu \pi}{2} - \frac{\pi}{4}\right) + \mathcal{O}\left(r^{-\frac{1}{2}}\right)$\\
     \hline
\end{tabular}
\caption{Orders and expansions for $\nu$'th-order Bessel functions of the first and second kind as $r\to0$ and $r\to\infty$, respectively, where $\nu\in\mathbb{N}_{0}$; see \cite[(9.1.10),(9.1.11), and \S9.2]{abramowitz1972handbook}.}
\label{table:Bessel}
\end{table} 

\begin{lem}\label{Lemma:Ntup;Core} 
 Fix $m,N\in\mathbb{N}$. For each fixed $r_{0}>0$, there are constants $\delta_{1},\delta_{2}>0$ such that the set $\mathcal{W}^{cu}_{-}(\mu)$ of solutions $\mathbf{U}(r)$ of \eqref{R-D:U;vec} for which $\sup_{0\leq r\leq r_{0}}\|\mathbf{U}(r)\|<\delta_{1}$ is, for $|\mu|<\delta_{1}$, a smooth $2(N+1)$ dimensional manifold. Furthermore, each $\mathbf{U}(r_{0})\in\mathcal{W}^{cu}_{-}(\mu)$ with $|\mathcal{P}^{cu}_{-}(r_{0})\mathbf{U}(r_{0})|<\delta_{2}$ can be written uniquely as 
\begin{equation}\label{U:Core;Ntup}
	\begin{split}
 		\mathbf{U}(r_{0}) &= \sum_{n=0}^{N} \{d_{1}^{(n)}\mathscr{V}_{1}^{(n)}(r_{0}) + d_{2}^{(n)}\mathscr{V}_{2}^{(n)}(r_{0}) + \mathscr{V}_{3}^{(n)}(r_{0})\mathcal{O}_{r_0}(|\mu||\mathbf{d}| + |\mathbf{d}|^{2})\\
 &\qquad \qquad  + \mathscr{V}_{4}^{(n)}(r_{0})[\nu Q^{m}_{n}(\mathbf{d}_{1}) + \mathcal{O}_{r_0}(|\mu||\mathbf{d}| + |\mathbf{d}_{2}|^{2} + |\mathbf{d}_{1}|^{3})]\},
	\end{split}
\end{equation}
where $\nu := \frac{1}{2}\sqrt{\frac{\pi}{6}}\big\langle \hat{U}_{1}^{*}, \mathbf{Q}(\hat{U}_{0},\hat{U}_{0})\big\rangle_{2}$, $\mathbf{d}_{j}:=\left(d_{j}^{(0)}, d_{j}^{(1)},\dots,d_{i}^{(N)}\right)$, $\mathbf{d}:=\left(\mathbf{d}_{1},\mathbf{d}_{2}\right)\in\mathbb{R}^{2(N+1)}$ with $|\mathbf{d}|<\delta_{2}$. Furthermore, the right-hand side of \eqref{U:Core;Ntup} depends smoothly on $(\mathbf{d},\mu)$, and the nonlinear functions $Q^{m}_{n}(\mathbf{d}_{1})$ are defined as
\begin{equation}\label{Psi:d1;defn}
	Q^{m}_{n}(\mathbf{d}_{1}) := 2\sum_{j=1}^{N-n} \cos\left(\frac{m\pi(n-j)}{3}\right) d^{(j)}_{1} d^{(n+j)}_{1} + \sum_{j=0}^{n} \cos\left(\frac{m\pi(n-2j)}{3}\right) d^{(j)}_{1} d^{(n-j)}_{1}.
\end{equation}
\end{lem}

\begin{proof}
This statement is proven in a similar way to \cite[Lemma~1]{lloyd2009localized}. We first note that the linear adjoint problem $\frac{\textnormal{d}}{\textnormal{d}r}\mathbf{W} = -\mathcal{A}^{T}(r)\mathbf{W}$ decouples into $\frac{\textnormal{d}}{\textnormal{d}r}\mathbf{W}_{n} = -\mathcal{A}^{T}_{n}(r)\mathbf{W}_{n}$ for each $n\in[0,N]$, which have independent solutions of the form
\begin{align}
\begin{split}
    \mathbf{W}_{1}^{(n)} &= \frac{\sqrt{2\pi}}{4}\begin{pmatrix}
    r\frac{\textnormal{d}}{\textnormal{d} r}\big[r Y_{mn +1}(r)\big]\hat{U}_{1}^{*} + 2 r\frac{\textnormal{d}}{\textnormal{d} r}\big[Y_{mn}(r)\big]\hat{U}_{0}^{*} \\ -r^2 Y_{mn +1 }(r) \hat{U}_{1}^{*} - 2 r Y_{mn}(r)\hat{U}_{0}^{*} 
    \end{pmatrix}, \qquad \mathbf{W}_{2}^{(n)} = \frac{\sqrt{2\pi}}{4}\begin{pmatrix}
    r\frac{\textnormal{d}}{\textnormal{d} r}\big[Y_{mn}(r)\big]\hat{U}_{1}^{*} \\ -r Y_{mn}(r)\hat{U}_{1}^{*}
    \end{pmatrix}, \\
    \mathbf{W}_{3}^{(n)} &= \frac{\sqrt{2\pi}}{4}\begin{pmatrix}
     -r\frac{\textnormal{d}}{\textnormal{d} r}\big[r J_{mn +1}(r)\big]\hat{U}_{1}^{*} - 2 r\frac{\textnormal{d}}{\textnormal{d} r}\big[J_{mn}(r)\big]\hat{U}_{0}^{*} \\ r^2 J_{mn +1 }(r) \hat{U}_{1}^{*} + 2  r J_{mn}(r)\hat{U}_{0}^{*}
    \end{pmatrix}, \qquad \mathbf{W}_{4}^{(n)} =  \frac{\sqrt{2\pi}}{4}\begin{pmatrix}
    -r\frac{\textnormal{d}}{\textnormal{d} r}\big[J_{mn}(r)\big]\hat{U}_{1}^{*}\\ r J_{mn}(r)\hat{U}_{1}^{*}
    \end{pmatrix}.
    \end{split}\label{Ntup:R-D;adj,soln}
\end{align}
To see this, we note that $\frac{\textnormal{d}}{\textnormal{d}r}\mathbf{W}_{n} = -\mathcal{A}^{T}_{n}(r)\mathbf{W}_{n}$ can be written as
\begin{equation}
    	\left(\frac{\textnormal{d}^{2}}{\textnormal{d} r^{2}} - \frac{1}{r}\frac{\textnormal{d}}{\textnormal{d} r} + \left(\frac{1-(mn)^{2}}{r^{2}}\right)\right)\mathbf{w}_{n} = 
    	\mathbf{M}_{1}^{T}\mathbf{w}_{n}
\end{equation}
where $\mathbf{W}_{n}=\left(-r\partial_{r}(\frac{1}{r}\mathbf{w}_{n}),\mathbf{w}_{n}\right)^{T}$. We decompose $\mathbf{w}_{n}$ onto the generalised eigenvectors $\{\hat{U}_{0}^{*}, \hat{U}_{1}^{*}\}$ of $\mathbf{M}_{1}^{T}$ such that $\mathbf{w}_{n}(r) = w_{1}^{(n)}(r)\,\hat{U}_{1}^{*} + w_{2}^{(n)}(r)\,\hat{U}_{0}^{*}$, and we arrive at the following ODEs
\begin{align}
    &\left(\frac{\textnormal{d}^{2}}{\textnormal{d} r^{2}} + \frac{1}{r}\frac{\textnormal{d}}{\textnormal{d} r} + \left(1 - \frac{(mn)^{2}}{r^{2}}\right)\right)\left[\frac{1}{r}\;w_{2}^{(n)}\right] = \frac{1}{r}\;w_{1}^{(n)},&\qquad 
    &\left(\frac{\textnormal{d}^{2}}{\textnormal{d} r^{2}} + \frac{1}{r}\frac{\textnormal{d}}{\textnormal{d} r} + \left(1 - \frac{(mn)^{2}}{r^{2}}\right)\right)\left[\frac{1}{r}\;w_{1}^{(n)}\right] = 0.&\nonumber
\end{align}
Hence, we see that $\frac{1}{r}\;w_{1}^{(n)}$ and $\frac{1}{r}\;w_{2}^{(n)}$ solve \eqref{Ntup:SH;lin}, and so solutions take the form stated in \eqref{Ntup:R-D;adj,soln}. We choose the particular ordering and scaling of our adjoint solutions $\mathbf{W}_{j}^{(n)}(r)$ such that the relation
\begin{equation}
    \langle \mathbf{W}^{(n)}_{i}(r), \mathbf{V}^{(n)}_{j}(r)\rangle_{4}  = \delta_{i,j}
\end{equation}
holds for all $r\in\mathbb{R}$, $n\in[0,N]$, $i,j \in\{1,2,3,4\}$, where $\langle \cdot,\cdot\rangle_{k}$ denotes the Euclidean inner product for $\mathbb{R}^{k}$. We introduce $\mathscr{W}_{i}^{(n)}(r)\in\mathbb{R}^{4(N+1)}$, for $n\in[0,N]$, $i\in\{1,2,3,4\}$, defined as
\begin{equation}
        \left[\mathscr{W}^{(n)}_{i}\right]_{k}=\left\{\begin{array}{cc}
        \left[\mathbf{W}_{i}^{(n)}\right]_{k-4n}, & 4n < k \leq 4(n+1), \\
        0, & \textnormal{otherwise},
    \end{array}\right.\nonumber
\end{equation}
such that each $\mathscr{W}^{(n)}_{j}$ satisfies the full $4(N+1)$-dimensional  linear adjoint problem $\frac{\textnormal{d}}{\textnormal{d}r}\mathbf{W} = -\mathcal{A}^{T}(r)\mathbf{W}$, for all $n\in[0,N]$, $i\in\{1,2,3,4\}$. It is clear that, by definition, 
\begin{equation}
	\langle \mathscr{W}_{i}^{(n)}, \mathscr{V}_{j}^{(n)}\rangle_{4(N+1)} = \langle \mathbf{W}_{i}^{(n)}, \mathbf{V}_{j}^{(n)}\rangle_{4} = \delta_{i,j},\nonumber
\end{equation}
for all $n\in[0,N]$, $i,j\in\{1,2,3,4\}$, and is independent of $r$. For a given $\mathbf{d}=(\mathbf{d}_{1},\mathbf{d}_{2})\in\mathbb{R}^{2(N+1)}$, where $\mathbf{d}_{i}:=(d_{i}^{(0)}, \dots, d_{i}^{(N)})^{T}\in\mathbb{R}^{(N+1)}$ for each $i=1,2$, we consider the fixed-point equation
\begin{equation}\label{var:U}
	\begin{split}
    	\mathbf{U}(r) &= \sum_{n=0}^{N}\left\{\sum_{j=1}^{2}  d^{(n)}_{j} \mathscr{V}_{j}^{(n)}(r) + \sum_{j=1}^{2} \mathscr{V}_{j}^{(n)}(r)\int_{r_{0}}^{r} \langle \mathscr{W}_{j}^{(n)}(s), \mathbf{F}(\mathbf{U}[s];\mu) \rangle_{4(N+1)} \,\textnormal{d} s \right. \\
    & \qquad  \left. + \sum_{j=3}^{4} \mathscr{V}_{j}^{(n)}(r)\int_{0}^{r} \langle \mathscr{W}_{j}^{(n)}(s), \mathbf{F}(\mathbf{U}[s];\mu) \rangle_{4(N+1)} \,\textnormal{d} s \right\}, \nonumber\\
     &= \sum_{n=0}^{N}\left\{\sum_{j=1}^{2}  d^{(n)}_{j} \mathscr{V}_{j}^{(n)}(r) + \sum_{j=1}^{2} \mathscr{V}_{j}^{(n)}(r)\int_{r_{0}}^{r} \langle \mathbf{W}_{j}^{(n)}(s), \mathbf{F}_{n}(\mathbf{U}[s];\mu) \rangle_{4} \,\textnormal{d} s \right. \nonumber\\
    & \qquad \left. + \sum_{j=3}^{4} \mathscr{V}_{j}^{(n)}(r)\int_{0}^{r}  \langle \mathbf{W}_{j}^{(n)}(s), \mathbf{F}_{n}(\mathbf{U}[s];\mu) \rangle_{4} \,\textnormal{d} s \right\}, 
    	\end{split}
\end{equation}
on $C([0,r_{0}],\mathbb{R}^{4(N+1)})$, the space of continuous functions from $[0,r_0]$ to $\R^{4(N+1)}$.

We first check that any solution $\mathbf{U}\in C([0,r_{0}],\mathbb{R}^{4(N+1)})$ of \eqref{var:U} gives a solution of \eqref{R-D:U;vec} that is bounded on $r\in[0,r_{0}]$. In the limit as $r\to0$, we see from Table \ref{table:Bessel} that the term $[\mathbf{W}_{j}^{(n)}(s)]_{2}$, denoting the second block element of $\mathbf{W}_{j}^{(n)}(s)$ in \eqref{Ntup:R-D;adj,soln}, is bounded by $s^{(mn+3)}$ for $j=3$, and $s^{(mn+1)}$ for $j=4$. Hence, the integrals multiplying the unbounded solutions $\mathscr{V}_{j}^{(n)}(r)$ are bounded by $r^{(mn+4)}$ for $j=3$ and $r^{(mn+2)}$ for $j=4$. Since $|\mathscr{V}_{j}^{(n)}(r)| = \mathcal{O}\left(r^{-(mn + 1)}\right)$ for both $j=3,4$ as $r\to0$, we see that the right-hand side of \eqref{var:U} is continuously differentiable on its domain whenever $\mathbf{U}\in C([0,r_{0}],\mathbb{R}^{4(N+1)})$. Since \eqref{var:U} is a specific case of a variation-of-constants formula, it is straightforward to check that $\mathbf{U}(r)$ satisfies \eqref{R-D:U;vec}.

We now need to check that any bounded solution $\mathbf{U}(r)\in C([0,r_{0}],\mathbb{R}^{4(N+1)})$ of \eqref{R-D:U;vec} satisfies \eqref{var:U}. Taking a variation-of-constants formula for small solutions of \eqref{R-D:U;vec}, we find
\begin{equation}
	\begin{split}
    		\mathbf{U}(r) &= \sum_{n=0}^{N}\left\{\sum_{j=1}^{4}  d^{(n)}_{j} \mathscr{V}_{j}^{(n)}(r) + \sum_{j=1}^{2} \mathscr{V}_{j}^{(n)}(r)\int_{r_{0}}^{r} \langle \mathscr{W}_{j}^{(n)}(s), \mathbf{F}(\mathbf{U}[s];\mu) \rangle_{4(N+1)} \,\textnormal{d} s \right. \\
    		& \qquad \left. + \sum_{j=3}^{4} \mathscr{V}_{j}^{(n)}(r)\int_{0}^{r} \langle \mathscr{W}_{j}^{(n)}(s), \mathbf{F}(\mathbf{U}[s];\mu) \rangle_{4(N+1)} \,\textnormal{d} s \right\},
	\end{split}
\end{equation}
for an appropriate $\widetilde{\mathbf{d}}:=(\mathbf{d}_{1}, \mathbf{d}_{2}, \mathbf{d}_{3}, \mathbf{d}_{4})\in\mathbb{R}^{4(N+1)}$. It is clear that, for each $n\in[0,N]$, the terms $d_{3}^{(n)}\mathscr{V}_{3}^{(n)}(r) + d_{4}^{(n)}\mathscr{V}_{4}^{(n)}(r)$ are unbounded as $r\to0$ unless $d_{3}^{(n)}=d_{4}^{(n)}=0$, which proves the assertion. Finally, we solve \eqref{var:U} by applying the uniform contraction mapping principle for sufficiently small $\mathbf{d}=(\mathbf{d}_{1},\mathbf{d}_{2})$ and $\mu$. Evaluating \eqref{var:U} at $r=r_{0}$, we arrive at
\begin{equation}\label{var:U;r0}
    \mathbf{U}(r_{0}) = \sum_{n=0}^{N}\left\{\sum_{j=1}^{2}  d^{(n)}_{j} \mathscr{V}_{j}^{(n)}(r_{0}) + \sum_{j=3}^{4} \mathscr{V}_{j}^{(n)}(r_{0})\int_{0}^{r_{0}}  \langle \mathbf{W}_{j}^{(n)}(s), \mathbf{F}_{n}(\mathbf{U}[s];\mu) \rangle_{4} \,\textnormal{d} s \right\}. 
\end{equation}
Then, we introduce
\begin{equation}
    c^{(n)}_{j}\left(\mathbf{d}_{1},\mathbf{d}_{2};\mu\right) := \int_{0}^{r_{0}} \langle \mathbf{W}_{j}^{(n)}(s), \mathbf{F}_{n}(\mathbf{U}[s];\mu) \rangle_{4} \,\textnormal{d} s,\label{cn:defn}
\end{equation}
for $j=3,4$ so that we can write our small-amplitude core solution as
\begin{equation}
    	\mathbf{U}(r_{0}) = \sum_{n=0}^{N}\left\{d^{(n)}_{1} \mathscr{V}_{1}^{(n)}(r_{0}) +  d^{(n)}_{2} \mathscr{V}_{2}^{(n)}(r_{0}) +  c^{(n)}_{3}\left(\mathbf{d}_{1}, \mathbf{d}_{2}; \mu\right) \mathscr{V}_{3}^{(n)}(r_{0}) + c^{(n)}_{4}\left(\mathbf{d}_{1}, \mathbf{d}_{2}; \mu\right) \mathscr{V}_{4}^{(n)}(r_{0}) \right\}. 
\end{equation}
In order to arrive at \eqref{U:Core;Ntup}, we apply a Taylor expansion to \eqref{cn:defn} about $|\mathbf{d}_{1}| = |\mathbf{d}_{2}|=\mu=0$ and find 
\begin{equation}
	\begin{split}
    		c^{(n)}_{3}\left(\mathbf{d}_{1},\mathbf{d}_{2};\mu\right) &= \mathcal{O}_{r_{0}}\left(|\mu||\mathbf{d}| + |\mathbf{d}|^{2}\right), \\
    		c^{(n)}_{4}\left(\mathbf{d}_{1},\mathbf{d}_{2};\mu\right) &= \sum_{i + j = n}\nu_{|i|,|j|,n} \, d_{1}^{(|i|)}d_{1}^{(|j|)} + \mathcal{O}_{r_{0}}\left(|\mu||\mathbf{d}| + |\mathbf{d}_{2}|^{2} + |\mathbf{d}_{1}|^{3}\right), \\
    	\end{split}
\end{equation}
where
\begin{equation}
    \nu_{i,j,n} :=  \frac{\left(2\pi\right)^{\frac{3}{2}}}{16} \langle \hat{U}_{1}^{*}, \mathbf{Q}(\hat{U}_{0},\hat{U}_{0})\rangle_{2}\int_{0}^{r_{0}}  \left[s \, J_{mn}(s)\right] \left[ J_{mi}(s)\,J_{mj}(s)\right] \,\drm s.
\end{equation}
Notably, $\nu_{|i|,|j|,n}$ is invariant under permutations of its indices and, for the restriction $ i + j = n$, one can always find some $a,b\in\mathbb{N}_{0}$ such that $\nu_{|i|,|j|,n} = \nu_{a,b,a+b}$. Computing the explicit value for $\nu_{a,b,a+b}$, we find that
\begin{equation}
	\begin{split}
    		\nu_{a,b,a+b} &= \nu \frac{\pi \sqrt{3}}{2} \int_{0}^{r_{0}} \left\{s J_{ma}(s) J_{mb}(s) J_{m(a+b)}(s)\right\}\textnormal{d}s,\\
    		&= \nu \frac{\pi \sqrt{3}}{2}\left[\int_{0}^{\infty} \left\{s J_{ma}(s) J_{mb}(s) J_{m(a+b)}(s)\right\}\textnormal{d}s + \mathcal{O}\left(r_{0}^{-\frac{1}{2}}\right)\right],\\
    		&=\nu \frac{\pi \sqrt{3}}{2} \left[\frac{\cos\left(\frac{m\pi}{3}(a-b)\right)}{\pi\sin\left(\frac{\pi}{3}\right)} + \mathcal{O}\left(r_{0}^{-\frac{1}{2}}\right)\right] = \nu\cos\left(\frac{m\pi}{3}(a-b)\right) + \mathcal{O}\left(r_{0}^{-\frac{1}{2}}\right),
	\end{split}
\end{equation}
where $\nu:= \frac{1}{2}\sqrt{\frac{\pi}{6}} \langle \hat{U}_{1}^{*}, \mathbf{Q}(\hat{U}_{0},\hat{U}_{0})\rangle_{2}$, and we have used \cite[Lemma~2.1]{mccalla2013spots} to obtain the $\mathcal{O}\left(r_{0}^{-\frac{1}{2}}\right)$ remainder term and \cite[(2.4)]{Gervois1984Bessel} to evaluate the final integral. We note that 
\begin{equation}
    \sum_{i + j = n}\nu_{|i|,|j|,n} \, d_{1}^{(|i|)}d_{1}^{(|j|)} = 2\sum_{j=1}^{N-n} \nu_{j, n, n+j} \, d_{1}^{(j)}d_{1}^{(n+j)} + \sum_{j=0}^{n} \nu_{j, n-j, n} \, d_{1}^{(j)}d_{1}^{(n-j)},
\end{equation} 
giving
\begin{equation}
    c^{(n)}_{4}\left(\mathbf{d}_{1},\mathbf{d}_{2};\mu\right) = \left[\nu + \mathcal{O}\left(r_{0}^{-\frac{1}{2}}\right)\right]Q_{n}^{m}(\mathbf{d}_{1}) + \mathcal{O}_{r_{0}}\left(|\mu||\mathbf{d}| + |\mathbf{d}_{2}|^{2} + |\mathbf{d}_{1}|^{3}\right), 
\end{equation}
where $Q^m_n(\mathbf{d}_1)$ is as defined in the statement of the lemma. This completes the proof.
\end{proof} 

\begin{rmk}
    If we replace the quadratic term $\mathbf{Q}(\mathbf{u},\mathbf{u})$ in \eqref{eqn:R-D} by $\mathbf{Q}(\partial_{r}\mathbf{u},\partial_{r}\mathbf{u}) + \frac{1}{r^2}\mathbf{Q}(\partial_{\theta}\mathbf{u},\partial_{\theta}\mathbf{u})$, which is analogous to replacing $|\mathbf{u}|^2$ terms by $|\nabla\mathbf{u}|^2$ in the nonlinearity, Lemma~\ref{Lemma:Ntup;Core} remains unchanged other than a rescaling of $\nu\mapsto\frac{1}{2}\nu$.
\end{rmk}

We conclude this section by noting the following. For sufficiently large $r_{0}$, we can use Table~\ref{table:Bessel} to write \eqref{U:Core;Ntup} in terms of its elements $\mathbf{u}_{n}(r_{0}),\mathbf{v}_{n}(r_{0})\in\mathbb{R}^{2}$ for each $n\in[0,N]$, where 
\begin{align}
    \begin{split}
    \mathbf{u}_{n}(r_{0}) &= r_{0}^{-\frac{1}{2}}\left[d_{2}^{(n)}r_{0}\left[1 + \mathcal{O}\left(r_{0}^{-1}\right)\right] \sin(y_{n}) + d_{1}^{(n)}\left[1 + \mathcal{O}\left(r_{0}^{-1}\right)\right]\cos(y_{n}) + \mathcal{O}_{r_{0}}(|\mu||\mathbf{d}| + |\mathbf{d}|^{2})\right]\hat{U}_{0}\\
    &\quad + 2 r_{0}^{-\frac{1}{2}}\left[[\nu + \mathcal{O}(r_{0}^{-\frac{1}{2}})]Q_{n}^{m}(\mathbf{d}_{1}) \sin(y_{n}) + d_{2}^{(n)}\left[1 + \mathcal{O}\left(r_{0}^{-1}\right)\right]\cos(y_{n}) + \mathcal{O}_{r_{0}}(|\mu||\mathbf{d}| + |\mathbf{d}_{2}|^{2} + |\mathbf{d}_{1}|^{3})\right]\hat{U}_{1},\\
    \mathbf{v}_{n}(r_{0}) &= r_{0}^{-\frac{1}{2}}\left[d_{2}^{(n)}r_{0}\left[1 + \mathcal{O}\left(r_{0}^{-1}\right)\right] \cos(y_{n}) - d_{1}^{(n)}\left[1 + \mathcal{O}\left(r_{0}^{-1}\right)\right]\sin(y_{n}) + \mathcal{O}_{r_{0}}(|\mu||\mathbf{d}| + |\mathbf{d}|^{2})\right]\hat{U}_{0}\\
    &\quad + 2 r_{0}^{-\frac{1}{2}}\left[[\nu + \mathcal{O}(r_{0}^{-\frac{1}{2}})]Q_{n}^{m}(\mathbf{d}_{1}) \cos(y_{n}) -d_{2}^{(n)}\left[1 + \mathcal{O}\left(r_{0}^{-1}\right)\right]\sin(y_{n}) + \mathcal{O}_{r_{0}}(|\mu||\mathbf{d}| + |\mathbf{d}_{2}|^{2} + |\mathbf{d}_{1}|^{3})\right]\hat{U}_{1},
    \end{split}\label{Core:un}
\end{align} 
for each $n\in[0,N]$. Here we have written $y_{n}:=r_{0} - \frac{mn\pi}{2} - \frac{\pi}{4}$, while the $\mathcal{O}_{r_{0}}(\cdot)$ remainders capture the higher order terms when $|\mathbf{d}|$ and $|\mu|$ are taken to be small.

\subsection{The Far-Field Manifold}\label{subsec:Far}

We now turn to characterising the far-field manifold, $\mathcal{W}_{+}^{s}(\mu)$. To this end, we will introduce the variable $\sigma(r)\geq0$ to take the place of the $\frac{1}{r}$-terms in \eqref{R-D:U;vec}. The result is the extended autonomous system
\begin{equation}\label{R-D:Farf}
	\begin{split}
    		\frac{\textnormal{d}}{\textnormal{d} r}\mathbf{u}_{n} &= \mathcal{A}_{\infty}\mathbf{u}_{n} + \widetilde{\mathbf{F}}_{n}(\mathbf{U}; \mu, \sigma),\qquad \forall n\in[0,N], \\
    		\frac{\textnormal{d}}{\textnormal{d} r}\sigma &= -\sigma^{2},
	\end{split}
\end{equation}
with the property that $\sigma(r)= r^{-1}$ is an invariant manifold of \eqref{R-D:Farf}, and by definition this invariant manifold recovers the non-autonomous system \eqref{R-D:U;vec}. To construct the far-field manifold $\mathcal{W}^{s}_{+}(\mu)$ we find the set of all small-amplitude solutions $(\mathbf{U},\sigma)(r)$ to \eqref{R-D:Farf} such that $\mathbf{U}(r)$ decays exponentially as $r\to\infty$.  Following this, we evaluate at $\sigma(r_0) = r_0^{-1}$, thus restricting to the invariant subspace $\sigma(r) = r^{-1}$, such that $\mathbf{U}(r)$ is an exponentially decaying solution of \eqref{R-D:U;vec}.

Before attempting to find exponentially decaying solutions to \eqref{R-D:Farf}, it is convenient to transform the system into the normal form for a Hamilton--Hopf bifurcation. We define complex amplitudes $\widetilde{\mathbf{A}}:=[\widetilde{A}_{n}]_{n=0}^{N}$, $\widetilde{\mathbf{B}}:=[\widetilde{B}_{n}]_{n=0}^{N}$, which satisfy the following relations
\begin{equation}
    \begin{split}
    \mathbf{u}_{n} &= (\widetilde{A}_{n} + \overline{\widetilde{A}}_{n})\hat{U}_{0} + 2\textnormal{i}(\widetilde{B}_{n} - \overline{\widetilde{B}}_{n})\hat{U}_{1} , \qquad \mathbf{v}_{n} = \big[\textnormal{i}(\widetilde{A}_{n} - \overline{\widetilde{A}}_{n}) + (\widetilde{B}_{n} + \overline{\widetilde{B}}_{n})\big]\hat{U}_{0} - 2(\widetilde{B}_{n} + \overline{\widetilde{B}}_{n})\hat{U}_{1},\\
    \widetilde{A}_{n} &= \frac{1}{2}\big\langle \hat{U}_{0}^{*}, \mathbf{u}_{n} - \textnormal{i}\mathbf{v}_{n}\big\rangle_{2} - \frac{\textnormal{i}}{4}\big\langle \hat{U}_{1}^{*}, \mathbf{v}_{n}\big\rangle_{2}, \qquad \widetilde{B}_{n} = -\frac{\textnormal{i}}{4}\big\langle \hat{U}_{1}^{*}, \mathbf{u}_{n} - \textnormal{i}\mathbf{v}_{n}\big\rangle_{2},
    \end{split}
    \label{R-D:Transformation;Farfield}
\end{equation}
for each $n\in[0,N]$, such that \eqref{R-D:Farf} becomes
\begin{equation}
	\begin{split}
   		 \frac{\textnormal{d}}{\textnormal{d} r} \widetilde{\mathbf{A}} &= \textnormal{i}\widetilde{\mathbf{A}} - \frac{\sigma}{2}(\widetilde{\mathbf{A}} - \overline{\widetilde{\mathbf{A}}}) + \widetilde{\mathbf{B}} - \frac{\textnormal{i}}{2}[ \sigma^{2}\mathcal{C}_{N}^{m}[(\widetilde{\mathbf{A}} + \overline{\widetilde{\mathbf{A}}}) + \textnormal{i}(\widetilde{\mathbf{B}} - \overline{\widetilde{\mathbf{B}}})] + \mathcal{F}_{A}(\widetilde{\mathbf{A}},\widetilde{\mathbf{B}})],\\
    		\frac{\textnormal{d}}{\textnormal{d} r} \widetilde{\mathbf{B}} &= \textnormal{i}\widetilde{\mathbf{B}} - \frac{\sigma}{2}(\widetilde{\mathbf{B}} + \overline{\widetilde{\mathbf{B}}})  - \frac{1}{2}[ \sigma^{2}\mathcal{C}_{N}^{m}\textnormal{i}(\widetilde{\mathbf{B}} - \overline{\widetilde{\mathbf{B}}}) + \mathcal{F}_{B}(\widetilde{\mathbf{A}},\widetilde{\mathbf{B}})],\label{amp:AB;tilde} \\
    \frac{\textnormal{d}}{\textnormal{d} r} \sigma &= -\sigma^{2},
	\end{split}
\end{equation}
where we have defined $\mathcal{C}^{m}_{N}:=\textrm{diag}(0,(m)^{2},\dots,(mN)^{2})$, $\left[\mathcal{F}_{A}\right]_{n} := \big\langle \hat{U}_{0}^{*} + \frac{1}{2}\hat{U}_{1}^{*}, \mathcal{F}_{n}\big\rangle$,  $\left[\mathcal{F}_{B}\right]_{n} := \big\langle \frac{1}{2}\hat{U}_{1}^{*}, \mathcal{F}_{n}\big\rangle$, and 
\begin{equation}
    \mathcal{F}_{n}(\widetilde{\mathbf{A}},\widetilde{\mathbf{B}}) := \bigg[\mu \mathbf{M}_{2}\mathbf{u}_{n} + \sum_{i+j=n} \mathbf{Q}(\mathbf{u}_{|i|},\mathbf{u}_{|j|}) + \sum_{i+j+k=n} \mathbf{C}(\mathbf{u}_{|i|},\mathbf{u}_{|j|},\mathbf{u}_{|k|})\bigg]_{\mathbf{u}_{j} = (\widetilde{A}_{j} + \overline{\widetilde{A}}_{j})\hat{U}_{0} + 2\textnormal{i}(\widetilde{B}_{j} - \overline{\widetilde{B}}_{j})\hat{U}_{1}}\nonumber
\end{equation}
for each $n\in[0,N]$.
\begin{rmk}
In the coordinates $\{\widetilde{A}_{n},\widetilde{B}_{n}\}_{n=0}^{N}$, defined by \eqref{R-D:Transformation;Farfield}, the core manifold \eqref{Core:un} can be expressed as 
\begin{align}
    \begin{split}
    \widetilde{A}_{n}(r_{0}) &= \frac{1}{2} r_{0}^{-\frac{1}{2}}\textnormal{e}^{\textnormal{i}y_{n}}\left[-\textnormal{i} d_{2}^{(n)}r_{0}\left[1 + \mathcal{O}\left(r_{0}^{-1}\right)\right] + d_{1}^{(n)}\left[1 + \mathcal{O}\left(r_{0}^{-1}\right)\right] + \mathcal{O}_{r_{0}}(|\mu||\mathbf{d}| + |\mathbf{d}|^{2})\right],\\
    \widetilde{B}_{n}(r_{0}) &= -\frac{1}{2} r_{0}^{-\frac{1}{2}}\textnormal{e}^{\textnormal{i}y_{n}}\left[[\nu + \mathcal{O}(r_{0}^{-\frac{1}{2}})]Q_{n}^{m}(\mathbf{d}_{1})  + \textnormal{i} d_{2}^{(n)}\left[1 + \mathcal{O}\left(r_{0}^{-1}\right)\right] + \mathcal{O}_{r_{0}}(|\mu||\mathbf{d}| + |\mathbf{d}_{2}|^{2} + |\mathbf{d}_{1}|^{3})\right],
    \end{split}\label{Core:An}
\end{align}
for each $n\in[0,N]$, where we recall that $\nu = \frac{1}{2}\sqrt{\frac{\pi}{6}} \langle \hat{U}_{1}^{*}, \mathbf{Q}(\hat{U}_{0},\hat{U}_{0})\rangle_{2}$ and $y_n = r_0 - \frac{mn\pi}{2} - \frac{\pi}{4}$.
\end{rmk}

We apply nonlinear normal form transformations to \eqref{amp:AB;tilde}, as seen in \cite{scheel2003radially,lloyd2009localized,Scheel2014Grain}, to remove the non-resonant terms from the right-hand side.

\begin{lem}\label{Lemma:Ntup;normal} 
Fix $N\in\mathbb{N}$. Then, for each $n\in[0,N]$, there exists a change of coordinates
\begin{equation}\label{Ntup:Normal;transf}
    \begin{pmatrix}
    A_{n} \\ B_{n}
    \end{pmatrix} := \textnormal{e}^{-\textnormal{i}\phi_{n}(r)}\left[\mathbbm{1} + \mathcal{T}_{n}(\sigma)\right]\begin{pmatrix}
    \widetilde{A}_{n} \\ \widetilde{B}_{n}
    \end{pmatrix} + \mathcal{O}((|\mu| + |\widetilde{\mathbf{A}}| + |\widetilde{\mathbf{B}}|)(|\widetilde{\mathbf{A}}| + |\widetilde{\mathbf{B}}|)),
\end{equation}
such that \eqref{amp:AB;tilde} becomes
\begin{equation}\label{Ntup:NormalForm}
	\begin{split}
    \frac{\textnormal{d}}{\textnormal{d}r} \mathbf{A} &= - \frac{\sigma}{2} \mathbf{A} + \mathbf{B} + \mathbf{R}_{\mathbf{A}}(\mathbf{A}, \mathbf{B},\sigma,\mu),\\
    \frac{\textnormal{d}}{\textnormal{d}r} \mathbf{B} &= -\frac{\sigma}{2} \mathbf{B} + c_{0}\,\mu \mathbf{A} + \mathbf{R}_{\mathbf{B}}(\mathbf{A}, \mathbf{B},\sigma,\mu),\\
    \frac{\textnormal{d}}{\textnormal{d} r} \sigma &= -\sigma^{2},
	\end{split}
\end{equation}
where $c_{0}:= \frac{1}{4}\big\langle \hat{U}_{1}^{*}, -\mathbf{M}_{2}\hat{U}_{0}\big\rangle_{2}$. For each $n\in[0,N]$, the coordinate change \eqref{Ntup:Normal;transf} is polynomial in $(A_{n},B_{n},\sigma)$ and smooth in $\mu$, and $\mathcal{T}_{n}(\sigma) = \mathcal{O}(\sigma)$ is linear and upper triangular for each $\sigma$. The remainder terms satisfy
\begin{equation}
     [\mathbf{R}_{\mathbf{A}/\mathbf{B}}]_{n} = \mathcal{O}([|\mu|^{2} + |\sigma|^{3} + (|\mathbf{A}| + |\mathbf{B}|)^{2}][|\mathbf{A}| + |\mathbf{B}|]), 
\end{equation}
while $\phi_{n}(r)$ satisfies
\begin{equation}
    \frac{\textnormal{d}}{\textnormal{d} r}\phi_{n} = 1 + \mathcal{O}\left(|\mu| + |\sigma|^{2}\right), \qquad \phi_{n}(0) = {\textstyle - \frac{m n \pi}{2}}, \qquad \forall n\in[0,N].\nonumber
\end{equation}
\end{lem}
\begin{proof}
It follows from \cite[Lemma 3.10]{scheel2003radially} that, for a given $0<M<\infty$, we can write \eqref{amp:AB;tilde} in the form
\begin{equation}\label{amp:AB;musigma}
	\begin{split}
   		 \frac{\textnormal{d}}{\textnormal{d} r} \hat{A}_{n} &= [\textnormal{i}k_{1,n}(\mu,\sigma) + k_{2,n}(\mu,\sigma)]\hat{A}_{n} + \hat{B}_{n} + \mathcal{O}((|\hat{\mathbf{A}}| + |\hat{\mathbf{B}}|)^2 + |\mu||\sigma|^M(|\hat{\mathbf{A}}| + |\hat{\mathbf{B}}|)),\\
    		\frac{\textnormal{d}}{\textnormal{d} r} \hat{B}_{n} &= [\textnormal{i}k_{1,n}(\mu,\sigma) + k_{2,n}(\mu,\sigma)]\hat{B}_{n} + c_{0,n}(\mu,\sigma)\hat{A}_{n} + \mathcal{O}((|\hat{\mathbf{A}}| + |\hat{\mathbf{B}}|)^2 + |\mu||\sigma|^M(|\hat{\mathbf{A}}| + |\hat{\mathbf{B}}|)),
	\end{split}
\end{equation}
for each $n\in[0,N]$ by a transformation 
\begin{equation}\label{Normal:Transf;Lin}
    \begin{pmatrix}
    \hat{A}_{n} \\ \hat{B}_{n}
    \end{pmatrix} := \left[\mathbbm{1} + \mathcal{T}_{n}(\sigma)\right]\begin{pmatrix}
    \widetilde{A}_{n} \\ \widetilde{B}_{n}
    \end{pmatrix} + \mathcal{O}(|\mu|(|\widetilde{\mathbf{A}}| + |\widetilde{\mathbf{B}}|)).
\end{equation}
Since the linearisation of \eqref{amp:AB;tilde} about $(\widetilde{\mathbf{A}},\widetilde{\mathbf{B}})=\mathbf{0}$ decouples for each $n\in[0,N]$, the derivation of these transformations follows in the same way as for the radial problem. We use matched asymptotics in order to calculate the following leading order expansions,
\begin{equation}
    \begin{split}
        \mathcal{T}_{n}(\sigma)\begin{pmatrix}
    		\widetilde{A}_{n} \\ \widetilde{B}_{n}
    		\end{pmatrix} &= \begin{pmatrix}
    		 - [\frac{\textnormal{i}\sigma}{4} + \mathcal{O}(\sigma^2)]\overline{\widetilde{A}}_{n} - [\frac{\sigma}{4}+\mathcal{O}(\sigma^2)]\overline{\widetilde{B}}_{n}\\    
    		[\frac{\textnormal{i}\sigma}{4} + \mathcal{O}(\sigma^2)]\overline{\widetilde{B}}_{n}
    		\end{pmatrix},\\
        k_{1,n}(\mu,\sigma) &= 1 + \mathcal{O}(|\mu| + |\sigma|^2), \qquad k_{2,n}(\mu,\sigma) = -{\textstyle\frac{\sigma}{2}} + \mathcal{O}(|\sigma|^3), \\
        c_{0,n}(\mu,\sigma) &= \mu\big[{\textstyle\frac{1}{4}}\big\langle \hat{U}_{1}^{*}, -\mathbf{M}_{2}\hat{U}_{0}\big\rangle_{2} + \mathcal{O}(|\mu| + |\sigma|^2)\big],\\
    \end{split}
\end{equation}
where we have used the symmetry of \eqref{amp:AB;tilde} with respect to the reverser $\mathscr{R}:(\widetilde{\mathbf{A}}, \widetilde{\mathbf{B}}, \sigma, r)\mapsto(\overline{\widetilde{\mathbf{A}}}, - \overline{\widetilde{\mathbf{B}}}, -\sigma, -r)$ in order to write down the higher-order terms containing $\sigma$. 

Following \cite[Lemma 2.6]{Scheel2014Grain}, we note that there exist smooth homogeneous polynomials $\{\Phi_n, \Psi_n\}_{n=0}^{N}$ of degree 2 such that the change of coordinates
\begin{equation}\label{Normal:Transf;Quad}
    \hat{A}_{n} = A_{n} + \Phi_{n}(\{A_{k}, B_{k}, \overline{A}_{k}, \overline{B}_{k}\}_{k=0}^{N}),\qquad
    \hat{B}_{n} = B_{n} + \Psi_{n}(\{A_{k}, B_{k}, \overline{A}_{k}, \overline{B}_{k}\}_{k=0}^{N}),
\end{equation}
transforms \eqref{amp:AB;musigma} when $\mu=\sigma=0$ into the normal form
\begin{equation}\label{amp:Normal;musigma}
	\begin{split}
   		 \frac{\textnormal{d}}{\textnormal{d} r} A_{n} &= \textnormal{i}A_{n} + B_{n} + \mathcal{O}((|\mathbf{A}| + |\mathbf{B}|)^3),\qquad 
    		\frac{\textnormal{d}}{\textnormal{d} r} B_{n} = \textnormal{i}B_{n} + \mathcal{O}((|\mathbf{A}| + |\mathbf{B}|)^3),
	\end{split}
\end{equation}
for each $n\in[0,N]$. As in \cite[Proof of Lemma 2.6]{Scheel2014Grain}, this result holds for $(\mathbf{A},\mathbf{B})$ in a neighbourhood of the origin if the quadratic terms in \eqref{amp:AB;musigma} belong to the range of the operator $(\mathcal{D}-\textnormal{i})$, where
\begin{equation}
    \mathcal{D} = \sum_{n=0}^{N} \left\{(\textnormal{i} A_n + B_n)\frac{\partial}{\partial A_n} + (\textnormal{i} B_n)\frac{\partial}{\partial B_n} + (-\textnormal{i} \overline{A}_n + \overline{B}_n)\frac{\partial}{\partial \overline{A}_n} + ( -\textnormal{i}\overline{B}_n)\frac{\partial}{\partial \overline{B}_n}\right\}.
\end{equation}
An explicit calculation verifies that every type of quadratic monomial belongs to the range of $(\mathcal{D} - \textnormal{i})$, and so we conclude that $\{\Phi_n, \Psi_n\}_{n=0}^{N}$ exist. Hence, applying both \eqref{Normal:Transf;Lin} and \eqref{Normal:Transf;Quad} transforms \eqref{amp:AB;tilde} into
\begin{equation}\label{amp:AB;Normal}
	\begin{split}
   		 \frac{\textnormal{d}}{\textnormal{d} r} A_{n} &= [\textnormal{i}k_{1,n}(\mu,\sigma) + k_{2,n}(\mu,\sigma)]A_{n} + B_{n} + \mathcal{O}((|\mathbf{A}| + |\mathbf{B}|)^3 + |\mu||\sigma|^M(|\mathbf{A}| + |\mathbf{B}|)),\\
    		\frac{\textnormal{d}}{\textnormal{d} r} B_{n} &= [\textnormal{i}k_{1,n}(\mu,\sigma) + k_{2,n}(\mu,\sigma)]B_{n} + c_{0,n}(\mu,\sigma)A_{n} + \mathcal{O}((|\mathbf{A}| + |\mathbf{B}|)^3 + |\mu||\sigma|^M(|\mathbf{A}| + |\mathbf{B}|)),
	\end{split}
\end{equation}
Finally, we remove a relative phase from $(A_{n}, B_{n})$ for each $n\in[0,N]$. We define $\phi_{n}(r)$ to be the solution of
\begin{equation}
    \frac{\textnormal{d}}{\textnormal{d} r} \phi_n = k_{1,n}(\mu,\sigma) = 1 + \mathcal{O}(|\mu| + |\sigma|^2), \qquad \phi_{n}(0) = {\textstyle-\frac{mn \pi}{2}}, 
\end{equation}
and employ the transformation $(A_n, B_n)\mapsto \textnormal{e}^{\textnormal{i}\phi_n}(A_n, B_n)$ for each $n\in[0,N]$. This transformation, after absorbing the higher order terms of $k_{2,n}(\mu,\sigma)$ and $c_{0,n}(\mu,\sigma)$ into the remainder, turns \eqref{amp:AB;Normal} into the desired form \eqref{Ntup:NormalForm} and thus completes the proof.
\end{proof} 

\begin{rmk}
    Let us briefly remark on the choice of $\phi_n(0)=-{\textstyle \frac{mn\pi}{2}}$ in Lemma \ref{Lemma:Ntup;normal}. We note that the core manifold \eqref{Core:An}, expressed in $(\widetilde{A}_{n},\widetilde{B}_n)$-coordinates, has a complex relative phase of $y_n:=r_0 - \frac{mn\pi}{2} - \frac{\pi}{4}$ for each $n\in[0,N]$. Thus, removing a phase of $\frac{mn\pi}{2}$ results in $(A_{n}, B_{n})$ having the same relative phase for all choices of $n\in[0,N]$.
\end{rmk}

Having transformed our equations into the radial normal form \eqref{Ntup:NormalForm} for a Hamilton--Hopf bifurcation, we also introduce the unconstrained variable $\kappa(r)$ to take the place of $\sqrt{\mu}$. Then, replacing $\mu$ with $\kappa^{2}$, the normal form \eqref{Ntup:NormalForm} can be extended to the following system,
\begin{equation}\label{Ntup:Normal:Ext}
	\begin{split}
    		\frac{\textnormal{d}}{\textnormal{d}r} \mathbf{A} &= - \frac{\sigma}{2} \mathbf{A} + \mathbf{B} + \mathbf{R}_{\mathbf{A}}(\mathbf{A}, \mathbf{B},\sigma,\kappa)\\
    		\frac{\textnormal{d}}{\textnormal{d}r} \mathbf{B} &= -\frac{\sigma}{2} \mathbf{B} + c_{0}\,\kappa^{2}\mathbf{A} + \mathbf{R}_{\mathbf{B}}(\mathbf{A}, \mathbf{B},\sigma,\kappa) \\
    		\frac{\textnormal{d}}{\textnormal{d}r}\sigma &= -\sigma^{2} \\
    		\frac{\textnormal{d}}{\textnormal{d}r}\kappa &= 0
	\end{split}
\end{equation}  
where we have
\begin{equation}
     	[\mathbf{R}_{\mathbf{A}/\mathbf{B}}]_{n} = \mathcal{O}([|\kappa|^{4} + |\sigma|^{3} + (|\mathbf{A}| + |\mathbf{B}|)^{2}][|\mathbf{A}| + |\mathbf{B}|]). 
\end{equation}

The rest of this section is dedicated to finding a parametrisation for exponentially decaying solutions in the far-field region and matching this with the core manifold \eqref{Core:un} at the point $r=r_{0}$. In order to find solutions to \eqref{Ntup:Normal:Ext} where $(\mathbf{A},\mathbf{B})$ decay exponentially as $r\to\infty$, we consider the far-field region to be made up of two distinct sub-regions. These regions are the `rescaling' region, where the radius $r$ is sufficiently large such that $r = \mathcal{O}(\mu^{-\frac{1}{2}})$, and the `transition' region, where the radius spans the gap between the rescaling region and the core region. We refer the reader to Figure \ref{fig:Geomtric-Blow-Up} for a visualisation of these different regions. In the rescaling region, we perform a coordinate transformation in order to find exponentially decaying solutions to \eqref{Ntup:NormalForm} for sufficiently small values of $\mu$; these solutions are tracked backwards through $r$ to the boundary of the rescaling region where they provide an initial condition for solutions in the transition region. Following this, solutions are tracked backwards in $r$ through the transition region, starting at the previously-found initial condition and staying sufficiently close to the linear algebraic flow of \eqref{Ntup:Normal:Ext}, until they arrive at the matching point $r=r_{0}$. See again, Figure \ref{fig:Geomtric-Blow-Up} for visual reference. Finally, we perform asymptotic matching to find intersections between the core and far-field parametrisations at the point $r=r_{0}$, thus defining localised solutions to \eqref{eqn:R-D;Galerk}. The dynamics in the aforementioned regions are described in the following subsection, followed by a subsection devoted entirely to the asymptotic matching.

\begin{figure}[t!] 
    \centering
    \includegraphics[width=0.85\linewidth]{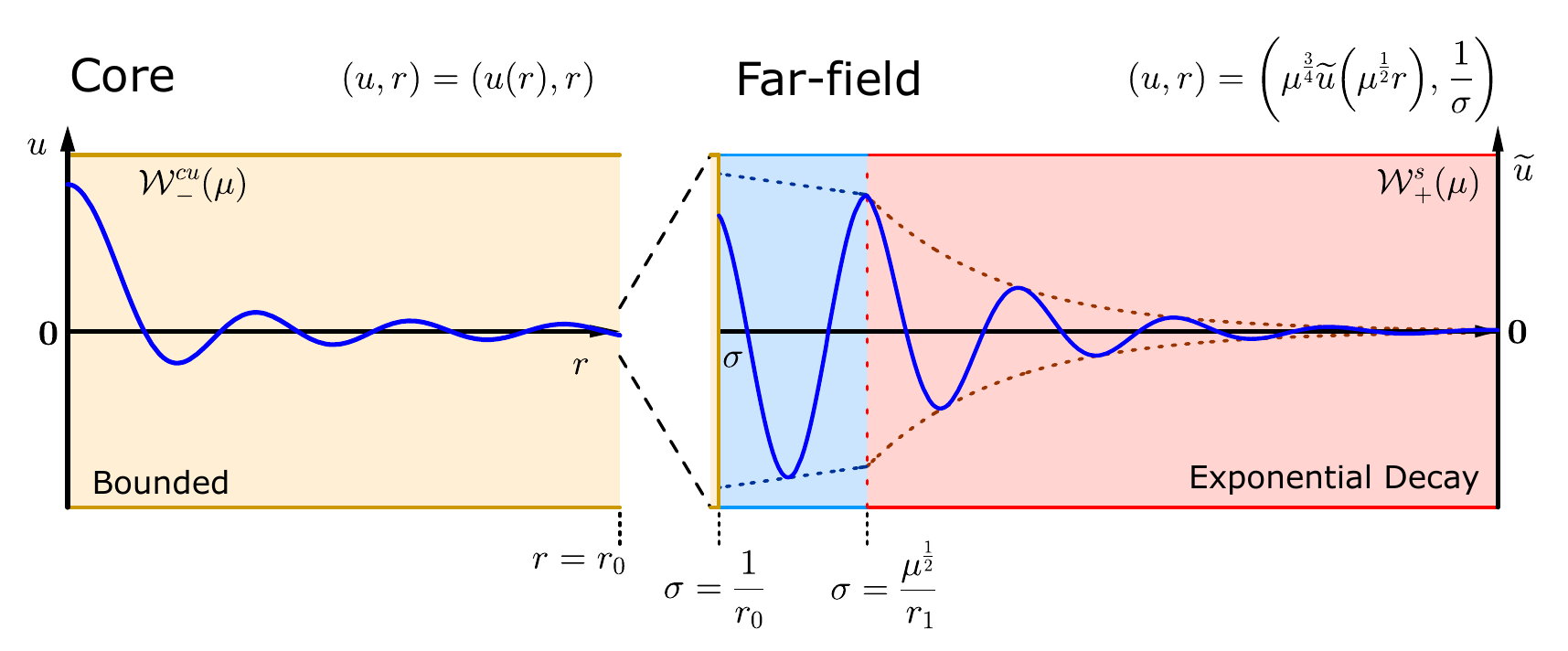}
    \caption{The far-field manifold is constructed for sufficiently small values of $\sigma\sim1/r$, which we divide into two distinct regions. We find exponentially decaying solutions in the rescaling region (red), where $\sigma=\mathcal{O}(\mu^{1/2})$ and $u=\mathcal{O}(\mu^{3/4})$, which we then evaluate at an intermediate point $\sigma=\mu^{1/2}/r_{1}$. Solutions are then tracked through the transition region (blue), where $\sigma$ and $u$ remain sufficiently small, to the point $\sigma=1/r_{0}$ and matched with the core manifold (yellow).}
    \label{fig:Geomtric-Blow-Up}
\end{figure}

\subsection{The Rescaling Chart}\label{subsec:Rescaling}

We now define rescaling coordinates in order to find exponentially decaying solutions for sufficiently large values of $r$; that is, we introduce
\begin{equation}\label{Ntup:resc}
    	\mathbf{A}_{R} :=  \kappa^{-\frac{3}{2}}\mathbf{A}, \quad \mathbf{B}_{R} := \kappa^{-\frac{5}{2}}\mathbf{B},\quad \sigma_{R} :=  \kappa^{-1}\sigma, \quad \kappa_{R} := \kappa, \quad s := \kappa r.
\end{equation}
These coordinates are of a similar form to the standard rescaling coordinates seen in \cite{mccalla2013spots}, except with the higher $\kappa$-scaling of $\mathbf{A}$ and $\mathbf{B}$ observed for axisymmetric spot A solutions in \cite{lloyd2009localized}. Then, we can write \eqref{Ntup:Normal:Ext} in the rescaling chart as,
\begin{equation}\label{Ntup:Normal;Resc}
	\begin{split}
    		&\frac{\textnormal{d}}{\textnormal{d}s} \mathbf{A}_{R} = - \frac{\sigma_{R}}{2} \mathbf{A}_{R} + \mathbf{B}_{R} + |\kappa_{R}|^{2}\mathbf{R}_{\mathbf{A},R}(\mathbf{A}_{R}, \mathbf{B}_{R},\sigma_{R},\kappa_{R})\\
    		&\frac{\textnormal{d}}{\textnormal{d}s} \mathbf{B}_{R} = -\frac{\sigma_{R}}{2} \mathbf{B}_{R} + c_{0}\,\mathbf{A}_{R} + |\kappa_{R}|\,\mathbf{R}_{\mathbf{B},R}(\mathbf{A}_{R}, \mathbf{B}_{R},\sigma_{R},\kappa_{R}) \\
    		&\frac{\textnormal{d}}{\textnormal{d}s}\sigma_{R} = -\sigma_{R}^{2} \\
    		&\frac{\textnormal{d}}{\textnormal{d}s}\kappa_{R} = 0
	\end{split}
\end{equation}
where
\begin{equation}
     \mathbf{R}_{\mathbf{A}/\mathbf{B}, R}(\mathbf{A}_{R}, \mathbf{B}_{R},\sigma_{R},\kappa_{R}) = \kappa_{R}^{-\frac{9}{2}}\mathbf{R}_{\mathbf{A}/\mathbf{B}}\left(\kappa_{R}^{\frac{3}{2}}\mathbf{A}_{R}, \kappa_{R}^{\frac{5}{2}}\mathbf{B}_{R}, \kappa_{R}\sigma_{R},\kappa_{R}\right)
\end{equation}
such that
\begin{equation}
     [\mathbf{R}_{\mathbf{A}/\mathbf{B}, R}]_{n} = \mathcal{O}([|\kappa_{R}| + |\sigma_{R}|^{3} + (|\mathbf{A}_{R}| + |\kappa_{R}||\mathbf{B}_{R}|)^{2}][|\mathbf{A}_{R}| + |\kappa_{R}||\mathbf{B}_{R}|])
\end{equation}
for all $n \in [0,N]$. We present the following lemma.

\begin{lem}\label{Lem:Resc;Evo} 
For each fixed choice of $r_{1}>0$, there is a constant $\kappa_{0}>0$ such that the set of exponentially decaying solutions 
to \eqref{Ntup:Normal;Resc} for $s\in[r_{1},\infty)$, evaluated at $s=r_{1}$, is given by
\begin{equation}
 	(\mathbf{A}_{R}, \mathbf{B}_{R}, \sigma_{R}, \kappa_{R})(r_{1}) = \bigg(\mathbf{a} r_{1}^{-\frac{1}{2}} \textnormal{e}^{\textnormal{i}Y}(1 + \mathcal{O}(\mu^{\frac{1}{2}})), -\sqrt{c_{0}}\mathbf{a} r_{1}^{-\frac{1}{2}} \textnormal{e}^{\textnormal{i}Y}(1 + \mathcal{O}(\mu^{\frac{1}{2}})),r_{1}^{-1}, \mu^{\frac{1}{2}}\bigg),
\end{equation}
for all $\mu<\kappa_{0}$, where $\mathbf{a}\in\mathbb{R}^{(N+1)}$, $\mathbf{a}\neq\mathbf{0}$, and $Y\in\mathbb{R}$ is arbitrary. 
\end{lem}

\begin{proof}
To begin, note that $\kappa_{R}$ acts as a parameter, and the subspace $\{\kappa_{R} = 0\}$ is invariant under the flow of \eqref{Ntup:Normal;Resc}. In this invariant subspace the dynamics reduce to
\begin{equation}\label{Ntup:Normal;Resc,0}
	\begin{split}
    		&\frac{\textnormal{d}}{\textnormal{d}s} \mathbf{A}_{R} = - \frac{\sigma_{R}}{2} \mathbf{A}_{R} + \mathbf{B}_{R} \\
		&\frac{\textnormal{d}}{\textnormal{d}s} \mathbf{B}_{R} = -\frac{\sigma_{R}}{2} \mathbf{B}_{R} + c_{0}\,\mathbf{A}_{R} \\
		&\frac{\textnormal{d}}{\textnormal{d}s}\sigma_{R} = -\sigma_{R}^{2},
	\end{split}
\end{equation}
which has an exponentially decaying solution for $s\in[r_{1},\infty)$ of the form
\begin{equation}\label{Ntup:RescSol;s}  
 	(\mathbf{A}_{R}, \mathbf{B}_{R}, \sigma_{R}, \kappa_{R})(s) = \bigg(\mathbf{a} s^{-\frac{1}{2}}\textnormal{e}^{-\sqrt{c_{0}}(s - r_{1})} \textnormal{e}^{\textnormal{i}Y}, -\sqrt{c_{0}}\mathbf{a} s^{-\frac{1}{2}}\textnormal{e}^{-\sqrt{c_{0}}(s - r_{1})} \textnormal{e}^{\textnormal{i}Y},\frac{1}{s}, 0\bigg),  
\end{equation}
where $\mathbf{a}\in\mathbb{R}^{(N+1)}$, $\mathbf{a}\neq\mathbf{0}$, and $Y\in\mathbb{R}$. Evaluating \eqref{Ntup:RescSol;s} at $s=r_{1}$, this becomes
\begin{equation}\label{Ntup:RescSol;delta0,0}
 	(\mathbf{A}_{R}, \mathbf{B}_{R}, \sigma_{R}, \kappa_{R})(r_{1}) = \bigg(\mathbf{a} r_{1}^{-\frac{1}{2}} \textnormal{e}^{\textnormal{i}Y}, -\sqrt{c_{0}}\mathbf{a} r_{1}^{-\frac{1}{2}} \textnormal{e}^{\textnormal{i}Y},r_{1}^{-1}, 0\bigg).    
\end{equation}
Any exponentially decaying solution for $\kappa_{R}(r)\equiv\varepsilon\ll1$ remains sufficiently close to the invariant subspace $\{\kappa_{R}(r)= 0\}$ since introducing $\kappa_R$ is a regular perturbation of \eqref{Ntup:Normal;Resc,0}. Therefore, setting $\kappa_{R}(s) = \mu^{\frac{1}{2}}$, we can express exponentially decaying solutions to \eqref{Ntup:Normal;Resc} evaluated at $s=r_{1}$ in the following form,
\begin{equation}\label{Ntup:RescSol;delta0} 
 	(\mathbf{A}_{R}, \mathbf{B}_{R}, \sigma_{R}, \kappa_{R})(r_{1}) = \bigg(\mathbf{a} r_{1}^{-\frac{1}{2}} \textnormal{e}^{\textnormal{i}Y}(1 + \mathcal{O}(\mu^{\frac{1}{2}})), -\sqrt{c_{0}}\mathbf{a} r_{1}^{-\frac{1}{2}} \textnormal{e}^{\textnormal{i}Y}(1 + \mathcal{O}(\mu^{\frac{1}{2}})),r_{1}^{-1}, \mu^{\frac{1}{2}}\bigg),   
\end{equation}
where the $\mathcal{O}(\mu^{\frac{1}{2}})$ terms are due to the $\mathcal{O}(|\kappa_{R}|)$ terms in \eqref{Ntup:Normal;Resc}. This completes the proof.
\end{proof} 

We conclude this subsection by noting that upon inverting the rescaling transformation \eqref{Ntup:resc}, exponentially decaying solutions to \eqref{Ntup:Normal:Ext} evaluated at $r=r_{1}\mu^{-\frac{1}{2}}$ take the form,
\begin{equation}\label{Ntup:Sol;delta0} 
    (\mathbf{A}, \mathbf{B}, \sigma, \kappa)(r_{1}\mu^{-\frac{1}{2}}) = \bigg(\mathbf{a} \mu^{\frac{3}{4}} r_{1}^{-\frac{1}{2}} \textnormal{e}^{\textnormal{i}Y}(1 + \mathcal{O}(\mu^{\frac{1}{2}})), -\sqrt{c_{0}}\mathbf{a}\mu^{\frac{5}{4}} r_{1}^{-\frac{1}{2}} \textnormal{e}^{\textnormal{i}Y}(1 + \mathcal{O}(\mu^{\frac{1}{2}})),\mu^{\frac{1}{2}}r_{1}^{-1}, \mu^{\frac{1}{2}}\bigg),  
\end{equation}
for sufficiently small values of $\mu$. Hence, with the previous lemma we parametrised the far-field manifold $\mathcal{W}^{s}_{+}(\mu)$ at a transition point $r=r_{1}\mu^{-\frac{1}{2}}$, after which exponential decay is guaranteed to occur. By tracking the trajectories of \eqref{Ntup:Normal:Ext} with initial condition \eqref{Ntup:Sol;delta0} backwards in $r$, we are able to extend this parametrisation to the point $r=r_{0}$, where we can then perform asymptotic matching to find intersections with the core manifold $\mathcal{W}^{cu}_{-}(\mu)$. This analysis of the transition chart is covered in the following subsection.

\subsection{The Transition Chart}\label{subsec:Transition}

In the previous subsection we parametrised the set of exponentially decaying solutions to \eqref{Ntup:Normal:Ext} for $r>r_{1}\mu^{-\frac{1}{2}}$. Now, in order to match these exponentially decaying solutions with the core manifold described in \S\ref{subsec:Core}, we must track the trajectories associated with \eqref{Ntup:Sol;delta0} backwards through the `transition' region $r_{0}\leq r\leq r_{1}\mu^{-\frac{1}{2}}$. We therefore look to solve the initial value problem
\begin{align}
    \frac{\textnormal{d}}{\textnormal{d} r}\mathbf{A} &= -\frac{\sigma}{2}\mathbf{A} + \mathbf{B} + \mathbf{R}_{\mathbf{A}}\left(\mathbf{A},\mathbf{B};\sigma,\kappa\right), & \mathbf{A}\left(r_{1}\mu^{-\frac{1}{2}}\right)&= \mu^{\frac{3}{4}}r_{1}^{-\frac{1}{2}}\mathbf{A}_{0},\nonumber\\
    \frac{\textnormal{d}}{\textnormal{d} r} \mathbf{B} &= -\frac{\sigma}{2}\mathbf{B} + c_{0}\,\kappa^{2}\mathbf{A} + \mathbf{R}_{\mathbf{B}}\left(\mathbf{A},\mathbf{B};\sigma,\kappa\right), & 
    \mathbf{B}\left(r_{1}\mu^{-\frac{1}{2}}\right)&= -\sqrt{c_{0}}\mu^{\frac{5}{4}}r_{1}^{-\frac{1}{2}}\mathbf{A}_{0},\nonumber\\
    \frac{\textnormal{d}}{\textnormal{d} r}\sigma &= -\sigma^{2}, & \sigma\left(r_{1}\mu^{-\frac{1}{2}}\right) &= \mu^{\frac{1}{2}}r_{1}^{-1},\label{Ntup:NormalForm;Initial}\\
    \frac{\textnormal{d}}{\textnormal{d} r}\kappa &= 0, & \kappa\left(r_{1}\mu^{-\frac{1}{2}}\right)&= \mu^{\frac{1}{2}},\nonumber
\end{align}
backwards in $r_{0}\leq r\leq r_{1}\mu^{-\frac{1}{2}}$, where $\mathbf{A}_{0}:=\mathbf{a}\textnormal{e}^{\textnormal{i}Y}(1 + \mathcal{O}(\mu^{\frac{1}{2}}))$. This leads to the following result.

\begin{lem} 
For each fixed choice of $0<r_{1}, r_{0}^{-1}\ll1$, there is a $\kappa_{0}>0$ such that solutions of the initial value problem \eqref{Ntup:NormalForm;Initial}, evaluated at $r=r_{0}$, are given by
\begin{equation}\label{Ntup:Farf;r0}
	\begin{split}
   		\mathbf{A}(r_{0}) &= \mu^{\frac{1}{2}} r_{0}^{-\frac{1}{2}} \textnormal{e}^{\textnormal{i}Y}\mathbf{a}(1 + \xi)\\
      		\mathbf{B}(r_{0}) &= -\sqrt{c_{0}}\mu r_{0}^{-\frac{1}{2}} \textnormal{e}^{\textnormal{i}Y}\mathbf{a}(1 + \xi) \\
		\sigma(r_{0}) &= r_{0}^{-1} \\
		\kappa(r_{0}) &=\mu^{\frac{1}{2}} \\
	\end{split}
\end{equation}
for all $\mu<\kappa_{0}$, where $\xi:= \mathcal{O}\left(\mu^{\frac{1}{2}} + r_{1} + r_{0}^{-1}\right)$, and $\mathbf{a}\in\mathbb{R}^{(N+1)}$ and $Y\in\mathbb{R}$ were introduced in Lemma \ref{Lem:Resc;Evo}.
\label{Lem:Transition;Evo}
\end{lem}
\begin{proof}
We begin by solving \eqref{Ntup:NormalForm;Initial} for $\kappa(r)$ and $\sigma(r)$ explicitly, giving
\begin{equation}
    \kappa(r) = \kappa\left(r_{1}\mu^{-\frac{1}{2}}\right) = \mu^{\frac{1}{2}}, \quad \quad \sigma(r) = \frac{1}{\left(r-r_{1}\mu^{-\frac{1}{2}}\right) + \sigma\left(r_{1}\mu^{-\frac{1}{2}}\right)^{-1}} = \frac{1}{r}.
\end{equation}
We then introduce the following transition coordinates
\begin{align}
    \mathbf{A}_{T} := r^{\frac{1}{2}}\mathbf{A}, \qquad \mathbf{B}_{T} := r^{\frac{1}{2}}\mathbf{B}, \label{Ntup:Transition}
\end{align}
such that the initial value problem \eqref{Ntup:NormalForm;Initial} becomes
\begin{align}
    \frac{\textnormal{d}}{\textnormal{d} r} \mathbf{A}_{T} &= \mathbf{B}_{T} + \mathbf{R}_{\mathbf{A},T}(\mathbf{A}_{T},\mathbf{B}_{T};r,\mu^{\frac{1}{2}}), & \mathbf{A}_{T}(r_{1}\mu^{-\frac{1}{2}})&= \mu^{\frac{1}{2}}\mathbf{A}_{0},\label{Ntup:NormalForm;Transition}\\
    \frac{\textnormal{d}}{\textnormal{d} r} \mathbf{B}_{T} &= c_{0}\,\mu\mathbf{A}_{T} + \mathbf{R}_{\mathbf{B},T}(\mathbf{A}_{T},\mathbf{B}_{T};r,\mu^{\frac{1}{2}}), & 
    \mathbf{B}_{T}(r_{1}\mu^{-\frac{1}{2}})&= -\sqrt{c_{0}}\mu\mathbf{A}_{0},\nonumber\\
\end{align}
where
\begin{equation}
    [\mathbf{R}_{\mathbf{A}/\mathbf{B},T}]_{n} = \mathcal{O}([\mu^{2} + r^{-3} + r^{-1}(|\mathbf{A}_{T}| + |\mathbf{B}_{T}|)^{2}][|\mathbf{A}_{T}| + |\mathbf{B}_{T}|]),
\end{equation}
for all $n \in [0,N]$. Then, integrating over $r_{0}\leq r\leq r_{1}\mu^{-\frac{1}{2}}$, \eqref{Ntup:NormalForm;Transition} becomes the integral equation,
\begin{equation}\label{Ntup:Transition;Initial}
	\begin{split}
    		\mathbf{A}_{T}(r) &= \mu^{\frac{1}{2}}\mathbf{A}_{0} +  \int^{r}_{r_{1}\mu^{-\frac{1}{2}}}\left\{\mathbf{B}_{T}(p) + \mathbf{R}_{\mathbf{A},T}\left(\mathbf{A}_{T}(p),\mathbf{B}_{T}(p);p,\mu^{\frac{1}{2}}\right)\right\}\textnormal{d}p, \\
    		\mathbf{B}_{T}(r) &= - \sqrt{c_{0}}\mu\mathbf{A}_{0} + \int_{r_{1}\mu^{-\frac{1}{2}}}^{r}\left\{ c_{0}\,\mu\mathbf{A}_{T}(p) + \mathbf{R}_{\mathbf{B},T}\left(\mathbf{A}_{T}(p),\mathbf{B}_{T}(p);p,\mu^{\frac{1}{2}}\right)\right\}\textnormal{d}p.
	\end{split}
\end{equation}
For sufficiently small values of $\mu$, we can apply the contraction mapping principle to show that \eqref{Ntup:Transition;Initial} has a unique solution $(\mathbf{A}_{T},\mathbf{B}_{T})$ in an appropriate small ball centred at the origin in $C([r_{0},r_{1}\mu^{-\frac{1}{2}}], \mathbb{C}^{2(N+1)})$; see, for example, \cite{Sandstede1997Convergence}. Furthermore, we can express the unique solution to \eqref{Ntup:Transition;Initial}, evaluated at $r=r_{0}$, as
\begin{equation}
    \mathbf{A}_{T}(r_{0}) = \mathbf{a}\mu^{\frac{1}{2}} \textnormal{e}^{\textnormal{i}Y}(1 + \mathcal{O}(\mu^{\frac{1}{2}} + r_{1} + r_{0}^{-1})), \quad \mathbf{B}_{T}(r_{0}) = -\sqrt{c_{0}}\mathbf{a}\mu \textnormal{e}^{\textnormal{i}Y}(1 + \mathcal{O}(\mu^{\frac{1}{2}} + r_{1} + r_{0}^{-1})).
\end{equation}
Inverting the transition transformation \eqref{Ntup:Transition} by multiplying the above expressions by $r_0^{-\frac{1}{2}}$ brings us to the desired result.
\end{proof} 

With the previous lemma we parametrised the far-field manifold $\mathcal{W}^{s}_{+}(\mu)$ at the matching point $r=r_{0}$. This now allows us to perform the asymptotic matching in order to find intersections of the core manifold $\mathcal{W}^{cu}_{-}(\mu)$ with $\mathcal{W}^s_+(\mu)$ for sufficiently small $\mu$. This matching is covered in the following subsection.

\subsection{Matching Core and Far Field}\label{subsec:CoreFarMatch}

We begin by applying \eqref{Ntup:Normal;transf} to the core parametrisation \eqref{Core:An}, such that the core manifold $\mathcal{W}^{cu}_{-}(\mu)$ can be expressed as 
\begin{equation}\label{Core:Am;param}
	\begin{split}
    		\mathbf{A}(r_{0}) &= \frac{1}{2}r_{0}^{-\frac{1}{2}}\;\textnormal{e}^{\textnormal{i}Y_{0}} \left[ \mathbf{d}_{1}(1 + \mathcal{O}(r_{0}^{-1})) -\textnormal{i} r_{0} \mathbf{d}_{2}(1 + \mathcal{O}(r_{0}^{-1})) + \mathcal{R}_{\mathbf{A}}(\mathbf{d}_{1},\mathbf{d}_{2},\mu^{\frac{1}{2}})\right], \\
    		\mathbf{B}(r_{0}) &= -\frac{1}{2}r_{0}^{-\frac{1}{2}}\;\textnormal{e}^{\textnormal{i}Y_{0}} \left[\mathbf{Q}_N^m(\mathbf{d}_{1})(\nu + \mathcal{O}(r_{0}^{-\frac{1}{2}})) + \textnormal{i}\mathbf{d}_{2}(1 + \mathcal{O}(r_{0}^{-1})) + \mathcal{R}_{\mathbf{B}}(\mathbf{d}_{1},\mathbf{d}_{2},\mu^{\frac{1}{2}})\right],
	\end{split}
\end{equation}
where $\nu = \frac{1}{2}\sqrt{\frac{\pi}{6}}\big\langle \hat{U}_1^*, \mathbf{Q}(\hat{U}_0, \hat{U}_0)\big\rangle$, $\mathbf{Q}_N^m:=\left(Q_{0}^m, \dots, Q_{N}^m\right)\in\mathbb{R}^{(N+1)}$ for $Q^{m}_{n}$ defined in \eqref{Psi:d1;defn}, $Y_{0}:=- \frac{\pi}{4} + \mathcal{O}(|\mu|r_{0} + r_{0}^{-2})$, and
\begin{equation}\label{R;AB;Defn}
     [\mathcal{R}_{\mathbf{A}}]_{n}  = \mathcal{O}_{r_{0}}(|\mu||\mathbf{d}| + |\mathbf{d}|^{2}), \quad  [\mathcal{R}_{\mathbf{B}}]_{n} = \mathcal{O}_{r_{0}}(|\mu||\mathbf{d}| + |\mathbf{d}_{2}|^{2} + |\mathbf{d}_{1}|^{3}),
\end{equation}
for all $n\in[0,N]$. Then, setting the far-field parametrisation \eqref{Ntup:Farf;r0} and the core parametrisation \eqref{Core:Am;param} equal to each other, we find that
\begin{equation}\label{Ntup:match;corefar}
	\begin{split}
    		2\mu^{\frac{1}{2}}\,  \textnormal{e}^{\textnormal{i}Y}\mathbf{a}\left(1 + \xi\right) &= \textnormal{e}^{\textnormal{i}Y_{0}} \left[ \mathbf{d}_{1}\left(1 + \mathcal{O}(r_{0}^{-1})\right) -\textnormal{i} r_{0} \mathbf{d}_{2}\left(1 + \mathcal{O}(r_{0}^{-1})\right) + \mathcal{R}_{\mathbf{A}}\right],\\
    		2\sqrt{c_{0}}\mu\,  \textnormal{e}^{\textnormal{i}Y}\mathbf{a}\left(1 + \xi\right) &= \textnormal{e}^{\textnormal{i}Y_{0}} \left[\mathbf{Q}_N^m(\mathbf{d}_{1})\left(\nu + \mathcal{O}(r_{0}^{-\frac{1}{2}})\right) + \textnormal{i}\mathbf{d}_{2}\left(1 + \mathcal{O}(r_{0}^{-1})\right) + \mathcal{R}_{\mathbf{B}}\right]. 
	\end{split}
\end{equation}
Let us briefly summarise the variables and parameters included in \eqref{Ntup:match;corefar}. We recall that the core parametrisation is determined by the linear coefficients $\mathbf{d}_{1},\mathbf{d}_{2}\in\mathbb{R}^{(N+1)}$, defined in Lemma \ref{Lemma:Ntup;Core}, and the fixed complex phase $Y_{0}\in\mathbb{R}$. Similarly, the far-field parametrisation is determined by the linear coefficient $\mathbf{a}\in\mathbb{R}^{(N+1)}$ and phase parameter $Y\in\mathbb{R}$, both introduced in Lemma \ref{Lem:Resc;Evo}. The other parameters in \eqref{Ntup:match;corefar} are the matching points $r_{1},r_{0}>0$, the bifurcation parameter $\mu\in\mathbb{R}$, and the quadratic coefficient $\nu=\frac{1}{2}\sqrt{\frac{\pi}{6}}\big\langle \hat{U}_1^*, \mathbf{Q}(\hat{U}_0, \hat{U}_0)\big\rangle \in\mathbb{R}$; $r_{1}$ and $r_{0}$ are introduced during the construction of the core and far-field manifolds, whereas $\mu$ and $\mathbf{Q}(\mathbf{u},\mathbf{u})$ are given by the equation \eqref{eqn:R-D}. Finally, we also note that the nonlinear terms $\mathbf{Q}^{m}_{N}=(Q_{0}^m,\dots,Q_{N}^m)$, $\mathcal{R}_{\mathbf{A}/\mathbf{B}}$, and $\xi$ are defined in \eqref{Psi:d1;defn}, \eqref{R;AB;Defn}, and Lemma \ref{Lem:Transition;Evo}, respectively.

Let us introduce the following coordinate transformations 
\begin{equation}\label{match:scaling}
    	\mathbf{d}_{1} = \frac{\sqrt{c_{0}}}{\nu}\mu^{\frac{1}{2}}\widetilde{\mathbf{d}}_{1},\qquad \mathbf{d}_{2} = \frac{c_{0}}{\nu}\mu\widetilde{\mathbf{d}}_{2},\qquad \mathbf{a} = \frac{\sqrt{c_{0}}}{2\nu}\widetilde{\mathbf{a}},\qquad Y = Y_{0} + \widetilde{Y},
\end{equation}
such that the leading order terms of \eqref{Ntup:match;corefar} scale with the same order in $\mu$. Eliminating common leading factors leads to
\begin{equation}\label{Ntup:match;corefar,scale}
	\begin{split}
    		\textnormal{e}^{\textnormal{i}\widetilde{Y}}\widetilde{\mathbf{a}}(1 + \xi) &= \widetilde{\mathbf{d}}_{1}(1 + \mathcal{O}(r_{0}^{-1})) + \mu^{\frac{1}{2}}\widetilde{\mathcal{R}}_{\mathbf{A}},\\
    \textnormal{e}^{\textnormal{i}\widetilde{Y}}\widetilde{\mathbf{a}}(1 + \xi) &= [\mathbf{Q}_N^m(\widetilde{\mathbf{d}}_{1})(1 + \mathcal{O}(r_{0}^{-\frac{1}{2}})) + \textnormal{i}\widetilde{\mathbf{d}}_{2}(1 + \mathcal{O}(r_{0}^{-1}))] + \mu^{\frac{1}{2}}\widetilde{\mathcal{R}}_{\mathbf{B}}, 
	\end{split}
\end{equation}
where
\begin{equation}
    	[\widetilde{\mathcal{R}}_{\mathbf{A}}]_{n} = \mathcal{O}_{r_{0}}(|\widetilde{\mathbf{d}}_{1}| + |\mu|^{\frac{1}{2}}|\widetilde{\mathbf{d}}_{1}|^{2} + |\mu||\widetilde{\mathbf{d}}_{2}| ), \quad [\widetilde{\mathcal{R}}_{\mathbf{B}}]_{n} = \mathcal{O}_{r_{0}}(|\widetilde{\mathbf{d}}_{1}| + |\mu|^{\frac{1}{2}}|\widetilde{\mathbf{d}}_{2}| ),
\end{equation}
for all $n\in[0,N]$. When $r_0 \gg 1$ and $0 < \mu,r_{1} \ll 1$, the leading-order system is given by
\begin{equation}\label{Ntup:match;corefar,scale0}
	\begin{split}
    		\textnormal{e}^{\textnormal{i}\widetilde{Y}}\widetilde{\mathbf{a}} &= \widetilde{\mathbf{d}}_{1} \\ 
		\textnormal{e}^{\textnormal{i}\widetilde{Y}}\widetilde{\mathbf{a}} &= \mathbf{Q}_N^m(\widetilde{\mathbf{d}}_{1}) + \textnormal{i}\widetilde{\mathbf{d}}_{2},
	\end{split}
\end{equation}
after dividing off the common factors of $\mu$. We introduce the real diagonal matrix $\mathbf{C}\in\mathbb{R}^{(N+1)\times(N+1)}$ given by $\mathbf{C} = \mathrm{diag}(1,-1,1,-1,\dots,(-1)^N)$. We prove in Lemma~\ref{lem:EquiQ} below that $\mathbf{Q}_N^m(\mathbf{C}^{m}\mathbf{a}) = \mathbf{C}^{m}\mathbf{Q}_N^m(\mathbf{a})$ for all $\mathbf{a} \in \R^{N+1}$, but for now it suffices to notice this can be verified by direct calculation. Hence, we emphasise that \eqref{Ntup:match;corefar,scale0} is invariant under the transformation
\begin{equation}\label{Transform:C}
	(\widetilde{\mathbf{d}}_{1}, \widetilde{\mathbf{d}}_{2}, \widetilde{\mathbf{a}}) \mapsto (\mathbf{C}^{m}\widetilde{\mathbf{d}}_{1}, \mathbf{C}^{m}\widetilde{\mathbf{d}}_{2}, \mathbf{C}^{m}\widetilde{\mathbf{a}}).
\end{equation} 
Then, taking real and imaginary parts, solutions to the leading order expression \eqref{Ntup:match;corefar,scale0} are equivalently expressed as the zeros of the  function
\begin{equation}\label{Ntup:G;defn}
    	\mathcal{G}:\mathbb{R}^{3(N+1)+1}\to\mathbb{R}^{3(N+1)+1}, \quad (\widetilde{\mathbf{d}}_{1},\widetilde{\mathbf{a}},\widetilde{\mathbf{d}}_{2},{\widetilde{Y}})^{T}\mapsto(\mathbf{G}_1,\mathbf{G}_2,\mathbf{G}_3,G_4)^{T},\\
\end{equation}
where
\begin{equation}\label{Ntup:G;expl}
	\begin{split}
    		\mathbf{G}_1 &= \widetilde{\mathbf{d}}_{1} - \cos(\widetilde{Y})\widetilde{\mathbf{a}}, \qquad  \mathbf{G}_2 = \mathbf{Q}_N^m(\widetilde{\mathbf{d}}_{1}) - \cos(\widetilde{Y})\widetilde{\mathbf{a}},\\
    		\mathbf{G}_3 &= \widetilde{\mathbf{d}}_{2} -\sin(\widetilde{Y})\widetilde{\mathbf{a}},\qquad G_4 = -\sin(\widetilde{Y}). 
	\end{split}
\end{equation}
It should be noted that taking the imaginary part of the first equation in \eqref{Ntup:match;corefar,scale0} results in the vector equation $\sin(\widetilde{Y})\widetilde{\mathbf{a}}=\mathbf{0}$. However, for $\widetilde{\mathbf{a}}\neq\mathbf{0}$, this reduces to obtaining $\widetilde{Y}$ such that $\sin(\widetilde{Y}) = 0$, which is captured by the final component of $\mathcal{G}$, defined in \eqref{Ntup:G;expl}. The following lemma characterises the roots of $\mathcal{G}$ by relating them to fixed points of $\mathbf{Q}_N^m$. Solving this fixed point problem is the focus of Section~\ref{s:patch;N} which follows, but for the remainder of this subsection we show that these fixed points lead to matched solutions of \eqref{R-D:U;vec} that are defined for all $r \geq 0$ and decay exponentially to 0 as $r \to \infty$.

\begin{lem}\label{Match:am;gen} 
Fix $m,N\in\mathbb{N}$. Then, the functional $\mathcal{G}$, defined in \eqref{Ntup:G;expl}, has zeros of the form
\begin{equation}\label{Ntup:Match;Soln}
    \mathcal{V}^{*}:=(\widetilde{\mathbf{d}}_{1},\widetilde{\mathbf{a}},\widetilde{\mathbf{d}}_{2},\widetilde{Y}) = ( \mathbf{C}^{m}\mathbf{a}^{*}, (-1)^{\alpha}\mathbf{C}^{m}\mathbf{a}^{*}, \mathbf{0}, \alpha \pi),
\end{equation}
where $\alpha\in\mathbb{Z}$ and $\mathbf{a}^{*}:=(a_{0},a_{1},\dots,a_{N})^{T}\in\mathbb{R}^{(N+1)}$ satisfies $\mathbf{a}^{*} = \mathbf{Q}_N^m(\mathbf{a}^{*})$, with $\mathbf{Q}_N^m:=(Q_{0}^m, \dots, Q_{N}^m)^{T}$ defined in \eqref{Psi:d1;defn}.
\end{lem}

\begin{proof} 
From the definition \eqref{Ntup:G;defn}, roots of $\mathcal{G}$ correspond to having $\mathbf{G}_1 = \mathbf{0}$, $\mathbf{G}_2 = \mathbf{0}$, $\mathbf{G}_3 = \mathbf{0}$, and $G_4 = 0$, as they are given in \eqref{Ntup:G;expl}. Then, we immediately find that solving $G_4 = 0$ gives that $\widetilde{Y} = \alpha\pi$, for any $\alpha\in\mathbb{Z}$. Substituting $\widetilde{Y}=\alpha \pi$ into $\mathbf{G}_3$, we find that $\widetilde{\mathbf{d}}_{2} = \mathbf{0}$ is the unique choice that solves $\mathbf{G}_3 = \mathbf{0}$. Then, solving $\mathbf{G}_1=\mathbf{0}$ results in $\widetilde{\mathbf{d}}_{1}=(-1)^{\alpha}\widetilde{\mathbf{a}}$, and so $\mathbf{G}_2=\mathbf{0}$ implies

\begin{equation}
    	\mathbf{Q}_N^m(\mathbf{a}^{*}) - \mathbf{a}^{*} = \mathbf{0},
\end{equation}
where we have defined $\widetilde{\mathbf{a}} = \left(-1\right)^{\alpha}\mathbf{a}^{*}$. 
Hence, after applying the transformation \eqref{Transform:C}, we conclude that \eqref{Ntup:Match;Soln} is a zero of $\mathcal{G}$ as long as $\mathbf{a}^{*}=(a_{0}, \dots, a_{N})^{T}$ is a fixed point of $\mathbf{Q}_N^m$, completing the proof.
\end{proof} 

In order to match solutions from the core manifold to the far-field manifold for small values of $\mu$, $r_{1}$ and $r_{0}^{-1}$, we require the Jacobian of $\mathcal{G}$, denoted throughout by $D\mathcal{G}$, to be invertible at the solution \eqref{Ntup:Match;Soln}. This follows from the fact that the higher order $(\mu,r_{1},r_0^{-1})$ in the matching equations \eqref{Ntup:match;corefar,scale} constitute a regular perturbation of $\mathcal{G}$ when these parameters are taken to be small. Hence, when the Jacobian $D\mathcal{G}$ is invertible at $\mathcal{V}^*$, we can evoke the implicit function theorem to solve \eqref{Ntup:match;corefar,scale} uniquely for all $0<\mu, r_{1}, r_{0}^{-1}\ll1$. Inverting the coordinate transformation \eqref{match:scaling}, we find that
\begin{equation}\label{d1:d2;defn}
	\begin{split}
    		d_{1}^{(n)} &= (-1)^{mn} \frac{\sqrt{c_{0}} a_{n}^*}{\nu} \mu^{\frac{1}{2}}\left(1 + \mathcal{O}(r_{0}^{-1} + r_{1} + \mu^{\frac{1}{2}}\right)) \\
    		d_{2}^{(n)} &= \mu\,\mathcal{O}(r_{0}^{-1} + r_{1} + \mu^{\frac{1}{2}}),
	\end{split}
\end{equation}
where $\mathbf{a}^* = \{a_{n}^*\}_{n=0}^{N}$ is a fixed point of $\mathbf{Q}_N^m$. 

Following our chain of arguments here, nondegenerate roots of $\mathcal{G}$ are used to solve the matching problem for $0<\mu, r_{1}, r_{0}^{-1}\ll1$, which in turn leads to localised $\mathbb{D}_{m}$ solutions to the Galerkin system \eqref{eqn:R-D;Galerk}. This therefore allows us to determine the leading order profile of $\mathbf{u}_{n}(r)$ for these localised solutions in each region of $r$. For the core region, we substitute \eqref{d1:d2;defn} into \eqref{var:U} to see that
\begin{equation}
    	\mathbf{u}_{n}(r) = \mu^{\frac{1}{2}}\frac{\sqrt{c_{0}} a_{n}}{\nu}(-1)^{mn}\sqrt{\frac{\pi}{2}}\,J_{mn}(r)\hat{U}_{0} + \mathcal{O}(\mu),\label{un:profile;core}
\end{equation}
for $r \in [0,r_{0}]$. For the transition region $r_0 \leq r \leq r_{1}\mu^{-\frac{1}{2}}$, solutions of \eqref{Ntup:NormalForm;Initial} take the form 
\begin{equation}
	\begin{split}
		\mathbf{A}(r) &= r^{-\frac{1}{2}}\mu^{\frac{1}{2}}\mathbf{A}_{0} \\ 
		\mathbf{B}(r) &= -\sqrt{c_{0}} r^{-\frac{1}{2}}\mu\mathbf{A}_{0},
	\end{split}
\end{equation} 
and so inverting the transformations in Lemma \ref{Lemma:Ntup;normal} and \eqref{R-D:Transformation;Farfield}, we recover,
\begin{equation}\label{un:profile;transition}
    	\textstyle \mathbf{u}_{n}(r) = \mu^{\frac{1}{2}}\frac{\sqrt{c_{0}} a_{n}}{\nu}(-1)^{m n}r^{-\frac{1}{2}}\,\cos(r - \frac{m n \pi}{2}-\frac{\pi}{4})\hat{U}_{0} + \mathcal{O}(\mu), \quad r \in [r_0,r_{1}\mu^{-\frac{1}{2}}].
\end{equation}
Finally, for the rescaling region $r \geq r_{1}\mu^{-\frac 1 2}$, leading-order solutions take the form of the connecting orbit \eqref{Ntup:RescSol;s} for $\mathbf{A}_{R}(s), \mathbf{B}_{R}(s)$ such that 
\begin{equation}
	\begin{split}
		\mathbf{A}(r) &= \mu^{\frac{1}{2}} r^{-\frac{1}{2}}\textnormal{e}^{\sqrt{c_{0}}(r_{1} - \mu^{\frac{1}{2}}r)} \mathbf{A}_{0} \\
		\mathbf{B}(r) &= -\sqrt{c_{0}}\mu r^{-\frac{1}{2}}\textnormal{e}^{\sqrt{c_{0}}(r_{1} - \mu^{\frac{1}{2}}r)}\mathbf{A}_{0}.
	\end{split}
\end{equation} 
Then, inverting the transformations in Lemma \ref{Lemma:Ntup;normal} and \eqref{R-D:Transformation;Farfield}, we recover
\begin{equation}
    	\textstyle \mathbf{u}_{n}(r) = \mu^{\frac{1}{2}}\frac{\sqrt{c_{0}} a_{n}}{\nu}(-1)^{m n} r^{-\frac{1}{2}}\,\textnormal{e}^{\frac{1}{2}(r_{1} - \mu^{\frac{1}{2}}r)}\,\cos(r - \frac{m n \pi}{2}-\frac{\pi}{4})\hat{U}_{0} + \mathcal{O}(\mu)\label{un:profile;rescaling}
\end{equation}
when $r \geq r_{1}\mu^{-\frac{1}{2}}$. To recover the case when $\mathbf{M}_{1}$ has a repeated eigenvalue of $\lambda=-k_{c}^{2}$, we invert the rescaling \eqref{k:rescaling} and transform the arbitrary values $r_{0},r_{1}$ accordingly. Then, for each $n\in[0,N]$ the radial amplitude $\mathbf{u}_{n}(r)$ has the following profile:
\begin{equation}
    \mathbf{u}_{n}(r) = \frac{\sqrt{k_{c}c_{0}}a_{n}\mu^{\frac{1}{2}}}{\nu}(-1)^{m n}
    \begin{cases}
         \sqrt{\frac{k_{c}\pi}{2}}\,J_{mn}(k_{c} r)\hat{U}_{0} + \mathcal{O}(\mu^{\frac{1}{2}}) & 0\leq r\leq r_{0},  \\
         r^{-\frac{1}{2}}\,\cos(k_{c} r - \frac{m n \pi}{2} -\frac{\pi}{4})\hat{U}_{0} + \mathcal{O}(\mu^{\frac{1}{2}}), & r_{0} \leq r \leq r_1\mu^{-\frac{1}{2}},\\
         r^{-\frac{1}{2}}\,\textnormal{e}^{\sqrt{c_{0}}(r_1- \mu^{\frac{1}{2}}r)}\,\cos( k_{c} r- \frac{m n \pi}{2} -\frac{\pi}{4})\hat{U}_{0} + \mathcal{O}(\mu^{\frac{1}{2}}), & r_1\mu^{-\frac{1}{2}} \leq r,
    \end{cases}
\end{equation} 
uniformly as $\mu\to0^{+}$. Hence, a localised $\mathbb{D}_{m}$ solution to the Galerkin system \eqref{eqn:R-D;Galerk} has a core profile of the form
\begin{equation}\label{Prof:NPatch}
    \mathbf{u}(r,\theta) = \frac{k_{c}\sqrt{3 \big\langle \hat{U}_{1}^{*}, -\mathbf{M}_{2}\hat{U}_{0}\big\rangle_{2}}}{\big\langle \hat{U}_{1}^{*}, \mathbf{Q}(\hat{U}_{0},\hat{U}_{0})\big\rangle_{2}}\mu^{\frac{1}{2}}\left[a_{0}\,J_{0}(k_{c} r) + 2\sum_{n=1}^{N} \left\{\left(-1\right)^{mn}a_{n}\,J_{m n}(k_{c} r)\,\cos\left(m n \theta\right)\right\}\right]\hat{U}_{0} + \mathcal{O}(\mu),
\end{equation}
for $r\in[0,r_{0}]$ uniformly as $\mu\to0^+$, where we have used $\nu=\frac{1}{2}\sqrt{\frac{\pi}{6}}\,\big\langle \hat{U}_{1}^{*}, \mathbf{Q}(\hat{U}_{0},\hat{U}_{0})\big\rangle_{2}$ and $c_{0}=\frac{1}{4}\big\langle \hat{U}_{1}^{*}, -\mathbf{M}_{2}\hat{U}_{0}\big\rangle_{2}$ in order to express \eqref{Prof:NPatch} in terms of the matrix $\mathbf{M}_{2}$ and the quadratic term $\mathbf{Q}$ defined in the original equation \eqref{eqn:R-D}.

Hence, it remains to (i) identify fixed points of $\mathbf{Q}_N^m$ to give zeros of $\mathcal{G}$, and (ii) verify that these zeros are nondegenerate to arrive at the main results in Section~\ref{s:Results}. The following lemma shows that nondegenerate roots of $\mathcal{G}$ lie in one-to-one correspondence with fixed points of $\mathbf{Q}_N^m$, $\mathbf{a}^{*}\in\mathbb{R}^{N+1}$, such that the matrix $\mathbbm{1}_N - D\mathbf{Q}_N^m(\mathbf{a}^{*})$ is invertible. That is, $\mathbf{a}^*$ is a nondegenerate fixed point of $\mathbf{Q}_N^m$. As stated above, the focus of Section \ref{s:patch;N} is to determine such fixed points, while in this section we provide the sufficient conditions that imply the existence of localised $\mathbb{D}_{m}$ patches to the Galerkin system \eqref{eqn:R-D;Galerk}. We therefore conclude this section with the following result.

\begin{lem}\label{lem:NondegenMatch} 
Fix $m,N\in\mathbb{N}$. Then, for $\mathcal{G}$, $\mathcal{V}^{*}$ and $\mathbf{Q}_{N}^{m}$ as in Lemma~\ref{Match:am;gen}, we have 
\begin{equation}\label{det:G}
    \det(D\mathcal{G}[\mathcal{V}^{*}]) = (-1)^{(\alpha+1)(N+1) }\det\,[\mathbbm{1}_N - D\mathbf{Q}_N^m(\mathbf{a}^{*})],
\end{equation}
\end{lem}

\begin{proof}
The Jacobian $D\mathcal{G}$ can be written as
\begin{equation}
    D\mathcal{G} = \begin{pmatrix}
    D_{\widetilde{\mathbf{d}}_{1}}\mathbf{G}_1 & D_{\widetilde{\mathbf{a}}}\mathbf{G}_1 & D_{\widetilde{\mathbf{d}}_{2}}\mathbf{G}_1 & D_{\widetilde{Y}}\mathbf{G}_1 \\
    D_{\widetilde{\mathbf{d}}_{1}}\mathbf{G}_2 & D_{\widetilde{\mathbf{a}}}\mathbf{G}_2 & D_{\widetilde{\mathbf{d}}_{2}}\mathbf{G}_2 & D_{\widetilde{Y}}\mathbf{G}_2 \\
    D_{\widetilde{\mathbf{d}}_{1}}\mathbf{G}_3 & D_{\widetilde{\mathbf{a}}}\mathbf{G}_3 & D_{\widetilde{\mathbf{d}}_{2}}\mathbf{G}_3 & D_{\widetilde{Y}}\mathbf{G}_3 \\
    D_{\widetilde{\mathbf{d}}_{1}}G_4 & D_{\widetilde{\mathbf{a}}}G_4 & D_{\widetilde{\mathbf{d}}_{2}}G_4 & D_{\widetilde{Y}}G_4 
    \end{pmatrix} = \begin{pmatrix}
    \mathbbm{1}_N & -\cos(\widetilde{Y})\mathbbm{1}_N & \mathbb{O}_N & \sin(\widetilde{Y})\widetilde{\mathbf{a}} \\
    D_{\widetilde{\mathbf{d}}_{1}}\mathbf{Q}_N^m(\widetilde{\mathbf{d}}_{1}) & -\cos(\widetilde{Y})\mathbbm{1}_N & \mathbb{O}_N & \sin(\widetilde{Y})\widetilde{\mathbf{a}} \\
    \mathbb{O}_N & -\sin(\widetilde{Y})\mathbbm{1}_N & \mathbbm{1}_n & -\cos(\widetilde{Y})\widetilde{\mathbf{a}} \\
    \mathbf{0}^{T} & \mathbf{0}^{T} & \mathbf{0}^{T} & -\cos(\widetilde{Y})
    \end{pmatrix}
\end{equation}
where $\mathbf{0}\in\mathbb{R}^{(N+1)}$, $\mathbb{O}_N\in\mathbb{R}^{(N+1)\times(N+1)}$ denote the zero vector and square matrix, respectively, and subscripts on the differential $D$ denote what variable the derivative is taken with respect to. Evaluating at the solution $\mathcal{V}^{*}$, defined in \eqref{Ntup:Match;Soln}, we find that
\begin{equation}
    D\mathcal{G}(\mathcal{V}^{*}) = \begin{pmatrix}
    \mathbbm{1}_N & -(-1)^{\alpha}\mathbbm{1}_N & \mathbb{O}_N & \mathbf{0} \\
    D_{\mathbf{C}^{m}\mathbf{a}^{*}}\mathbf{Q}_N^m(\mathbf{C}^{m}\mathbf{a}^{*}) & -(-1)^{\alpha}\mathbbm{1}_N & \mathbb{O}_N & \mathbf{0} \\
    \mathbb{O}_N & \mathbb{O}_N & \mathbbm{1}_N & -\mathbf{C}^{m}\mathbf{a}^{*} \\
    \mathbf{0}^{T} & \mathbf{0}^{T} & \mathbf{0}^{T} & -(-1)^{\alpha}
    \end{pmatrix}.
\end{equation}
We note that $D_{\mathbf{C}^{m}\mathbf{a}^{*}}\mathbf{Q}_N^m(\mathbf{C}^{m}\mathbf{a}^{*}) = \mathbf{C}^{m}D_{\mathbf{a}^{*}}\mathbf{Q}_N^m(\mathbf{a}^{*})\mathbf{C}^{-m}$, and so the determinant of $D\mathcal{G}(\mathcal{V}^{*})$ is
\begin{equation}
	\begin{split}
    \det\,[D\mathcal{G}(\mathcal{V}^{*})]&= -(-1)^{\alpha}\det\,[-(-1)^{\alpha}[\mathbbm{1}_{N} - \mathbf{C}^{m}D_{\mathbf{a}^{*}}\mathbf{Q}_N^m(\mathbf{a}^{*})\mathbf{C}^{-m}]] \\ 
    &= (-1)^{(1+\alpha)(N+1)}\det\,[\mathbbm{1}_N - D_{\mathbf{a}^{*}}\mathbf{Q}_N^m(\mathbf{a}^{*})].
	\end{split}
\end{equation}
Hence, the proof is complete.
\end{proof} 

%
%
%
%
%
%
%
%
%
%
\section{Satisfying the Matching Condition}\label{s:patch;N}

In Section~\ref{s:Matching}, and in particular \S\ref{subsec:CoreFarMatch}, we showed that to prove the existence of localised $\mathbb{D}_{m}$ patch in the $N$-truncated Galerkin system \eqref{eqn:R-D;Galerk} we are required to identify isolated solutions of the matching problem
\begin{equation}\label{GeneralMatch}
    a_{n} = 2\sum_{j=1}^{N-n} \cos\left(\frac{m\pi(n-j)}{3}\right) a_{j} a_{n+j} + \sum_{j=0}^{n} \cos\left(\frac{m\pi(n-2j)}{3}\right) a_{j}a_{n-j}, \quad \forall n \in[0,N].
\end{equation}
To this end, our goal in this section is to identify nontrivial solutions of \eqref{GeneralMatch} which in turn provide the results from Section~\ref{s:Results}. We recall the notation of Lemma~\ref{Lemma:Ntup;Core} to write $\mathbf{Q}_N^m = (Q_{0}^m,\dots,Q_{N}^m)^{T}$, with
\begin{equation}
    Q_{n}^m(\mathbf{a}) = 2\sum_{j=1}^{N-n} \cos\left(\frac{m\pi(n-j)}{3}\right) a_{j} a_{n+j} + \sum_{j=0}^{n} \cos\left(\frac{m\pi(n-2j)}{3}\right) a_{j}a_{n-j},
\end{equation}
for each $n\in[0,N]$. The matching problem \eqref{GeneralMatch} is then written compactly as $\mathbf{a} = \mathbf{Q}_N^m(\mathbf{a})$. As detailed in Lemmas~\ref{Match:am;gen} and \ref{lem:NondegenMatch}, obtaining solutions of the matching problem \eqref{GeneralMatch} is equivalent to identifying nontrivial and nondegenerate fixed points of $\mathbf{Q}_N^m$. Recall that for an identified fixed point $\mathbf{a}^*$, the corresponding non-degeneracy condition is 
\begin{equation}
	\det\,\left(\mathbbm{1}_N - D\mathbf{Q}_N^m(\mathbf{a}^{*})\right)\neq0,
\end{equation} 
where we recall that $\mathbbm{1}_N$ is the identity matrix of size $(N+1)\times(N+1)$ and $D\mathbf{Q}_N^m$ denotes the Jacobian matrix of the nonlinear function $\mathbf{Q}_N^m$.

In the following subsections we detail the existence and general properties of nondegenerate fixed points of $\mathbf{Q}_N^m$. We begin in \S\ref{subsec:MatchProperties} with some important properties of the mapping $\mathbf{Q}_N^m$ that allow us to quotient the search for fixed points by important symmetries coming from the system \eqref{FourierExp:R-D}. Then, in \S\ref{subsec:SmallMatch} we provide explicit solutions up to these symmetries for small-layer patches, which completes the proof of Theorem~\ref{thm:SmallPatch}. We then provide a theoretical analysis which details the existence of solutions to the matching problem with $N \gg 1$ in \S\ref{subsec:BigMatch}, which completes the proof of Theorem~\ref{thm:BigPatch}. All proofs of solutions for the small-layer patches are left to \S\ref{subsec:SmallProofs} in the Appendix.

\subsection{Properties of the Matching Equations}\label{subsec:MatchProperties}

We begin by noting some qualitative properties of the Fourier-polar decomposition \eqref{FourierExp:R-D}, and how these affect the matching problem \eqref{GeneralMatch}. These properties are summarised in the following lemmas and can be used to classify solutions of the matching equations up to symmetry.

\begin{lem}\label{lem:EquiQ} 
   Fix $m,N \geq 1$ and define the linear transformation $\mathcal{R}:\mathbb{R}^{N+1} \to \mathbb{R}^{N+1}$ acting element-wise by $a_n \mapsto (-1)^na_n$. Then, $\mathbf{Q}_N^m$ is equivariant with respect to $\mathcal{R}$, i.e. $\mathbf{Q}_N^m(\mathcal{R} \mathbf{a}) = \mathcal{R}\mathbf{Q}_N^m(\mathbf{a})$ for all $\mathbf{a}\in\R^{N+1}$.
\end{lem}

\begin{proof}
	For any $\mathbf{a}\in\R^{N+1}$ and $n \in [0,N]$ we have
	\begin{equation}
		\begin{split}
			Q_n^m(\mathcal{R}\mathbf{a}) &= 2\sum_{j = 1}^{N-n}(-1)^{2j + n}\cos\bigg(\frac{m\pi(n - j)}{3}\bigg)a_ja_{n+j} + \sum_{j = 0}^n(-1)^n\cos\bigg(\frac{m\pi(n - 2j)}{3}\bigg)a_ja_{n-j} \\
			 &= (-1)^n\bigg[2\sum_{j = 1}^{N-n}(-1)^{2j}\cos\bigg(\frac{m\pi(n - j)}{3}\bigg)a_ja_{n+j} + \sum_{j = 0}^n \cos\bigg(\frac{m\pi(n - 2j)}{3}\bigg)a_ja_{n-j}\bigg] \\
			 &= (-1)^nQ_n^m(\mathbf{a}).
		\end{split}
	\end{equation}
	Hence, from the definition of $\mathcal{R}$ we have $\mathbf{Q}_N^m(\mathcal{R} \mathbf{a}) = \mathcal{R}\mathbf{Q}_N^m(\mathbf{a})$ for all $\mathbf{a}\in\R^{N+1}$.
\end{proof} 

From Lemma~\ref{lem:EquiQ} we see that if $\mathbf{a}^*$ is a fixed point of $\mathbf{Q}_N^m$, then so is $\mathcal{R}\mathbf{a}^*$ since 
\begin{equation}
	\mathbf{Q}_N^m(\mathcal{R} \mathbf{a}^*) = \mathcal{R}\mathbf{Q}_N^m(\mathbf{a}) = \mathcal{R}\mathbf{a}^*. 
\end{equation}
Transforming a fixed point of $\mathbf{Q}_N^m$ into another one by applying $\mathcal{R}$ is a consequence of the fact that one may rotate a localised pattern of \eqref{eqn:R-D;Galerk} by $\pi/m$ to obtain the same localised pattern with a different orientation. Indeed, using the Fourier-polar decomposition we see that 
\begin{equation}
	\begin{split}
		U_{m} (r,\theta + \pi/m) &= \mathbf{u}_0(r) + 2\sum_{n= 1}^N \mathbf{u}_n(r)\cos(mn\theta + n\pi) \\
		&= \mathbf{u}_0(r) + 2\sum_{n= 1}^N (-1)^n \mathbf{u}_n(r)\cos(mn\theta).
	\end{split}
\end{equation}
 Hence, rotating a solution by $\pi/m$ is equivalent to sending $\mathbf{u}_n(r) \mapsto (-1)^n \mathbf{u}_n(r)$ for all $n = 0,1,\dots,N$. This rotation is reflected in the matching equation $\mathbf{Q}_N^m(\mathbf{a}) = \mathbf{a}$ by applying $\mathcal{R}$ to $\mathbf{a}$.

\begin{lem}\label{lem:HarmQ} 
	Fix $m_{0},N_{0} \in\mathbb{N}$; then, for any $m,N\in\mathbb{N}$ such that $i:= \frac{m_{0}}{m}\in\mathbb{Z}$ and $\lfloor \frac{N}{i} \rfloor = N_{0}$, there exists the linear transformation $\mathcal{H}^{i}_{N}:\mathbb{R}^{N_{0}+1}\to\mathbb{R}^{N+1}$, defined by
	\begin{equation}
		\left[\mathcal{H}^{i}_{N}\mathbf{a}\right]_{\ell} = \begin{cases}
			a_n, & \ell = ni \\
			0, & \mathrm{otherwise}
		\end{cases}
	\end{equation}
	for each $\ell\in[0,N]$, such that any $\mathbf{a}^{*}$ that satisfies $\mathbf{Q}_{N_{0}}^{m_{0}}(\mathbf{a}^{*}) = \mathbf{a}^{*}$ also satisfies $\mathbf{Q}_{N}^{m}(\mathcal{H}^{i}_{N}\mathbf{a}^{*}) = \mathcal{H}^{i}_{N}\mathbf{a}^{*}$.
\end{lem}

\begin{proof}
	Take $\mathbf{b}^*:=\mathcal{H}_{N}^{i}\mathbf{a}^{*}$ to be as defined in the statement of the lemma. Then, if $\ell$ is not a multiple of $i$, we get $Q_\ell^m(\mathbf{b}^*) = 0 = b_\ell^*$ since  $b^*_{j} b^*_{\ell + j} = b^*_{j} b^*_{\ell - j} = 0$ for all $j$ by definition of $\mathbf{b}^*$. On the other hand, if $\ell = ni$ for some $n = 0,1,\dots,N$, we have 
	\begin{equation}
		\begin{split}
			Q_\ell^m(\mathbf{b}^*) &= 2\sum_{j=1}^{N-\ell} \cos\left(\frac{m\pi( \ell-j)}{3}\right) b^*_{j} b^*_{\ell+j} + \sum_{j=0}^{\ell} \cos\left(\frac{m\pi( \ell - 2j)}{3}\right) b^*_{j}b^*_{\ell - j} \\
			&= 2\sum_{j=1}^{N_{0}-n} \cos\left(\frac{m i\pi(n-j)}{3}\right) a^*_{j} a^*_{n+j} + \sum_{j=0}^{n} \cos\left(\frac{m i\pi(n -2j)}{3}\right) a^*_{j}a^*_{n - j} \\
			&= Q_n^{m_{0}}(\mathbf{a}^*)\\
			&= a^*_n \\
			&= b^*_\ell, 
		\end{split}
	\end{equation}
	where we have used the fact that $\mathbf{a}^*$ is a fixed point of $\mathbf{Q}_{N_{i}}^{m_{0}}$. This concludes the proof. 
\end{proof}

As with Lemma~\ref{lem:EquiQ}, the results of Lemma~\ref{lem:HarmQ} represent a fundamental fact about solutions of \eqref{eqn:R-D;Galerk}. Indeed, a $\mathbb{D}_{mi}$ pattern expanded in $N$ Fourier modes can be seen as a $\mathbb{D}_{m}$ pattern expanded in $Ni$ Fourier modes. To see this, note that if 
\begin{equation}
	U_{mi}(r,\theta) = \mathbf{u}_0(r) + 2\sum_{n = 1}^{N}\mathbf{u}_n(r)\cos(min\theta)
\end{equation} 
we may define  
\begin{equation}
	U_{m}(r,\theta) = \mathbf{u}_0(r) + 2\sum_{\ell = 1}^{N i}\begin{cases} \mathbf{u}_\ell(r)\cos(m \ell \theta), & \ell = ni \\ \mathbf{0}, & \mathrm{otherwise} \end{cases},
\end{equation} 
which is analogous to the definition of $\mathcal{H}_{N}^{i}\mathbf{a}^*$ in Lemma~\ref{lem:HarmQ}. This property follows from the fact that $\mathbb{D}_{m}$ is a subgroup of $\mathbb{D}_{mi}$ for any $i \geq 1$. 

\begin{lem}\label{lem:RecipQ} 
	Let $m,N\in\mathbb{N}$ such that $6\mid m$, and suppose $\mathbf{a}^* = (a_0^*,\dots,a_N^*) \in \R^{N+1}$ satisfies $\mathbf{Q}_N^{m}(\mathbf{a}^*) = \mathbf{a}^*$. Then, defining $\mathbf{b}^* = (b_0^*,\dots,b_{N}^*) \in \R^{N + 1}$ by
	\begin{equation}
		b_\ell = \begin{cases}
			1-a_\ell, & \ell = 0 \\
			-a_\ell, & \mathrm{otherwise}
		\end{cases}
	\end{equation}
	for each $\ell\in[0,N]$ gives that $\mathbf{Q}_{N}^m(\mathbf{b}^*) = \mathbf{b}^*$. 
\end{lem}

\begin{proof}
	We begin by noting that with $m \in\mathbb{N}$ such that $6 \mid m$, we have 
	\begin{equation}
		\cos\left(\frac{m\pi(n-j)}{3}\right) = \cos\left(\frac{m\pi(n-2j)}{3}\right) = 1
	\end{equation} 
	for all integers $n$ and $j$. Then, taking $\mathbf{b}^*$ to be as defined in the statement of the lemma, for $\ell = 1,\dots,N$ we have 
	\begin{equation}
		\begin{split}
			Q_\ell^m(\mathbf{b}^*) &= 2\sum_{j=1}^{N-\ell} b^*_{j} b^*_{\ell+j} + \sum_{j=1}^{\ell-1} b^*_{j}b^*_{\ell - j} + 2b_{0}^{*}b_{\ell}^{*}\\
			&= 2\sum_{j=1}^{N-\ell} a^*_{j} a^*_{\ell+j} + \sum_{j=1}^{\ell-1} a^*_{j}a^*_{\ell - j} + 2(a^{*}_{0}-1)a_{\ell}^{*} \\
			&= Q_\ell^{m}(\mathbf{a}^*) - 2a_{\ell}^{*}\\
			&= -a^*_\ell \\
			&= b^*_\ell, 
		\end{split}
	\end{equation}
	where we have used the fact that $\mathbf{a}^*$ is assumed to be a fixed point of $\mathbf{Q}_N^{m}$. Turning to the case when $\ell=0$, we similarly have
	\begin{equation}
		\begin{split}
			Q_0^m(\mathbf{b}^*) &= 2\sum_{j=1}^{N} (b^*_{j})^{2} + (b_{0}^{*})^{2} \\
			&= 2\sum_{j=1}^{N} (a^*_{j})^{2} + (1-a_{0}^{*})^{2} \\
			&= Q_0^{m}(\mathbf{a}^*) + 1 - 2a_{0}^{*}\\
			&= 1-a^*_0 \\
			&= b^*_0, 
		\end{split}
	\end{equation}
	where we have again used the fact that $Q_0^m(\mathbf{a}^*) = a_0^*$. Hence, we have shown that $\mathbf{Q}_{N}^m(\mathbf{b}^*) = \mathbf{b}^*$, proving the claim. 
\end{proof} 

The result of Lemma~\ref{lem:RecipQ} is less intuitive than those of Lemmas~\ref{lem:EquiQ} and \ref{lem:HarmQ}. From Lemma~\ref{lem:RecipQ} we see that if $\mathbf{a}^{*}$ is a fixed point of $\mathbf{Q}^{6m}_{N}$, then $\mathbf{b}^{*}:=\mathbf{a}_{1} - \mathbf{a}^{*}$ is also a fixed point of $\mathbf{Q}^{6m}_{N}$, where $\mathbf{a}_{1}=(1,0,\dots,0)$. The fixed point $\mathbf{a}_{1}$ corresponds to a radial spot $U_{0}(r)$ and so Lemma~\ref{lem:RecipQ} implies that, for any localised pattern $U_{6m}(r,\theta)$ of \eqref{eqn:R-D;Galerk}, there exists a corresponding localised pattern of the form $V_{6m}(r,\theta) = U_{0}(r) - U_{6m}(r,\theta)$. Such solutions are analogous to dark solitons in nonlinear optics, where localised solutions appear as `holes' on a nontrivial uniform background state \cite{KIVSHAR1998Dark}. If we think of $U_{6m}(r,\theta)$ as a perturbation from the trivial state, which we could call a `bright' solution, then $V_{6m}(r,\theta)$ can be thought of as an inverse perturbation from the localised radial spot, which we correspondingly call a `dark' solution.

\subsection{Small-Layer Patches}\label{subsec:SmallMatch}

Here we will provide the nontrivial solutions of the matching equations \eqref{GeneralMatch}, up to the symmetries described in the previous subsection, for small-layered patches. In particular, we will focus on the cases when $N \leq 4$. We note that for any $N,m \geq 1$ we have that $\mathbf{a} = \mathbf{0}$ and $\mathbf{a}=(1, 0 ,\dots)$ are solutions of $\mathbf{Q}_N^m(\mathbf{a}) = \mathbf{a}$, representing solutions $U_{m}(r,\theta) = \mathbf{0}$ and $U_{m}(r,\theta) = \mathbf{u}_{0}(r)$, which lead to the well-studied radially-symmetric solutions of the Galerkin system and are therefore deemed to be trivial solutions of the matching equations. With $N \leq 4$ the matching equations are of a low enough dimension that we can obtain nontrivial solutions of \eqref{GeneralMatch} explicitly. These investigations further give us information on the role of the lattice $m$ in finding localised $\mathbb{D}_{m}$ patches and the goal is to provide all solutions that are unique to the given value of $N$. That is, Lemma~\ref{lem:HarmQ} shows that solutions for $N = 1$ can be extended to solutions with $N = 2,3,4$ and solutions for $N = 2$ can be extended to solutions with $N = 4$; hence, we only detail the solutions for the smallest value of $N$ that they appear for. Similarly, Lemma~\ref{lem:EquiQ} demonstrates that solutions map into solutions by applying $\mathcal{R}$, thus giving an equivalence class of solutions. The propositions that follow in this subsection will only provide one representative from each equivalence class. 

Finally, for the ease of notation, we will define $C_{m}:=2\cos\left(\frac{m\pi}{3}\right)$, which gives 
\begin{equation}
    C_m = \left(-1\right)^{m}\begin{cases}
         2, & 3\mid m \\
         -1, & 3\nmid m. \\
    \end{cases}
\end{equation}
That is, if $m$ is a multiple of 6 then $C_m = 2$, if it is multiple of 3 and odd then $C_m = -2$, if it is even and not a multiple of 3 then $C_m = -1$, and if it is not divisible by 2 or 3 then $C_m = 1$.

We first present our results for $N = 1,2,3$ and $m \in \mathbb{N}$. The proof of this result is left to \S\ref{subsec:SmallProofs}.

\begin{prop}\label{prop:N123} 
	Fix $m \in \mathbb{N}$; then, up to the symmetries described by Lemmas~\ref{lem:EquiQ}, \ref{lem:HarmQ}, and \ref{lem:RecipQ}, the matching equation $\mathbf{Q}_N^m(\mathbf{a}) = \mathbf{a}$ has the following nontrivial, isolated solutions:
	\begin{enumerate}
		\item If $N = 1$ we have
			\[
				\mathbf{a}_1^1 = \bigg(\frac{1}{C_m},\sqrt{\frac{C_m - 1}{C_m^3}}\bigg).
			\]
			Importantly, there are no nontrivial solutions when $2 \nmid m$ and $3 \nmid m$.
			
			\item If $N = 2$ we have
			\[
				\mathbf{a}_j^2 = \bigg(\frac{1 - 2\lambda_j}{C_m},\sqrt{\big[(1 - (-1)^m) + 2(-1)^m\lambda_j\big]\lambda_j}, \lambda_j\bigg), \quad j = 1,2
			\]
			where the $\lambda_j$ are real roots of the $m$-dependent polynomial 
			\[
				p_m(\lambda) = (4 + 3(-1)^m C_m^3)\lambda^2 + (2C_m + (1 - (-1)^m)C_m^3 - 4)\lambda - (C_m - 1).
			\]
			Importantly, there are no nontrivial solutions when $2 \nmid m$ and $3 \nmid m$.
		\item If $N=3$ and $3\mid m$, we have
\[
		\mathbf{a}_0^3 = \bigg(\frac{(-1)^m}{2},\, \sqrt{\frac{2- (-1)^m}{10}}, \, 0,\, - \frac{(-1)^m}{2}\sqrt{\frac{2- (-1)^m}{10}} \bigg).    
		\]		
		
		 Furthermore, for all $m\in\mathbb{N}$, we have
\[
		\mathbf{a}_j^3 = \bigg(\lambda_j,\, \sqrt{q_1^m(\lambda_j)}, \, (1-(-1)^m2\lambda_j)q_2^m(\lambda_j),\, C_m q_2^m(\lambda_j)\sqrt{q_1^m(\lambda_j)} \bigg)    
		\]
		where
		    \[
		    \begin{split}
		    (q_1^m, q_2^m)(x) &= \begin{cases}
		            \big(\frac{10 (x-1)x + 3}{30}, \frac{110(x-1)x + 21}{6}\big), & 2\mid m, 3\mid m,\\
		            \big(\frac{-440x^3 - 170 x^2 + 38x  + 17}{50}, \frac{1760x^3+430x^2-142x-23}{30}\big), & 2\nmid m, 3\mid m,\\
		            \big(\frac{-5x^4 -559x^3 +3851x^2 -185x -204}{369}, \frac{463x^4 +53387x^3 -137638x^2 +25249x + 17931}{12546}\big), & 2\mid m, 3\nmid m,\\
		            \big(5x+3, x+2\big),& 2\nmid m, 3\nmid m,
		    \end{cases}
		    \end{split}\]
		    and $\lambda_j$ are real roots of the $m$-dependent polynomial
		    \[
		    p_m(x) = \begin{cases}
		            220(x-1)^2x^2+ 90 (x-1)x + 9, & 2\mid m, 3\mid m,\\
		            220x^4 + 40x^3 - 34x^2 - 2x + 1, & 2\nmid m, 3\mid m,\\
		            x^5 + 116x^4 - 217x^3 - 113x^2 + 24x + 9, & 2\mid m, 3\nmid m,\\
		            x^2+3x+1, & 2\nmid m, 3\nmid m,
		    \end{cases}
		\]
		Importantly, there are five nontrivial solutions when $3\mid m$ or $2\mid m$ (or both), and a single nontrivial solution when $2\nmid m$, $3\nmid m$.
	\end{enumerate}
\end{prop}

\begin{rmk}
	Let us briefly comment on the role of Lemmas~\ref{lem:EquiQ} and \ref{lem:HarmQ}. First, applying $\mathcal{R}$, as defined in Lemma~\ref{lem:EquiQ}, to a solution of the matching equations gives another solution. For example, with $N = 1$ above, we have the solution $\mathbf{a}_1^1$, and so we also have the solution $\mathcal{R}\mathbf{a}_1^1 = (C_m^{-1},-\sqrt{(C_m-1)C_{m}^{-3}})$. Furthermore, from Lemma~\ref{lem:HarmQ}, the solution $\mathbf{a}_1^1$ extends to the solutions $\mathcal{H}^{2}_{2}\mathbf{a}_{1}^{1}=(C_{2m}^{-1},0,\sqrt{(C_{2m}-1)C_{2m}^{-3}})$ for $N=2$, $\mathcal{H}^{2}_{3}\mathbf{a}_{1}^{1}=(C_{2m}^{-1},0,\sqrt{(C_{2m}-1)C_{2m}^{-3}},0)$ and $\mathcal{H}^{3}_{3}\mathbf{a}_{1}^{1}=(C_{3m}^{-1},0,0,\sqrt{(C_{3m}-1)C_{3m}^{-3}})$ for $N=3$, and so on. Hence, Proposition~\ref{prop:N123} lists all solutions up to these equivalences to significantly condense the notation.  
\end{rmk} 

We now turn our attention to the case when $N = 4$; here, we restrict ourselves to the case when $6 \mid m$ since the other cases are significantly more complicated. Then, the equation $\mathbf{a} = \mathbf{Q}_4^m(\mathbf{a})$ is given by
\begin{equation}
	\begin{split}
    		a_{0} &= a_{0}^{2} + 2a_{1}^{2} + 2a_{2}^{2} + 2a_{3}^{2} + 2a_{4}^{2},\\
    		a_{1} &= 2a_{0}a_{1} + 2a_{1}a_{2} + 2a_{2}a_{3} + 2a_{3}a_{4},\\
    		a_{2} &= a_{1}^{2} + 2a_{0}a_{2} + 2a_{1}a_{3} + 2a_{2}a_{4},\\
    		a_{3} &= 2a_{0}a_{3} + 2a_{1}a_{4} + 2a_{1}a_{2},\\
    		a_{4} &= a_{2}^{2} + 2a_{0}a_{4} + 2a_{1}a_{3}.
	\end{split}
\end{equation}
We summarise our findings with the following proposition whose proof is left to \S\ref{subsec:SmallProofs} and completes the proof of Theorem~\ref{thm:SmallPatch}.

\begin{prop}\label{prop:N4} 
	Fix $m \in \mathbb{N}$ such that $6 \mid m$. Up to the symmetries described by Lemmas~\ref{lem:EquiQ}, \ref{lem:HarmQ}, and \ref{lem:RecipQ}, the matching equation $\mathbf{Q}_4^m(\mathbf{a}) = \mathbf{a}$ has the following nontrivial, isolated solutions:
	\[
		\begin{split}
		\mathbf{a}_j^4 &= \bigg(\frac 1 2 - \frac{1}{2\sqrt{q_0(\lambda_j)}}, \frac{1}{2}\sqrt{\frac{q_1(\lambda_j)}{q_0(\lambda_j)}},\frac{1 - \lambda_j^2}{2\sqrt{q_0(\lambda_j)}},\frac{\lambda_j}{2}\sqrt{\frac{q_1(\lambda_j)}{q_0(\lambda_j)}},\frac{\lambda_j^2 + \lambda_j - 1}{2\sqrt{q_0(\lambda_j)}}\bigg), 
		\end{split}
	\]	
	with $j = 1,2,\dots,5$, where
	\[
		\begin{split}
			q_0(\lambda) &= 3\lambda^4 + 2\lambda^3 - 10\lambda^2 - 4\lambda + 7, \\
			q_1(\lambda) &= 6\lambda^6 + 4\lambda^5 - 10\lambda^4 - 12\lambda^2 - 12\lambda + 19,
		\end{split}
	\]	
	and the $\lambda_j$ are the real roots of the polynomial
	\[
		p(\lambda) = 6\lambda^5 + 5\lambda^4 - 20\lambda^3 - 12\lambda^2 + 12\lambda + 3.
	\]
\end{prop}
We note that the roots $\{\lambda_{j}\}_{j=1}^{5}$ are such that $\lambda_{1}<\frac{-1-\sqrt{5}}{2}<\lambda_{2}<-1<\lambda_{3}<0<\frac{-1+\sqrt{5}}{2}<\lambda_{4}<1<\lambda_{5}$, and so the ordered set of signs of the elements of $\mathbf{a}^{4}_{j}$ is distinct for each $j=1,\dots,5$.

\subsection{Large \texorpdfstring{$N$}{TEXT} Layer Patches with \texorpdfstring{$6 \mid m$}{TEXT}}\label{subsec:BigMatch}

Let us now focus exclusively on the case $6 \mid m$, representing solutions of the $N$-truncated Galerkin system \eqref{eqn:R-D;Galerk} with a $\mathbb{D}_{6d}$-symmetry, for some $d \geq 1$. With $6 \mid m$ we have that $\cos(nm\pi/3) = \cos(2n\pi) = 1$ for every integer $n$, and so the matching equations become
\begin{equation} \label{HexMatch}
	a_{n} = 2\sum_{j=1}^{N-n} a_{j} a_{n+j} + \sum_{j=0}^{n} a_{j}a_{n-j},
\end{equation}
for each $n = 0,1,\dots,N$. In this section we will investigate solutions to \eqref{HexMatch} with $N \gg 1$. To this end, this section will be dedicated to determining a continuous solution $\alpha(t)$ to the integral equation 
\begin{equation}\label{IntMatch}
	\alpha(t) = 2\int_0^{1 - t} \alpha(s)\alpha(s + t)\drm s + \int_0^t \alpha(s)\alpha(t-s)\drm s,
\end{equation}
and using this solution to demonstrate the existence of solutions to \eqref{HexMatch} when $N \gg 1$. Formally, one can see that upon defining $a_n = \frac{1}{N+1}\alpha(\frac{n}{N+1})$ for each $n = 0,1,\dots, N$, \eqref{HexMatch} takes the form
\begin{equation}\label{HexMatchRiemann}
	\alpha\bigg(\frac{n}{N+1}\bigg) = \frac{2}{N+1}\sum_{j=1}^{N-n} \alpha\bigg(\frac{j}{N+1}\bigg) \alpha\bigg(\frac{n + j}{N+1}\bigg) + \frac{1}{N+1}\sum_{j=0}^{n} \alpha\bigg(\frac{j}{N+1}\bigg)\alpha\bigg(\frac{n - j}{N+1}\bigg),	
\end{equation}
which can be interpreted as a Riemann sum approximation of \eqref{IntMatch}. In what follows we will make the correspondence between solutions of \eqref{IntMatch} and solutions of \eqref{HexMatch} with $N \gg 1$ explicit.

To begin, we will consider the closure of the space of real-valued step functions on $[0,1]$, denoted $X$, with respect to the supremum norm
\begin{equation}
	\|\alpha\|_\infty := \sup_{t \in [0,1]} |\alpha(t)|,
\end{equation} 
for each $\alpha \in X$. Since we are taking the closure, $X$ is a Banach space, and the elements of the space $X$ are referred to as {\em regulated functions}. Regulated functions can be characterised by the fact that each $\alpha \in X$ and all $t \in [0,1]$ we have that the left and right limits $\alpha(t^-)$ and $\alpha(t^+)$ exist, aside from $\alpha(0^-)$ and $\alpha(1^+)$ \cite{RegulatedFns}. We note that the space of continuous functions on $[0,1]$, denoted $C([0,1])$, is a closed subspace of $X$. Then, the matching equation \eqref{IntMatch} can be compactly stated as the fixed point problem $\alpha = Q_\infty(\alpha)$, where we define the nonlinear operator $Q_\infty : X \to X$ by
\begin{equation}
	Q_\infty(\alpha) := 2\int_0^{1-t} \alpha(s)\alpha(s +t)\drm s + \int_0^t \alpha(s)\alpha(t-s)\drm s,
\end{equation}
for all $\alpha \in X$. Using the fact that elements of $X$ are integrable, one finds that $Q_\infty$ is well-defined, and $Q_\infty$ maps the closed subspace $C([0,1])$ back into itself. We present the following theorem.

\begin{thm}\label{prop:IntExist} 
	There exists $\alpha^* \in X$ such that $\alpha^* = Q_\infty(\alpha^*)$. Furthermore, $\alpha^*$ is continuous and positive on $[0,1]$, and the linearisation of $I - Q_\infty$ about $\alpha^*$ is invertible as an operator from $X \to X$. 
\end{thm}   

\begin{proof}
The existence of $\alpha^*\in X$ is achieved by a computer-assisted proof that follows the work of \cite{vandenberg2015Rigorous}, with the details of the implementation left to Appendix~\ref{app:CompProof}. The general idea is to identify the existence of a fixed point of $Q_\infty(\alpha)$ near a continuous piecewise linear approximation of the solution generated by solving \eqref{HexMatchRiemann} numerically. We use interval arithmetic to verify a number of bounds that give way to the fact that $Q_\infty$ is contracting in a neighbourhood of our constructed approximate solution. With this, the Banach fixed point theorem thus gives that generating a sequence from any element in a neighbourhood of our approximate solution by continually applying $Q_\infty$ converges to a locally unique fixed point, giving the existence of $\alpha^*$. 

Now, since $Q_\infty$ maps $C([0,1])$ into $C([0,1])$, initiating the sequence with our continuous linear approximation of the solution gives that each successive iterate is again continuous. Since $C([0,1])$ is a closed subspace of $X$, it follows that the locally unique fixed point is also continuous since initiating the sequence with our approximate solution would make it the limit in the $\|\cdot\|_\infty$ norm of a sequence of continuous functions. The non-degeneracy of this fixed point is a consequence of the local uniqueness of the solution and is proven by following the arguments in \cite[Lemma~5.2]{vandenberg2015Rigorous}. The proof of positivity comes from the fact that our constructed numerical approximation is positive and the fixed point is close to it, as proven in \cite[Lemma~4.3]{vandenberg2015Rigorous}. This concludes the proof. 
\end{proof} 

Theorem~\ref{prop:IntExist} gives the existence of a fixed point $\alpha^*$ of $Q_\infty$, which can equivalently be viewed as a root of the nonlinear function $G(\alpha) := \alpha - Q_\infty(\alpha)$. Furthermore, the Fr\'echet derivative of $G$ about $\alpha^*$, denoted $DG(\alpha^*) : X \to X$, is the bounded linear operator acting on $f \in X$ by 
\begin{equation}
	[DG(\alpha^*)f](t) = f(t) - 2\int_0^{1 - t} \alpha^*(s)f(t+s)\drm s - 2\int_0^{1 - t} f(s)\alpha^*(t+s)\drm s - \int_0^t \alpha^*(s)f(t-s)\drm s - \int_0^t f(s)\alpha^*(t-s)\drm s,
\end{equation} 
for which Theorem~\ref{prop:IntExist} gives that this operator is invertible on $X$. Similarly, we introduce the functions $\mathbf{F}_N:\R^N \to \R^N$ for each $N \geq 1$ by 
\begin{equation}\label{Fmap}
	\mathbf{F}_N(a)_n = a_n - \frac{2}{N+1}\sum_{j = 1}^{N-n} a_ja_{n+j} - \frac{1}{N+1}\sum_{j = 0}^n a_ja_{n-j}, \quad n = 0,1,\dots,N.
\end{equation} 
That is, $\mathbf{F}_N(\mathbf{a}) = \mathbf{a} - (N+1)\mathbf{Q}_N^m((N+1)^{-1}\mathbf{a})$, so that roots of $\mathbf{F}_N$ are exactly solutions of \eqref{HexMatchRiemann}. In the remainder of this section we will prove the following proposition, which details that the fixed point $\alpha^*$ of $Q_\infty$ can be used to closely approximate roots of $F_N$ for sufficiently large $N \geq 1$, which themselves lie in one-to-one correspondence with solutions of \eqref{HexMatch} after rescaling $\mathbf{a} \mapsto (N+1)^{-1}\mathbf{a}$ for each $N \geq 1$. In the remainder of this section we will take all finite-dimensional vector norms to be the maximum norm, or $\infty$-norm, which returns the maximal element in absolute value of the vector in analogy with the norm on $X$. To emphasise the difference between finite- and infinite-dimensional vector norms, the norm on $X$ will keep the $\infty$ subscript, while the norms on finite-dimensional spaces will not have a subscript.

\begin{thm}\label{thm:Inf2Finite}
	Let $\alpha^*$ be as given in Theorem~\ref{prop:IntExist}. Then, there exists $N_0 \geq 1$ such that for all $N \geq N_0$ there exists a positive vector $\mathbf{a}^*_N \in \R^{N+1}$ such that $\mathbf{F}_N(\mathbf{a}^*_N) = 0$. Moreover, $D\mathbf{F}_N(\mathbf{a}^*)$ is invertible for all $N \geq N_0$ and, defining $\mathbf{a}_N = [\alpha^*(0),\alpha^*(\frac{1}{N+1})), \dots, \alpha^*(\frac{N}{N+1})]\in \R^{N+1}$, we have that $\|\mathbf{a}^*_N - \mathbf{a}_N\| \to 0$ as $N \to \infty$.  
\end{thm}

To prove Theorem~\ref{thm:Inf2Finite} we will make use of the following lemma, which is a variant of the Newton--Kantorovich theorem, coming from \cite{Bramburger2019localizedRolls}.

\begin{lem}\label{lem:Roots} 
	If $\mathbf{F}:\R^d \to \R^d$ is smooth and there are constants $0 < \kappa < 1$ and $\rho > 0$, a vector $\mathbf{w}_0 \in \R^d$, and an invertible matrix $\mathbf{A} \in \R^{d\times d}$ so that 
	\begin{enumerate}
		\item $\|1 - \mathbf{A}^{-1}D\mathbf{F}(\mathbf{w})\| \leq \kappa$ for all $\mathbf{w} \in B_\rho(\mathbf{w}_0) := \{\mathbf{w}\in\R^d|\ |\mathbf{w} - \mathbf{w}_0| < \rho\}$, and
		\item $\|\mathbf{A}^{-1}\mathbf{F}(\mathbf{w}_0)\| \leq (1 - \kappa)\rho$,
	\end{enumerate}
	then $F$ has a unique root $\mathbf{w}_*$ in $B_\rho(\mathbf{w}_0)$, and this root satisfies $|\mathbf{w}_* - \mathbf{w}_0|\leq \frac{1}{1 - \kappa}\|\mathbf{A}^{-1}\mathbf{F}(\mathbf{w}_0)\|$. 
\end{lem}

We now present the following lemmas which will then allow for the application of Lemma~\ref{lem:Roots} to arrive at the proof of Theorem~\ref{thm:Inf2Finite}.

\begin{lem}\label{lem:Inf2Finite1} 
	Let $\mathbf{a}_N$ be as in Theorem~\ref{thm:Inf2Finite}. Then, for any $\varepsilon > 0$, there exists $N_1 \geq 1$ such that $\|\mathbf{F}_N(\mathbf{a}_N)\| < \varepsilon$ for all $N \geq N_1$.	
\end{lem}

\begin{proof}
	First note that we can equivalently write $\mathbf{F}_N$ in \eqref{Fmap} as
	\begin{equation}\label{Fmap2}
		\mathbf{F}_N(a)_n = a_n - \frac{2}{N+1}a_0a_n - \frac{2}{N+1}\sum_{j = 0}^{N-n} a_ja_{n+j} - \frac{1}{N+1}\sum_{j = 0}^n a_ja_{n-j}, 
	\end{equation} 
	for each $n = 0,1,\dots,N$, by simply adding and subtracting the $j = 0$ term from the first sum. Then, the summations in $\mathbf{F}_N$ can be seen as left Riemann sums for the integrals in $G$. Since $\alpha^*(t)$ is continuous at all $t \in [0,1]$, it follows that $\alpha^*$ is bounded and uniformly continuous, thus giving that the error between the integrals and the Riemann sums in $\mathbf{F}_N$ are bounded uniformly for $t \in [0,1]$. The rate of convergence depends only on the modulus of continuity of $\alpha^*$, which is fixed and independent of $N \geq 1$. Since $G(\alpha^*) = 0$, this means that the terms 
	\begin{equation}
		\bigg|a_n - \frac{2}{N+1}\sum_{j = 0}^{N-n} a_ja_{n+j} - \frac{1}{N+1}\sum_{j = 0}^n a_ja_{n-j}\bigg|
	\end{equation}
	can be made arbitrarily small, uniformly in $n = 0,1,\dots,N$. Finally, the terms $2a_0a_n/(N+1)$ that appear in \eqref{Fmap2} can be bounded by $2\|\alpha^*\|_\infty/(N+1)$, which converges to $0$ as $N \to \infty$. Hence, the triangle inequality gives us that by taking $N$ sufficiently large we can make $\|\mathbf{F}_N(\mathbf{a}_N)\|$ as small as we wish. This therefore proves the lemma.
\end{proof}

\begin{lem}\label{lem:Inf2Finite2} 
	Let $\mathbf{a}_N$ be as in Theorem~\ref{thm:Inf2Finite}. Then, there exist $C > 0$ and $N_2 \geq 1$ for which the Jacobian matrices of $\mathbf{F}_N$ evaluated at $\mathbf{a}_N$, denoted $D\mathbf{F}_N(\mathbf{a}_N) \in \R^{N\times N}$, is invertible with $\|D\mathbf{F}_N(\mathbf{a}_N)^{-1}\| \leq C$ for all $N \geq N_2$.	
\end{lem}

\begin{proof}
	Let us start by showing that for $N$ sufficiently large $D\mathbf{F}_N(\mathbf{a}_N) $ is invertible. We will assume the contrary to arrive at a contradiction. Hence, we may assume that there exists an infinite subsequence of matrices $\{D\mathbf{F}_{N_k}(\mathbf{a}_{N_k})\}_{k = 1}^\infty$ which are all not invertible. For each $k \geq 1$ let us denote $\mathbf{v}_k = [v_1,v_2,\dots,v_{N_k}]^T \in \R^{N_k}$ as a nontrivial element of the kernel of $D\mathbf{F}_{N_k}(\mathbf{a}_{N_k})$ with the property that $\|\mathbf{v}_k\|_\infty := \max |v_n| = 1$. We will show that this implies that $DG(\alpha^*)$ is not invertible, which from Theorem~\ref{prop:IntExist} is a contradiction.   
	
	We may consider the elements $\mathbf{a}_{N_k},\mathbf{v}_k \in \R^{N_k}$ as step functions in $X$ which take the value $\alpha^*(\frac{n}{N_k + 1})$ and $v_{n+1}$, respectively, on the interval $[\frac{n}{N_k+1},\frac{n+1}{N_k+1})$, for each $n = 0,\dots,N_k$. By abuse of notation we will again denote these step functions as $\mathbf{a}_{N_k},\mathbf{v}_k \in X$ for each $k \geq 1$. By assumption we have that $\|\mathbf{v}_k\|_\infty = 1$ as an element in $X$ as well. Furthermore, since the elements $\mathbf{a}_{N_k}$ and $\mathbf{v}_k$ are constant on the $N_k$ intervals of length $1/(N_k+1)$ we have that the integrals in $DG(\mathbf{a}_{N_k})$ reduce to discrete sums so that for each $t \in [\frac{n}{N_k+1},\frac{n+1}{N_k+1})$, using \eqref{Fmap2}, we have 
	\begin{equation}\label{FmapDerivative}
		\begin{split}
		(DG(\mathbf{a}_{N_k})\mathbf{v}_k)(t) &= \underbrace{[D\mathbf{F}_{N_k}(\mathbf{a}_{N_k})\mathbf{v}_k]_n}_{= 0} - \frac{2\alpha^*(0)v_n + 2v_0\alpha^*(\frac{n}{N_k+1})}{N_k+1} \\
		&= -\frac{2\alpha^*(0)v_n + 2v_0\alpha^*(\frac{n}{N_k+1})}{N_k+1},
		\end{split}
	\end{equation}
	for all $n = 0,1,\dots,N_k$ and $k \geq 1$. We note that the remaining term in the above equation comes from the $j = 0$ terms being added and subtracted from the first summation in the definition of each component of $\mathbf{F}_N$. Therefore, for each $k \geq 1$ we have
	\begin{equation}\label{WeylBnd}
		\begin{split}
			\|DG(\alpha^*)\mathbf{v}_k\|_\infty &\leq \|DG(\alpha^*)\mathbf{v}_k - DG(\mathbf{a}_{N_k})\mathbf{v}_k\|_\infty + \|DG(\mathbf{a}_{N_k})\mathbf{v}_k\|_\infty  \\ 
			&\leq \|DG(\alpha^*)\mathbf{v}_k - DG(\mathbf{a}_{N_k})\mathbf{v}_k\|_\infty + \frac{2}{N_k+1} \sup_{t \in [0,1]} |\alpha^*(t)\mathbf{v}_k(t)|\\ 
			&= \sup_{t \in [0,1]} \bigg| -2\int_0^{1 - t} [\alpha^*(s) - \mathbf{a}_{N_k}(s)]\mathbf{v}_k(t+s)\drm s \\ 
			& \quad \quad - 2\int_0^{1 - t} \mathbf{v}_k(s)[\alpha^*(t+s) - \mathbf{a}_{N_k}(t+s)]\drm s - \int_0^t [\alpha^*(s) - \mathbf{a}_{N_k}(s)]\mathbf{v}_k(t-s)\drm s \\ &\quad \quad - \int_0^t \mathbf{v}_k(s)[\alpha^*(t-s) - \mathbf{a}_{N_k}(t-s)]\drm s\bigg| + \frac{2\|\alpha^*\|_\infty\|\mathbf{v}_k\|_\infty}{N+1} \\
			&\leq \frac{2\|\alpha^*\|_\infty}{N_k+1} +  6 \|\mathbf{v}_k\|_\infty \sup_{n \in [0,N_k],\ t \in [\frac{n}{N_k+1},\frac{n+1}{N_k+1})} \bigg|\alpha^*(t) - \alpha^*\bigg(\frac{n}{N_k + 1}\bigg)\bigg|\cdot\int_0^1 \drm s \\
			&= \frac{2\|\alpha^*\|_\infty}{N_k+1} +  6\sup_{n \in [0,N_k],\ t \in [\frac{n}{N_k+1},\frac{n+1}{N_k+1})} \bigg|\alpha^*(t) - \alpha^*\bigg(\frac{n}{N_k + 1}\bigg)\bigg|,
		\end{split}
	\end{equation}
	where we have used the fact that $\mathbf{a}_{N_k}(t)$ as a function in $X$ takes the value $\alpha^*(n/(N_k + 1))$ for each $t \in [\frac{n}{N_k+1},\frac{n+1}{N_k+1})$. Since $\alpha^*(t)$ is uniformly continuous, it follows that
	\begin{equation}
		\sup_{n \in [0,N_k],\ t \in [\frac{n}{N_k+1},\frac{n+1}{N_k+1})} \bigg|\alpha^*(t) - \alpha^*\bigg(\frac{n}{N_k + 1}\bigg)\bigg| \to 0
	\end{equation}
	as $k \to \infty$, giving that 
	\begin{equation}
		\lim_{k \to \infty} \|DG(\bar a)v_k\|_\infty = 0,
	\end{equation} 
	since $N_k \to \infty$ as $k \to \infty$. We therefore have that the sequence $\{\mathbf{v}_k\}_{k = 1}^\infty \subset X$ is a Weyl sequence for the linear operator $DG(\alpha^*)$, giving that $DG(\alpha^*)$ is not invertible. This is a contradiction, thus proving that for $N$ sufficiently large we have that $D\mathbf{F}_N(\mathbf{a}_N)$ is invertible.  
	
	To show that the operator norm of $D\mathbf{F}_N(\mathbf{a}_N)^{-1}$ is uniformly bounded for sufficiently large $N$, we proceed with a nearly identical argument to that which was performed above to prove invertibility. That is, we may assume the contrary, thus giving way to a subsequence $\{D\mathbf{F}_N(\mathbf{a}_N)\}_{k = 1}^\infty$ for which there exists vectors $\mathbf{v}_k \in \R^{N_k+1}$ with $\|\mathbf{v}_k\|_\infty = 1$ such that $\|D\mathbf{F}_N(\mathbf{a}_N)\mathbf{v}_k\| \to 0$ as $k \to \infty$. Arguing as above, we can find that $\|DG(\mathbf{a}_{N_k})\mathbf{v}_k\|_\infty \to 0$ as $k \to \infty$, after extending the elements $\mathbf{a}_{N_k}$ and $\mathbf{v}_k$ to step functions in $X$. Indeed, following as in \eqref{WeylBnd} we get 
	\begin{equation}
		\|DG(\alpha^*)\mathbf{v}_k\|_\infty \leq 6\sup_{n \in [0,N_k],\ t \in [\frac{n}{N_k+1},\frac{n+1}{N_k+1})} \bigg|\alpha^*(t) - \alpha^*\bigg(\frac{n}{N_k + 1}\bigg)\bigg| + \|DG(\mathbf{a}_{N_k})\mathbf{v}_k\|_\infty. 	
	\end{equation} 
	The rightmost term $\|DG(\mathbf{a}_{N_k})\mathbf{v}_k\|_\infty$ can be expanded as in \eqref{FmapDerivative}, where it can again be shown that it converges to zero as $k \to \infty$ by following as in \eqref{WeylBnd} and replacing the condition that $D\mathbf{F}_N(\mathbf{a}_N)\mathbf{v}_k = 0$ with $\|D\mathbf{F}_N(\mathbf{a}_N)\mathbf{v}_k\| \to 0$. Therefore, we again find that $\{\mathbf{v}_k\}_{k = 1}^\infty \subset X$ is a Weyl sequence for $DG(\alpha^*)$, contradicting the result of Theorem~\ref{prop:IntExist} which details that $DG(\alpha^*)$ is invertible. This contradictions means that we must have that the operator norm of $D\mathbf{F}_N(\mathbf{a}_N)^{-1}$ is uniformly bounded for sufficiently large $N$. This concludes the proof. 
\end{proof} 

We now use the preceding lemmas to prove Theorem~\ref{thm:Inf2Finite}.

\begin{proof}[Proof of Theorem~\ref{thm:Inf2Finite}]
	To prove this result we will apply Lemma~\ref{lem:Roots}. Let us fix $\kappa = \frac{1}{2}$. From the continuity of $D\mathbf{F}_N(\mathbf{a})$ near $\mathbf{a}_N$, we may evoke Lemma~\ref{lem:Inf2Finite1} to take $N$ sufficiently large to guarantee the existence of a fixed $\rho > 0$ sufficiently small to guarantee that 
	\begin{equation}
		\|1 - D\mathbf{F}_N(\mathbf{a}_N)^{-1}D\mathbf{F}_N(a)\| \leq \frac{1}{2}
	\end{equation}	
	for all $\mathbf{a} \in B_\rho(\mathbf{a}_N)$. This therefore satisfies condition (1) to apply Lemma~\ref{lem:Roots}. Then, Lemma~\ref{lem:Inf2Finite2} gives that there exists a $C > 0$ so that $\|D\mathbf{F}_N(\mathbf{a}_N)^{-1}\| \leq C$ for all $N \geq N_2$, and so 
	\begin{equation}
		\|D\mathbf{F}_N(\mathbf{a}_N)^{-1}\mathbf{F}_N(\mathbf{a}_N)\| \leq C\|\mathbf{F}_N(\mathbf{a}_N)\|.
	\end{equation}
	Since $\|\mathbf{F}_N(\mathbf{a}_N)\| \to 0$ as $N \to \infty$, we may consider $N$ sufficiently large to guarantee that 
	\begin{equation}
		\|D\mathbf{F}_N(\mathbf{a}_N)^{-1}\mathbf{F}_N(\mathbf{a}_N)\| \leq \frac{\rho}{2},
	\end{equation} 
	thus satisfying condition (2) from Lemma~\ref{lem:Roots}. Hence, we have satisfied the conditions to apply Lemma~\ref{lem:Roots}, and so for all sufficiently large $N$ we find there exists $\mathbf{a}^*_N \in \R^{N+1}$ such that $\mathbf{F}_N(\mathbf{a}^*_N) = 0$ and 
	\begin{equation}
		\|\mathbf{a}^*_N - \mathbf{a}_N\| \leq 2\|D\mathbf{F}_N(\mathbf{a}_N)^{-1}\mathbf{F}_N(\mathbf{a}_N)\|.
	\end{equation}  
	Since $\|D\mathbf{F}_N(\mathbf{a}_N)^{-1}\mathbf{F}_N(\mathbf{a}_N)\| \to 0$ as $N \to \infty$, it follows that $\|\mathbf{a}^*_N - \mathbf{a}_N\| \to 0$ as $N \to \infty$. Finally, the non-degeneracy condition $D\mathbf{F}_N(\mathbf{a}^*_N)$ is invertible is satisfied for sufficiently large $N$ since these matrices can be made arbitrarily close to the uniformly invertible matrices $D\mathbf{F}_N(\mathbf{a}_N)$ by taking $N$ large enough. This therefore concludes the proof.   
\end{proof} 

%
%
%
%
%
%
%
%
%
%
\section{Conclusion}\label{s:Conclusion}

In this paper, we have investigated a finite mode truncation of planar two-component RD systems in order to understand the existence of localised dihedral patches. In particular, we showed that through an application of radial normal form theory we could establish precise conditions that guarantee the existence of localised solutions to the reduced Galerkin system \eqref{eqn:R-D;Galerk} in a neighbourhood of a Turing bifurcation. These conditions came in the form of a nonlinear algebraic condition, which we were able to solve by hand for small finite mode decompositions. This uncovered new (approximate) localised patterns that can be used to initialise numerical continuation schemes to trace their corresponding bifurcation curves in parameter space. We further demonstrated the existence of localised $\mathbb{D}_{6m}$, $m \in\mathbb{N}$, patches with $N \gg 1$ that bifurcate off the trivial state at the Turing instability point. 

Let us now briefly describe what happens in the limit $N\rightarrow\infty$ in our analysis. We find from our `Core - Transition - Rescaling' approach that the core boundary $r_{0}$ must be chosen such that $(mN)^2 \ll r_{0} < \mu^{-\frac{1}{2}}$. Although we only provide the existence of a $\mu_0,r_0 > 0$ in Theorems~\ref{thm:SmallPatch} and \ref{thm:BigPatch}, we see from the above that our analysis requires that $\mu_0 \to 0^+$ and $r_0 \to \infty$ as $N \to \infty$. This means that as $N$ becomes large, the region of validity for our analysis shrinks, thus giving no insight into the existence of these patterns in the case of a full Fourier decomposition, i.e. $N = \infty$.

 It seems natural to use polar coordinates when studying fully localised patterns, especially when using ideas from spatial dynamics, but this creates several additional problems in our analysis. First, to the author's knowledge, there is no established theory for a global reduction of a radial PDE system to a Ginzburg--Landau-type amplitude equation. By using the `Core - Far-field' matching from radial spatial dynamics, we sidestep this issue and only locally reduce to a complex amplitude in the far-field through our normal form analysis. To do this, we must remove the oscillatory phase of our solutions, which then introduces the above conditions for $r_0$ and $\mu$. Another issue with using polar coordinates, which makes a direct proof even more difficult, is that the PDE system inherently becomes a small divisor problem due to the $\tfrac{1}{r^2}\partial_{\theta\theta}$ term in the Laplace operator, $\Delta$. As a result, many of the analytical issues in this problem are reminiscent of recent work on small divisors for quasipatterns \cite{Braaksma2017quasi,Iooss2010quasi}, where the region of validity also shrinks to zero as the level of approximation improves, and the wider literature in this area might provide a possible solution to our problem.

Despite these analytical issues, our numerical results lead to the conjecture that for any $\mu>0$ taken sufficiently small we expect that true localised dihedral solutions to \eqref{eqn:R-D}, coming from the limit $N \to \infty$, exist and bifurcate from the Turing instability point. In order to establish this existence and persistence of the localised states found in the finite-mode truncation in full planar planar RD systems one could potentially apply similar computer-assisted proofs to \cite{Berg2020,vandenberg2015Rigorous} that are initialised with our finite mode solutions. Extending these finite mode solutions to true solutions of RD systems is an important area of future work that we hope to report on in a follow-up investigation.

\begin{figure}
    \centering
    \includegraphics[width=\linewidth]{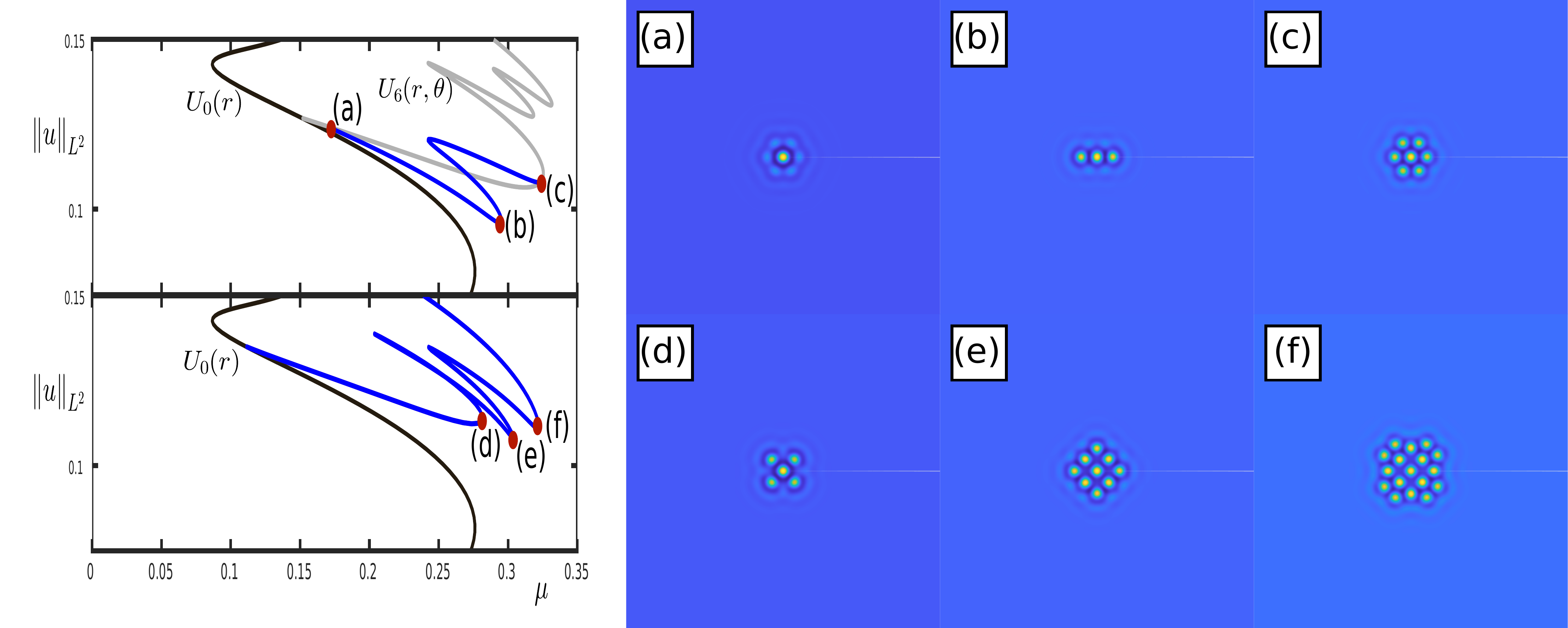}
    \caption{Localised patterns can further be found bifurcating from curves of localised hexagons $U_6(r,\theta)$ (light grey) and the localised radial spot $U_0(r)$ (black) in the Galerkin Swift--Hohenberg equation with $N=10$ and $\gamma=1.6$. Panels (a)-(c) and (d)-(f) show contour plots of localised $\mathbb{D}_2$ and $\mathbb{D}_4$ patches, respectively, at various points (red circles) on the (top) $\mathbb{D}_2$ and (bottom) $\mathbb{D}_4$ solution curves.}
    \label{fig:N2Solutions}
\end{figure}

Although we have focused on parameter values in a neighbourhood of a Turing instability, numerical investigations have revealed that localised dihedral patterns can further be found bifurcating from curves of localised solutions away from the Turing point. For example, Figure~\ref{fig:N2Solutions} presents numerical continuations of localised $\mathbb{D}_2$ and $\mathbb{D}_4$ patches bifurcating from the spot A and localised hexagon curves. Such branch points manifest themselves as symmetry-breaking bifurcations from the main localised solution curves. It may be possible to characterise these bifurcations using techniques from symmetric bifurcation theory, but predicting where these symmetry-breaking bifurcations take place along the curves remains an open problem.  

Beyond just extending our approximate solutions to true solutions of RD PDEs, this work opens up several interesting future research directions. First, note that we have only investigated the simplest possible type of localised radial solutions in \eqref{eqn:R-D;Galerk}. We expect there to also exist the equivalent of radial spot B and radial ring solutions that have been established for the SHE \cite{mccalla2013spots,lloyd2009localized}. It would be interesting to see whether the existence of spot B or ring localised dihedral solutions reduces to satisfying a matching condition akin to \eqref{MatchEq}. Second, the investigation of our spot A solutions could potentially be extended to fluid dynamic systems such as the ferrofluid problem~\cite{lloyd2015homoclinic,Hill2020Localised}, integral equations~\cite{rankin2014,faye2013localized}, and more general, multi-component reaction diffusion systems~\cite{scheel2003radially} where a centre-manifold reduction will have to be first carried out. We emphasise that while we have shown that the matching condition \eqref{MatchEq} turns up in the wide class of RD systems satisfying our hypotheses, it may be possible to have other matching conditions for different systems. 

We recall that the $\mathcal{H}^{i}_{N}$ transformation introduced in Lemma~\ref{lem:HarmQ} captures solutions for a system with $(m,N)=(i m_0, \lfloor\tfrac{N_0}{i}\rfloor)$ and embeds them in a larger system where $(m,N)=(m_0, N_0)$. Hence, rather than considering patterns with $\mathbb{D}_{m}$ dihedral symmetry for a fixed parameter $m\in\mathbb{N}$, one could set $m=1$ and then recover all $\mathbb{D}_{i}$ solutions from $\mathcal{H}^{i}_{N}$, for each respective $i\in\mathbb{N}$. This would avoid the issue of having to remove certain solutions from our low truncation results, but also obscure any connections between solutions with different values of $m$. For example, it may not be clear that the matching problem reduces to the discrete cases of $\{(2\mid m, 3\mid m), (2\nmid m, 3\mid m), (2\mid m, 3\nmid m), (2\nmid m, 3\nmid m)\}$. Beyond dihedral symmetry, one could also introduce a general truncated Fourier expansion
\begin{equation}
    \mathbf{u}(r,\theta) = \mathbf{u}_{0}(r) + 2\sum_{n=1}^{N} \left\{ \mathbf{u}_{n}(r)\cos(n\theta) + \mathbf{v}_{n}(r)\sin(n\theta)\right\},
\end{equation}
and follow the same procedure presented above. The Galerkin dynamical system would be solvable using an identical approach to the dihedral patterns studied in this paper, however it is unclear whether one could solve the resultant matching equations without fixing additional restrictions on the solution.

Finally, one could follow our ideas to examine the existence of localised structures in three spatial dimensions. In this case our Fourier decomposition would have to be changed to a spherical or cylindrical harmonic decomposition, which may provide insight into pattern formation in phase field crystal models~\cite{Subramanian2018localizedPFC}. In summary, with this work we have opened the door to the investigation of existence and persistence of localised solutions beyond just those in one spatial dimension or those that are axisymmetric.

\section*{Code availability}
Codes used to produce the results in this paper are available at: 

\href{https://github.com/Dan-Hill95/Localised-Dihedral-Patterns}{https://github.com/Dan-Hill95/Localised-Dihedral-Patterns}.

\section*{Data availability statement}
All data used to produce the results in this paper will be made available upon reasonable request.

\section*{Acknowledgements}

The authors would like to thank Hannes Uecker for helpful comments and questions regarding this work, as well as Tom Bridges and Bj\"orn Sandstede for their comments on an earlier draft of this manuscript. DJH acknowledges the EPSRC whose Doctoral Training Partnership Grant (EP/N509772/1) helped fund his PhD. JJB gratefully acknowledges support from the Institute of Advanced Studies.

\appendix

\setcounter{equation}{0}
\renewcommand\theequation{\Alph{section}.\arabic{equation}}

%
%
%
%
%
%
%
%
%
%
\section{Numerical Implementation}\label{s:num;impl}

We now briefly describe our numerical procedure for finding localised $\mathbb{D}_{m}$ solutions to the $(N+1)$-dimensional Galerkin system \eqref{eqn:R-D;Galerk}, given some fixed $m,N\in\mathbb{N}$. Recall that all computations use the SHE \eqref{e:SH} and so we precisely target our discussion on the Galerkin system derived from it. We discretise the radial variable $r$ using finite difference methods and fix $r_* > 0$ so that we define our radial domain to be $0 \leq r \leq r_{*}$. The result is a decomposition of the domain $r \in [0,r_*]$ into $T$ mesh points $\{r_{i}\}_{i=1}^{T}$. In order for the operator $(1+\Delta_{n})^2$ to be well-defined in \eqref{Delta_n}, the Galerkin system is also equipped with the following boundary conditions
\begin{equation}\label{SH:Num;bd}
    \begin{cases}
        \partial_r u_n = (mn) u_n = 0, & r=0,\\
        \partial_r \Delta_n u_n = (mn) \Delta_n u_n = 0, & r=0,\\
         \partial_r u_n = \partial_r \Delta_n u_n = 0, & r=r_*,
    \end{cases}
\end{equation}
for each $n\in[0,N]$. We introduce differentiation matrices $\mathbf{D}_{1}$, $\mathbf{D}_{2}$ using the central finite difference formula for first and second order differentiation, respectively. By defining the matrix $\mathbf{R}$ as
\begin{align}
    \mathbf{R}_{i,j} &= \left\{\begin{array}{cc}
        1, & i=j=0, \\
        \frac{1}{r_{i}}, & i=j\neq0, \\
         0, & i\neq j,
    \end{array}\right.\nonumber 
\end{align}
we can express $\Delta_{n}$ as the finite-difference matrix $\mathscr{D}_{n}$, where
\begin{align}
    \mathscr{D}_{n} &= \mathbf{D}_{2} + \mathbf{R}\mathbf{D}_{1} - \left(m n\right)^{2}\mathbf{R}^{2}.\nonumber
\end{align}
Here the first and last rows of the matrices $\mathbf{D}_{1}$, $\mathbf{D}_{2}$, and $\mathbf{R}$ have been chosen such that \eqref{SH:Num;bd} is satisfied. Then, defining $\mathbf{v}_{i}:\left[\mathbf{v}_{i}\right]_{j}=u_{i}(r_{j})$, the Galerkin SHE becomes $\mathbf{G}(V;\mu)=\mathbf{0}$, defined as
\begin{align}
    \left[\mathbf{G}(V;\mu)\right]_{n} &:= -\left[\left( \mathbbm{1}_{T} + \mathscr{D}_{n}\right)^{2} + \mu\mathbbm{1}_{T}\right]\mathbf{v}_{n} + \gamma\sum_{i + j=n}  \mathbf{v}_{|i|}\circ\mathbf{v}_{|j|} - \sum_{i + j + \ell = n} \mathbf{v}_{|i|}\circ\mathbf{v}_{|j|}\circ\mathbf{v}_{|\ell|},\label{SH:Num;finite}
\end{align}
for all $n\in[0,N]$, where $\circ$ denotes the \textit{Hadamard}, or element-wise, product and $V=(\mathbf{v}_{0},\dots,\mathbf{v}_{N})\in\mathbb{R}^{T(N+1)}$. For computational speed, we do not evaluate \eqref{SH:Num;finite} in this form; rather, we introduce the block matrices $\mathbf{C}_{1},\mathbf{C}_{2}\in\mathbb{R}^{T(N+1)\times T(N+1)}$ defined as
\begin{align}
    \mathbf{C}_{1}(V) &= \begin{pmatrix} 
    \mathcal{D}(\mathbf{v}_{0})   &  &       &       \\
    \mathcal{D}(\mathbf{v}_{1})   & \mathcal{D}(\mathbf{v}_{0}) &  &       \\
    \vdots  &\ddots &\ddots &  \\
    \mathcal{D}(\mathbf{v}_{N})   &\dots &\dots &\mathcal{D}(\mathbf{v}_{0})  
    \end{pmatrix}, \qquad  \mathbf{C}_{2}(V) = \begin{pmatrix} 
    \mathcal{D}(\mathbf{v}_{N})   &\dots &\dots &\mathcal{D}(\mathbf{v}_{0})  \\
     &\ddots &\ddots &\mathcal{D}(\mathbf{v}_{1})  \\
            &  &\ddots &\vdots \\
            &       &  &\mathcal{D}(\mathbf{v}_{N})  
    \end{pmatrix},\nonumber
\end{align}
for $V=(\mathbf{v}_{0},\dots,\mathbf{v}_{N})\in\mathbb{R}^{T(N+1)}$, where $\mathcal{D}(\mathbf{v}) := \textrm{diag}(v_{0},\dots,v_{T})$. We also introduce a block exchange matrix $\mathcal{J}_{k}\in\mathbb{R}^{k(N+1)\times k(N+1)}$, defined as
\begin{align}
\mathcal{J}_{k} = \begin{pmatrix}
        &   & \mathbbm{1}_{k}  \\
        & \reflectbox{$\ddots$} &   \\
    \mathbbm{1}_{k}    &   &   
    \end{pmatrix},\nonumber    
\end{align} 
so that we can express the nonlinear terms in \eqref{SH:Num;finite} as
\begin{align}
    \sum_{i + j=n}  \mathbf{v}_{|i|}\circ\mathbf{v}_{|j|} = \mathbf{K}_{2}(V) :&= \bigg[2 \mathbf{C}_{2}(V)\mathcal{J}_{T} + \mathbf{C}_{1}(V) - 2\mathbbm{1}_{N+1}\otimes\mathcal{D}(\mathbf{v}_{0})\bigg] V ,\nonumber\\
    \sum_{i + j + \ell = n} \mathbf{v}_{|i|}\circ\mathbf{v}_{|j|}\circ\mathbf{v}_{|\ell|} = \mathbf{K}_{3}(V) :&= \bigg[\mathbf{C}_{2}(V)\mathcal{J}_{T} + \mathbf{C}_{1}(V) + \mathbf{C}_{1}^{T}(V) - \mathbbm{1}_{N+1}\otimes\mathcal{D}(\mathbf{v}_{0})\bigg]\mathbf{K}_{2}(V) \nonumber\\
    &\quad + \bigg[ \mathbf{C}_{2}^{T}(V)\mathbf{C}_{2}(V) - \mathcal{J}_{1}\otimes\mathcal{D}\bigg(\big[\mathbf{K}_{2}(V)\big]_{N}\bigg) - \mathbbm{1}_{N+1}\otimes\mathcal{D}\bigg(\big[\mathbf{K}_{2}(V)\big]_{0}\bigg)\bigg]V,\nonumber
\end{align}
where $\otimes$ denotes the Kronecker product of two matrices. We use MATLAB's Newton trust-region solver \verb1fsolve1, starting with an initial guess of the form \eqref{num:guess}. If an initial guess $(V, \mu^{*})$ converges to a localised solution for \eqref{SH:Num;finite} with a trust-region accuracy of $10^{-7}$, we then define $(V_{0},\mu_{0}) = (V,\mu^{*})$ and use this output as initial data for our numerical continuation scheme. To numerically approximate the branch of solutions in parameter space that this initial data belongs to we employ secant continuation. See \cite{avitabile2020Numerical} for details on the implementation of this continuation technique. The result is an approximation of a continuous branch of solutions $(\widetilde{u}(r),\mu)[s]$ which solves the Galerkin SHE for $s\in\mathbb{R}$. This allows us to map out solution curves for localised solutions in $\mu$-parameter space, resulting in the figures from Section~\ref{s:Continuation}.

%
%
%
%
%
%
%
%
%
%
\section{Proof of Propositions~\ref{prop:N123} and \ref{prop:N4}}\label{subsec:SmallProofs}

Let us begin by proving Proposition~\ref{prop:N123}. Throughout we will fix $m \in \mathbb{N}$ and we recall that $C_m = 2$ if $6 \mid m$, $C_m = 1$ if $2 \nmid m$ and $3\nmid m$, $C_m = -1$ if $2 \mid m$ and $3\nmid m$, and $C_m = -2$ if $2 \nmid m$ and $3 \mid m$. We also note that we can use trigonometric identities to find that $C_{2m} = (-1)^m C_m$, which will be used frequently throughout the proof to simplify expressions. 

\begin{proof}[Proof of Proposition~\ref{prop:N123}]
	In the case $N = 1$ the matching equation $\mathbf{a} = \mathbf{Q}^{m}_1(\mathbf{a})$ is given by
	\begin{equation}\label{1Match}
		\begin{split}
			a_0 &= a_0^2 + C_ma_1^2 \\
			a_1 &= C_ma_0a_1.
		\end{split}
	\end{equation}
	These equations can be solved by straightforward computation to arrive at the results of Proposition~\ref{prop:N123} for $N = 1$, and so we omit this from the proof. 
	
	In the case $N = 2$ the matching equations become
	\begin{equation}\label{2Match}
		\begin{split}
			a_0 &= C_m a_1^2 + (-1)^m C_m a_2^2 + a_0^2 \\
			a_1 &= 2a_1a_2 + C_ma_0a_1 \\
			a_2 &= (-1)^mC_ma_0a_2 + a_1^2.
		\end{split}
	\end{equation}  
    First note that if $a_1 = 0$, we are left with only the first and third equations, which give
    \begin{equation}
        (a_0,a_1,a_2) = \bigg(\frac{(-1)^m}{C_m},0,\sqrt{\frac{C_m - (-1)^m}{C_m^3}}\bigg) = \bigg(\frac{1}{C_{2m}},0,\sqrt{\frac{C_{2m} - 1}{C_{2m}^3}}\bigg).
    \end{equation}
    We note that these solutions are guaranteed by the symmetries outlined in Lemma~\ref{lem:HarmQ}. Now, assuming that $a_1 \neq 0$, we can divide it off of the second equation in \eqref{2Match} to give
    \begin{equation}
        1 = 2a_2 + C_ma_0 \implies a_0 = \frac{1 - 2a_2}{C_m}.
    \end{equation}
    Substituting this expression for $a_0$ into the third equation of \eqref{2Match} and rearranging gives
    \begin{equation}
        a_1^2 = \big[1 - (-1)^m(1-2a_2)\big]a_2.
    \end{equation}
    Then, putting these expressions for $a_0$ and $a_1^2$ into the first equation of \eqref{2Match} and rearranging results in the quadratic equation involving only $a_2$:
    \begin{equation}
        0 = (4 + 3(-1)^m C_m^3)a_2^2 + (2(C_m-2) + (1 - (-1)^m)C_m^3)a_2 - (C_m - 1).
    \end{equation}
    The right hand side of the above expression is exactly the function $p_m$ defined in the statement of the proposition, and so the conclusion follows by solving for $a_2$ and back-substituting to find $a_0$ and $a_1$. Moreover, in the case that $2 \nmid m$ and $3 \nmid m$ the quadratic equation above reduces to $a_2^2 = 0$, which gives only the radial spot solution to the matching problem. For both the $N=1$ and $N=2$ cases, one can verify that each solution $\mathbf{a}^{*}$ is nondegenerate by explicitly computing $\mathbbm{1}_{N} - D\mathbf{Q}_{N}^{m}(\mathbf{a}^{*})$. For each non-trivial solution we find that $\det[\mathbbm{1}_{N} - D\mathbf{Q}_{N}^{m}(\mathbf{a}^{*})]\neq0$, and hence we conclude that these solutions are isolated roots of the matching condition.
	
	We now turn to the case $N = 3$. The matching equation $\mathbf{a} = \mathbf{Q}^{m}_3(\mathbf{a})$ are given by 
	\begin{equation}\label{3Match}
		\begin{split}
			a_0 &= C_m a_1^2 + (-1)^m C_ma_2^2 + (-1)^m 2a_3^2 + a_0^2\\
			a_1 &= 2a_1a_2 + C_ma_2a_3 + C_ma_0a_1 \\
			a_2 &= C_ma_1a_3 + (-1)^m C_ma_0a_2 + a_1^2 \\
			a_3 &= (-1)^m2a_0a_3 + C_ma_1a_2.
		\end{split}
	\end{equation}  
	We note that $a_{1}=0$ implies that either $a_{2}=0$ or $a_{3}=0$. Taking $a_{1}=a_{3}=0$, we see that \eqref{3Match} reduces to the $N=1$ problem for $a_{0},a_{2}$ whose solution has already been found. Taking $a_{1}= a_{2}=0$ instead gives that the remaining equations reduce to the $N = 1$ problem with $m$ a multiple of 3. This situation was previously solved, and so we consider now the case $a_1 \neq 0$ by introducing the variable $\lambda\in\mathbb{R}$, defined by
\begin{equation}
    \lambda =  \frac{a_{3}}{C_m a_{1}}.
\end{equation}
We note that $\lambda=0$ would imply that $a_{3}=0$, which in turn gives $a_{2}=0$. But, this would then imply that $a_{1}=0$, and since we are assuming $a_1 \neq 0$, we must have that at least one of $a_{2}$ and $a_{3}$ is non-zero, giving $\lambda\neq0$. Substituting $\lambda$ and $\alpha$ into \eqref{3Match}, we find that
\begin{equation}\label{3Matchv2}
		\begin{split}
			a_0 &= C_m a_1^2 + (-1)^m C_m a_2^2 + (-1)^m 2 C_m^2 \lambda^2 a_1^2 + a_0^2\\
			1 &= 2a_2 + C_m^2 \lambda a_2 + C_ma_0 \\
			a_2 &= C_m^2 \lambda a_1^2 + (-1)^m C_ma_0a_2 + a_1^2 \\
			\lambda &= (-1)^m 2 \lambda a_0 + a_2.
		\end{split}
	\end{equation}  
	We input \eqref{3Matchv2} into Wolfram Alpha's \verb1GroebnerBasis1 function \cite{wolfram_2021_groebnerbasis}, which reduces our system to a set of ideal polynomials equipped with the same roots. Then, finding a solution to this reduced system yields a solution to \eqref{3Matchv2}, which in turn gives us a solution to \eqref{3Match} after we recover $a_3 = C_m \lambda a_1$. In order to find explicit solutions, we consider each case of $m$ separately.
	
	\paragraph{$6\mid m$:} We begin with the case $6\mid m$, such that $C_m=2$ and $(-1)^m=1$. Then the associated Gr\"obner basis can be written as
	\begin{equation}\label{Grob:6}
    \begin{split}
    a_2 &= (1- 2 a_0)\lambda,\\
    a_1^2 &= \frac{220(a_0-1)^2a_0^2+ 90 (a_0-1)a_0 + 9}{3} + \frac{10(a_0-1)a_0+3}{30}\\
    \lambda &= \frac{10(220(a_0-1)^2a_0^2+ 90 (a_0-1)a_0 + 9)}{3} + \frac{110(a_0-1)a_0 + 21}{6}\\
    0&=(1-2a_0)(220(a_0-1)^2a_0^2+ 90 (a_0-1)a_0 + 9).
    \end{split}
\end{equation}
Clearly, there exists a solution at $a_0=\frac{1}{2}$, which in turn implies $a_1^2 = \frac{1}{10}$, $\lambda = -\frac{1}{4}$, and $a_2=0$; this is described by the solution $\mathbf{a}_{0}^{3}$ in Proposition \ref{prop:N123} when $6\mid m$. Assuming that $a_0\neq\frac{1}{2}$, we can simplify \eqref{Grob:6} to the following,
\begin{equation}\label{Grob:6;simp}
    \begin{split}
    a_2 &= (1- 2 a_0)\lambda,\\
    a_1^2 &= \frac{10(a_0-1)a_0+3}{30}\\
    \lambda &= \frac{110(a_0-1)a_0 + 21}{6}\\
    0&=220(a_0-1)^2a_0^2 + 90 (a_0-1)a_0 + 9,
    \end{split}
\end{equation}
	which is exactly in the form of $\mathbf{a}_{j}^{3}$ in Proposition \ref{prop:N123}, when $6\mid m$.

	\paragraph{$2\nmid m$, $3\mid m$:} We now study the case when $2\nmid m$ and $3\mid m$, such that $C_m=-2$ and $(-1)^m=-1$. Then the associated Gr\"obner basis can be written as
	\begin{equation}\label{Grob:3}
    \begin{split}
    a_2 &= (1 + 2 a_0)\lambda,\\
    a_1^2 &= \frac{220a_0^4 + 40a_0^3 - 34a_0^2 - 2a_0 + 1}{25} + \frac{-440a_0^3-170a_0^2+38a_0+17}{50}\\
    \lambda &= \frac{38(220a_0^4 + 40a_0^3 - 34a_0^2 - 2a_0 + 1)}{45} + \frac{1760a_0^3+430a_0^2-142a_0-23}{30}\\
    0&= (1+2a_0)(220a_0^4 + 40a_0^3 - 34a_0^2 - 2a_0 + 1).
    \end{split}
\end{equation}
Similar to the case when $6\mid m$, we first note the existence of a solution of the form $a_0=-\frac{1}{2}$, $a_1^2 = \frac{3}{10}$, $\lambda = -\frac{1}{4}$, and $a_2=0$, which is again described by the solution $\mathbf{a}_{0}^{3}$ in Proposition \ref{prop:N123} when $2\nmid m$ and $3\mid m$. Taking the case when $a_0\neq -\frac{1}{2}$, we can again simplify \eqref{Grob:3} to the following,
	\begin{equation}\label{Grob:3;simp}
    \begin{split}
    a_2 &= (1 + 2 a_0)\lambda,\\
    a_1^2 &= \frac{-440a_0^3-170a_0^2+38a_0+17}{50}\\
    \lambda &= \frac{1760a_0^3+430a_0^2-142a_0-23}{30}\\
    0&= 220a_0^4 + 40a_0^3 - 34a_0^2 - 2a_0 + 1,
    \end{split}
\end{equation}
which is exactly in the form of $\mathbf{a}_{j}^{3}$ in Proposition \ref{prop:N123}, when $2\nmid m$ and $3\mid m$. 
\paragraph{$2\mid m$, $3\nmid m$:} We now study the case when $2\mid m$ and $3\nmid m$, such that $C_m=-1$ and $(-1)^m=1$. Then the associated Gr\"obner basis can be written as
	\begin{equation}\label{Grob:2}
    \begin{split}
    a_2 &= (1-2a_0)\lambda\\
    a_1^2 &= \frac{-5a_0^4 - 559a_0^3+3851a_0^2-185a_0-204}{369}\\
    \lambda &= \frac{463a_0^4 + 53387a_0^3 - 137638a_0^2 + 25249a_0 + 17931}{12546}\\
    0&= a_0^5 + 116a_0^4 - 217a_0^3 - 113a_0^2 + 24a_0 + 9,
    \end{split}
\end{equation}	
	which is already in the form of $\mathbf{a}_{j}^{3}$ in Proposition \ref{prop:N123}, when $2\mid m$ and $3\nmid m$.
	
	\paragraph{$2\nmid m$, $3\nmid m$:} Finally we study the case when $2\nmid m$ and $3\nmid m$, such that $C_m=1$ and $(-1)^m=-1$. Then the associated Gr\"obner basis can be written as
	\begin{equation}\label{Grob:5}
    \begin{split}
    a_2 &= (1 + 2 a_0)\lambda,\\
    a_1^2 &= \frac{(1-a_0)(7a_0-27)(a_0^2 + 3a_0 + 1)}{100} - \frac{7(a_0^2 + 3a_0 + 1)}{5} + 5a_0+2\\
    \lambda &= \frac{(1-a_0)(409a_0-1069)(a_0^2 + 3a_0 + 1)}{1800} - \frac{3(a_0^2 + 3a_0 + 1)}{5} + a_0+2\\
    0&= (1-a_0)^3(a_0^2 + 3a_0 + 1).
    \end{split}
\end{equation}	
	Clearly, there is a repeated solution $a_0=1$ to the final equation of \eqref{Grob:5}. This implies $a_1^2=0$, which contradicts our assumption that $a_1\neq0$, and so we set $a_0\neq1$ and simplify \eqref{Grob:5} to the following
	\begin{equation}\label{Grob:5;simp}
    \begin{split}
    a_2 &= (1 + 2 a_0)\lambda,\\
    a_1^2 &= 5a_0+2\\
    \lambda &= a_0+2\\
    0&= a_0^2 + 3a_0 + 1,
    \end{split}
\end{equation}
which is exactly in the form of $\mathbf{a}_{j}^{3}$ in Proposition \ref{prop:N123}, when $2\nmid m$ and $3\nmid m$. We note that, although the final polynomial has two distinct solutions for $a_0$, only one solution results in a positive value for $a_1^2$; hence, there is only one real nontrivial solution to \eqref{3Match} in the case when $2\nmid m$ and $3\nmid m$.

In the above examples \eqref{Grob:6;simp},\eqref{Grob:3;simp},\eqref{Grob:2} and \eqref{Grob:5;simp}, one can verify that each solution $a_0$ is a simple root of the respective polynomial $p_{m}(x)$, as defined in Proposition~\ref{prop:N123}. Therefore each solution $a_0=\lambda_{j}$ is isolated in $\mathbb{R}$ and, since $a_0$ corresponds to an element of $\mathbf{a}^{3}_{j}$ that solves \eqref{3Match}, we conclude that each root $\mathbf{a}^{3}_{j}$ is also isolated in $\mathbb{R}^{N+1}$.
\end{proof} 

\begin{proof}[Proof of Proposition~\ref{prop:N4}]
Let us consider $N = 4$ and $6 \mid m$. In this case the equation $\mathbf{a} = \mathbf{Q}^{m}_4(\mathbf{a})$ is given by
\begin{equation}\label{N4:match}
	\begin{split}
    		a_{0} &= a_{0}^{2} + 2a_{1}^{2} + 2a_{2}^{2} + 2a_{3}^{2} + 2a_{4}^{2},\\
    		a_{1} &= 2a_{0}a_{1} + 2a_{1}a_{2} + 2a_{2}a_{3} + 2a_{3}a_{4},\\
    		a_{2} &= a_{1}^{2} + 2a_{0}a_{2} + 2a_{1}a_{3} + 2a_{2}a_{4},\\
    		a_{3} &= 2a_{0}a_{3} + 2a_{1}a_{4} + 2a_{1}a_{2},\\
    		a_{4} &= a_{2}^{2} + 2a_{0}a_{4} + 2a_{1}a_{3}.
	\end{split}
\end{equation}
Starting with the case when $a_{1}=0$ we can see that the second equation in \eqref{N4:match} implies that either $a_{3}=0$ or $a_{2}+a_{4}=0$. In fact, the third and fifth equations in \eqref{N4:match} also tell us that $a_{1}=a_{2}+a_{4}=0$ implies that $a_{2}=a_{4}=0$. Hence, we arrive at two possible systems,
\begin{equation}
    a_{0} = a_{0}^{2} + 2a_{2}^{2} + 2a_{4}^{2},\qquad \qquad 
    a_{2} = a_{1}^{2} + 2a_{0}a_{2} + 2a_{2}a_{4}, \qquad \qquad a_{4} = a_{2}^{2} + 2a_{0}a_{4},
\end{equation}
when $a_{1}=a_{3}=0$, and 
\begin{equation}
    a_{0} = a_{0}^{2} + 2a_{3}^{2},\qquad \qquad
    a_{3} = 2a_{0}a_{3},
\end{equation}
when $a_{1}=0$ and $a_{3}\neq0$. These systems are identical to the systems \eqref{1Match} and  \eqref{2Match}, respectively, when $6\mid m$, and so we can preclude this case since solutions to the matching equations are consequences of Lemma~\ref{lem:HarmQ}.

We now make the assumption that $a_{1}\neq0$. We recall that the case when $a_{1}<0$ can be recovered from the transformation $\mathcal{R}$ in Lemma \ref{lem:EquiQ}, and so we restrict our calculations to the case when $a_{1}>0$. Introducing $x_{0}:=\frac{1}{2}-a_{0}$, we see that \eqref{N4:match} becomes,
\begin{equation}\label{N4:match;alt}
	\begin{split}
   		 \frac{1}{4} &= x_{0}^{2} + 2a_{1}^{2} + 2a_{2}^{2} + 2a_{3}^{2} + 2a_{4}^{2},\\
    		2x_{0}a_{1} &= 2a_{1}a_{2} + 2a_{2}a_{3} + 2a_{3}a_{4},\\
    		2x_{0}a_{2} &= a_{1}^{2} + 2a_{1}a_{3} + 2a_{2}a_{4},\\
    		2x_{0}a_{3} &= 2a_{1}a_{4} + 2a_{1}a_{2},\\
    		2x_{0}a_{4} &= a_{2}^{2} + 2a_{1}a_{3}.
    \end{split}
\end{equation}
With $a_{1}>0$ the quantity $\lambda:=\frac{a_{3}}{a_{1}}$ is well-defined. Substituting $\lambda$ into \eqref{N4:match;alt}, we find
\begin{equation}\label{N4:match,alpha}
	\begin{split}
    		\frac{1}{4} &= x_{0}^{2} + 2\left(1 + \lambda^{2}\right) a_{1}^{2} + 2a_{2}^{2} + 2a_{4}^{2},\\
    		0 &= \left[\left(x_{0}-a_{2}\right) - \left(a_{2} + a_{4}\right)\lambda\right] a_{1},\\
    		2\left(x_{0}-a_{4}\right)a_{2} &= \left(1 + 2\lambda \right)a_{1}^{2}, \\
    		0 &= \left[\lambda x_{0} -  \left(a_{2}+a_{4}\right)\right]a_{1},\\
    		2x_{0}a_{4} - a_{2}^{2} &= 2\lambda a_{1}^{2},
	\end{split}
\end{equation}
which, since $a_{1}>0$, implies that
\begin{equation}
    \left(x_{0}-a_{2}\right) - \lambda^{2}x_{0} =0,\qquad \qquad \lambda x_{0} - \left(a_{2}+a_{4}\right) =0,\nonumber
\end{equation}
and so,
\begin{equation}
    a_{2} = \left(1-\lambda^{2}\right)x_{0}, \qquad\qquad a_{3} = \lambda a_{1}, \qquad\qquad a_{4} = \left(\lambda^{2} + \lambda - 1\right)x_{0}.\nonumber
\end{equation}
Again, substituting these relations into \eqref{N4:match,alpha}, we obtain the following system,
\begin{equation}\label{N4:match,alpha2}
	\begin{split}
    		\frac{1}{4} &= \left[1 + 2\left(1-\lambda^{2}\right)^{2} + 2\left(\lambda^{2} + \lambda -1\right)^{2}\right]x_{0}^{2} + 2\left(1 + \lambda^{2}\right) a_{1}^{2},\\
   		 \left(1 + 2\lambda \right)a_{1}^{2} &= 2\left(2-\lambda-\lambda^{2}\right)\left(1 - \lambda^{2}\right)x_{0}^{2}, \\
    		2\lambda a_{1}^{2} &= \left[2\left(\lambda^{2} + \lambda - 1\right) - \left(1-\lambda^{2}\right)^{2}\right]x_{0}^{2}.
	\end{split}
\end{equation}
We can manipulate the second equation in \eqref{N4:match,alpha2} to find that
\begin{equation}
    \frac{1}{4} = q_{0}(\lambda) x_{0}^{2},\qquad \qquad a_{1}^{2} = q_{1}(\lambda)x_{0}^{2}
\end{equation}
where
\begin{equation}
	\begin{split}
        		q_{1}(\lambda) &= 2\left(2-\lambda-\lambda^{2}\right)\left(1 - \lambda^{2}\right) - \left[2\left(\lambda^{2} + \lambda - 1\right) - \left(1-\lambda^{2}\right)^{2}\right],\\
     		&= 3\lambda^{4} + 2\lambda^{3} - 10\lambda^{2} - 4\lambda + 7,\\
    		q_{0}(\lambda) &= 1 + 2(1-\lambda^{2})^{2} + 2(\lambda^{2}+\lambda-1)^{2} + 2(1+\lambda^{2})\,q_{1}(\lambda),\\
     		&=  6\lambda^{6} + 4\lambda^{5} -10\lambda^{4}-12\lambda^{2}-12\lambda+19.
	\end{split}
\end{equation}
Thus, for the case when $a_{1}>0$, solutions take the form
\begin{equation}\label{N4:Sol}
    \begin{split}
    \left(x_{0}, a_{1}, a_{2}, a_{3}, a_{4}\right)(\lambda) &= \frac{1}{2\sqrt{q_0(\lambda)}}\left(1,\, \sqrt{q_{1}(\lambda)},\, (1-\lambda^{2}),\, \lambda\sqrt{q_{1}(\lambda)},\, (\lambda^{2} + \lambda -1) \right), \\
    \left(x_{0}, a_{1}, a_{2}, a_{3}, a_{4}\right)(\lambda) &= \frac{1}{2\sqrt{q_0(\lambda)}}\left(-1,\, \sqrt{q_{1}(\lambda)},\, -(1-\lambda^{2}),\, \lambda\sqrt{q_{1}(\lambda)},\, -(\lambda^{2} + \lambda -1) \right),
    \end{split}
\end{equation}
where $\lambda$ is a root of the polynomial
\begin{equation}\label{N4:p}
	\begin{split}
    		p(\lambda) &= 4\lambda\left(1-\lambda^{2}\right)\left(2-\lambda-\lambda^{2}\right) - \left(1+2\lambda\right)\left[2\left(\lambda^{2} + x -1\right) - \left(1-\lambda^{2}\right)^{2}\right]\\
     		&= 6\lambda^{5} + 5\lambda^{4} - 20\lambda^{3}-12\lambda^{2}+12\lambda+3.
	\end{split}
\end{equation}
We note that $\textnormal{sgn}(a_{2})=\textnormal{sgn}(1-\lambda^{2})$, $\textnormal{sgn}(a_{3})=\textnormal{sgn}(\lambda)$, and $\textnormal{sgn}(a_{4})=\textnormal{sgn}(\lambda^{2}+\lambda-1)$, and we so verify the following:
\begin{align}
    \begin{array}{|c|c|c|c|c|c|c|c|}
\hline
\lambda & \to -\infty & \frac{-1-\sqrt{5}}{2} & -1  & 0  & \frac{-1+\sqrt{5}}{2}  & 1  & \to\infty \\
\hline
\textnormal{sgn}\,[p(\lambda)]  & -1  & 1  &  -1 &  1 &  1 & -1  & 1\\
\hline
    \end{array}\nonumber
\end{align}
Hence, by the intermediate value theorem, we conclude there exists the following roots $\{\lambda_{j}\}_{j=1}^{5}$ such that $p(\lambda_{j})=0$, where
\begin{equation}
	\lambda_{1}<\frac{-1-\sqrt{5}}{2}<\lambda_{2}<-1<\lambda_{3}<0<\frac{-1+\sqrt{5}}{2}<\lambda_{4}<1<\lambda_{5}.
\end{equation} 
Since $p(\lambda)$ is a quintic polynomial, these are all of the possible roots of $p(\lambda)$. For each $j=1,\dots,5$, we reintroduce $a_{0}=\frac{1}{2}-x_{0}$ for the two solutions \eqref{N4:Sol} in order to recover the solutions $\mathbf{a}_{j}^{4}$. This concludes the proof.
\end{proof} 

%
%
%
%
%
%
%
%
%
%
\section{Computer-Assisted Proof Details of Theorem~\ref{prop:IntExist}}\label{app:CompProof}

In this section we describe the computer-assisted proof for Theorem~\ref{prop:IntExist} which asserts the existence of a solution in the Banach space $X$ of regulated functions on $[0,1]$ that solves the equation
\begin{equation}
	\alpha(t) = 2\int_0^{1-t}a(s)a(s+t)\drm s + \int_0^t\alpha(s)\alpha(t-s)\drm s,\qquad t\in[0,1].
\end{equation}
The proof follows the Newton--Kantorovich argument set out in~\cite{vandenberg2015Rigorous} which employs a numerically computed linear spline approximation to the continuous function on a equispaced mesh. 

The set up is as follows. We consider the function $G:X\rightarrow X$, defined in Section~\ref{subsec:BigMatch}, given by
\begin{equation}
	G(w):= t\mapsto w(t) - 2\int_0^{1-t} w(s)w(s+t)\drm s - \int_0^t w(s)w(t-s)\drm s.
\end{equation}
Our goal is to find a solution $\bar w$ with $G(\bar w) =0$, which by definition gives $\bar w = Q_\infty(\bar w)$. Let us define a mesh $\Delta_M = \{0=t_0<t_1<\cdots<t_M=1\}$ such that $t_{k+1}-t_k=\delta t$, $\forall k\in[0,M-1]$ and consider the subspace of $X$ given by the linear splines $S_M\subset X$ with base points in $\Delta_M$. Clearly, $S_M \simeq \mathbb{R}^{M+1}$ since each element of $S_M$ can be uniquely determined by its value at the base points, and so the $S_M$ is a closed subspace of $X$. Define the projection $\Pi_M:X\rightarrow S_M$
\begin{equation}
	\Pi_Mw = (w(t_0),w(t_1),\ldots, w(t_M))\in\mathbb{R}^{M+1}\simeq S_M,
\end{equation}
with the complementary projection
\begin{equation}
	\Pi_\infty = I - \Pi_M.
\end{equation}
For $w\in X$, we denote $\Pi_Mw$ by $w_M$. Since $\Pi_M$ is a projection one can decompose $X$ such that
\begin{equation}
X = \Pi_M X \oplus \Pi_\infty X = S_M\oplus S_\infty,
\end{equation}
where $S_\infty:=(I-\Pi_M)X$. We note that since $S_M$ is finite-dimensional, the Hahn--Banach theorem gives $S_\infty\subset X$ is also closed, and hence both $S_M$ and $S_\infty$ are Banach spaces with $\Pi_M$ and $\Pi_\infty$ continuous.

We define $G^M:S_M\rightarrow S_M\simeq\mathbb{R}^{M+1}$ where $G^M(w) =  \Pi_M G(w_M)$, so that
\begin{equation}
	G^M(w)_k = w(t_k) - 2\int_{0}^{1-t_k} w(s)w(s+t_k)ds + \int_0^{t_k} w(s)w(t_k-s)ds.
\end{equation}
A numerical approximation $\hat w_M \in S^M$ of the zero of $G(w)$ is found by Newton iteration of \eqref{HexMatchRiemann} and creating a spline interpolation of these values. We remark that our approximate solution $\hat w_M$ is positive, which leads to the fact that the unique fixed point in a neighbourhood of $\hat w_M$ is also positive. We further numerically compute the approximate inverse $A^\dagger_M$ of the Jacobian $DG^M(\hat w_M)$ of $G^M$ at the approximate zero $\hat w_M$. 

Let us now define the ball 
\begin{equation}
	B_{\omega}(r) = \{w\in X : \|\Pi_Mw\|_\infty \leq r\;\;\mbox{and}\;\; \|\Pi_\infty w\|_\infty\leq \omega r \},
\end{equation}
where $\omega>0$ is fixed and $r$ is treated as a parameter. The aim is the show that the map $T:X\rightarrow X$ given by
\begin{equation}
    T(w) := (\Pi_M - A^\dagger_M\Pi_MG)(w)+\Pi_\infty(w-G(w)),
\end{equation}
is a contraction in $\hat w_M + B_\omega(r)$ such that $T$ has a unique zero in $\hat w_M + B_\omega(r)$. Following the argument in \cite{vandenberg2015Rigorous}, if $T$ has a unique fixed point in $\hat w_M + B_\omega(r)$ then there must be a corresponding zero of $G$ provided $A^\dagger_M$ is injective. The following theorem states the hypotheses that we need to check numerically and was proved in \cite{vandenberg2015Rigorous}.

\begin{thm}[\cite{vandenberg2015Rigorous}] 
Let $\omega, r>0$, such that the radii polynomials $p_k(r)<0$ for all $0\leq k\leq M$ and $p_\infty(r)<0$, given by
\begin{equation}
p_k(r) := Y_k + Z_k(r) - r, \qquad p_\infty(r) = Y_\infty + Z_\infty - \omega r
\end{equation}
where 
\begin{equation}
    \begin{split}
        |(\Pi_M (T(\hat w_M) - \hat w_M))_k|\leq& Y_k ,\qquad \mathrm{for}\ 0\leq k\leq M,\\
        \|\Pi_\infty(T(\hat w_M) - \hat w_M)) \|_\infty\leq& Y_\infty,\\
        \sup_{w_1,w_2\in B_\omega(r)}|(\Pi_MDT(\hat w_M + w_1)w_2)_k|\leq& Z_k,\qquad \mathrm{for}\ 0\leq k\leq M,\\
        \sup_{w_1,w_2\in B_\omega(r)}\|(\Pi_MDT(\hat w_M + w_1)w_2)\|_\infty\leq& Z_\infty.
    \end{split}
\end{equation}
Furthermore, assume $A^\dagger_M$ is injective. Then the map $T$ is a contraction on $\hat w_M + B_\omega(r)$ such that $G$ has a zero inside $\hat w_M +B_\omega(r)$. Furthermore, if $\min \hat w_M > r(1+\omega)$, then the zero of $G$ is strictly positive on $[0,1]$. 
\end{thm}

The $Y$ and $Z$ bounds and verification that the radii polynomials are all negative are carried out using interval arithmetic using IntLab \cite{Rump1999} so as to be rigorous. To carry out these operations in IntLab, we introduce the following interval arithmetic notation for the evaluation of a given function $f(t)$. The notation $\mathcal{E}^\pm_k(f)$ is defined as
\begin{equation}
	\mathcal{E}^-_k(f)\leq f(t_k)\leq \mathcal{E}^+_k(f),
\end{equation}
and $\mathcal{E}^\pm_{[t_k,t_{k+1}]}(f)$ is defined as
\begin{equation}
	\mathcal{E}^-_{[t_k,t_{k+1}]}(f)\leq f(t)\leq \mathcal{E}^+_{[t_k,t_{k+1}]}(f),\qquad \forall t\in[t_k,t_{k+1]}].
\end{equation}
Furthermore, we define the interval notation
\begin{equation}
	\mathcal{E}_k(f) := [\mathcal{E}^-_k(f),\mathcal{E}^+(f)],\qquad \mbox{and}\qquad \mathcal{E}_{[t_k,t_{k+1}]}(f) := [\mathcal{E}^-_{[t_k,t_{k+1}]}(f),\mathcal{E}^+_{[t_k,t_{k+1}]}(f)],
\end{equation}
such that $f(t_k)\in\mathcal{E}_k(f)$ and $f(t)\in\mathcal{E}_{[t_k,t_{k+1}}$ for all $t\in[t_k,t_{k+1}]$. It will also be useful to define the vector of intervals $\mathcal{E}(f)$ such that $(\mathcal{E}(f))_k:=\mathcal{E}_k(f)$, for $k=0,\ldots,M$. 

In what follows we demonstrate how one computes the $Y$ and $Z$ bounds. All bounds follow a similar derivation to~\cite{vandenberg2015Rigorous} except the $Z_\infty$ bound which is slightly different in our case since our mapping lacks some of the stronger differentiability properties to that studied in \cite{vandenberg2015Rigorous}. 

For convenience in what follows, let us define the function $g:[0,1]\rightarrow\mathbb{R}$
\begin{equation}
	g(t) = -2 \int_{0}^{1-t}\left\{\hat w_M\left(s\right)\,\hat w_M\left(t+s\right)\right\} \textnormal{d}s - \int_{0}^{t}\left\{\hat w_M\left(s\right)\,\hat w_M\left(t-s\right)\right\} \textnormal{d}s
\end{equation}
such that $G(\hat w_M) = \hat w_M + g$. Note that $g\in C^2([0,1])$ since $\hat w_M$ is a linear spline. 

\begin{lem} 
Let 
\begin{equation}
	Y_k:=\sup |(A^\dagger_M[\hat w_M +\mathcal{E}(g) ])_k|\qquad\mathrm{and}\qquad Y_\infty := \max_{0\leq k\leq M}\left\{ \frac{\delta t^2}{8}\mathcal{E}_{[t_k,t_{k+1}]}(h)\right\},
\end{equation}
where 
\begin{equation}
	h(t) = 2\hat w_M(1-t)\hat w_M'(1) + 2\hat w_M'(1-t)\hat w_M(1) + \hat w_M'(t)\hat w_M(0) + \hat w_M(t)\hat w_M'(0),
\end{equation}
then 
\begin{equation}
    \begin{split}
        |(\Pi_M (T(\hat w_M) - \hat w_M))_k| =& |(A^\dagger_M[\hat w_M +\Pi_Mg])_k|\leq Y_k,\\
        \|\Pi_\infty(T(\hat w_M) - \hat w_M)) \|_\infty =&\|(I-\Pi_M) g\|_\infty \leq Y_\infty.
    \end{split}
\end{equation}
\end{lem}

\begin{proof}
The $Y_k$ bound is proved in~\cite[Lemma 7.1]{vandenberg2015Rigorous}, while the $Y_\infty$ bound follows the discussion of~\cite[Lemma 7.2]{vandenberg2015Rigorous} where we use the standard linear interpolant estimate for $t\in[t_k,t_{k+1}]$
\begin{equation}
	|g(t)-(\Pi_Mg)(t)|\leq \frac18(\delta t)^2\sup_{t\in[t_k,t_{k+1}]}|g''(t)|.
\end{equation}
Here we have $g''(t) = h_1(t) + h_2(t) + h_3(t)$ and 
\begin{equation}
    \begin{split}
        h_1(t) =& -2\int_0^{1-t}\hat w_M(s)\hat w_M''(t+s)ds,\\
        h_2(t) =& -\int_0^t\hat w_M(s)\hat w_M''(t-s)ds,\\
        h_3(t)=& 2\hat w_M(1-t)\hat w_M'(1) -2\hat w_M'(1-t)\hat w_M(1) - \hat w_M'(t)\hat w_M(0) - \hat w_M(t)\hat w_M'(0).
    \end{split}
\end{equation}
Since $\hat w_M$ is a linear interpolant $h_1=h_2\equiv 0$, giving way to the proof of the lemma.
\end{proof} 

\begin{rmk}
We note that in the numerical calculation of $Y_\infty$, we use the fact that 
\begin{equation}
	(\hat w_M'(t))_k = ((\hat w_M)_{k+1}-(\hat w_M)_k)/(\delta t)
\end{equation} 
such that the $Y_\infty$ bound becomes order $\delta t$.
\end{rmk}

We now derive the $Z_k$ bounds, again following~\cite{vandenberg2015Rigorous}. For $w_1,w_2 \in B_\omega(r)$, let us write
\begin{equation}
	z(w_1,w_2) = DT(\hat w_M + w_1)w_2.
\end{equation}
Our goal is to estimate $|\Pi_M z(w_1,w_2)|$ for $w_1,w_2 \in B_\omega(r)$ as a function of $r$. Let us write $w_i = r\tilde w_i$, $i = 1,2$, so that $\tilde w_i \in B_\omega(1)$. This gives (after following the paper's computations) 
\begin{equation}
	\Pi_M z(w_1,w_2) = r(\mathbbm{1}_{M+1} - A_M^\dagger DG^m(\hat w_M))\Pi_M \tilde w_2 - rA_M^\dagger\eta'(0),
\end{equation}
where $\eta$ is given by
\begin{equation}
	\eta(\tau) := \Pi_M G(\hat w_M + r\tilde w_1 + \tau \tilde w_2) - \Pi_M G(\hat w_M + \tau \Pi_M \tilde w_2).
\end{equation}

We start with the first term: $(\mathbbm{1}_{M+1} - A_M^\dagger DG^M(\hat w_M))\Pi_M \tilde w_2$. We view $DG^M(\hat w_M)$ as an $(M+1)\times (M+1)$ matrix, which is broken into three cases: diagonal, above diagonal, and below diagonal elements. Here any integral with domain of integration outside $[0,1]$ should be read as 0. 

\underline{$k_1 = k_2$}: 
\begin{equation}
	\begin{split}
		DG^M(\hat w_M)_{k_1,k_2} &= \phi_0(t_{k_1},t_{k_2}) \\
		&= 1 - 2\int_{t_{k_1}}^{t_{k_1 + 1}} \hat w(s - t_{k_1})\bigg(\frac{t_{k_1+1} - s}{t_{k_1+1} - t_{k_1}}\bigg)\drm s - 2\int_{t_{k_1} + t_{k_1-1}}^{2t_{k_1}}\hat w_M(s)\bigg(\frac{s - t_{k_1} - t_{k_1 - 1}}{t_{k_1} - t_{k_1-1}}\bigg)\drm s \\
		&-2\int_{2t_{k_1}}^{t_{k_1} + t_{k_1+1}}\hat w_M(s)\bigg(\frac{t_{k_1+1} + t_{k_1} - s}{t_{k_1+1} - t_{k_1}}\bigg)\drm s - \int_{t_{k_1-1}}^{t_{k_1}}\hat w_M(t_{k_1} - s)\bigg(\frac{s - t_{k_1-1}}{t_{k_1} - t_{k_1-1}}\bigg)\drm s \\
		& - \int_{0}^{t_{k_1} - t_{k_1 - 1}}\hat w_M(s)\bigg(\frac{s}{t_{k_1} - t_{k_1-1}}\bigg)\drm s    
	\end{split}
\end{equation} 

\underline{$k_1 < k_2$}: 
\begin{equation}
	\begin{split}
		DG^M(\hat w_M)_{k_1,k_2} &= \phi_{-1}(t_{k_1},t_{k_2}) \\
		&= - 2\int_{t_{k_2-1}}^{t_{k_2}} \hat w(s - t_{k_1})\bigg(\frac{s - t_{k_2 - 1}}{t_{k_2} - t_{k_2-1}}\bigg)\drm s  - 2\int_{t_{k_2}}^{t_{k_2 + 1}} \hat w(s - t_{k_1})\bigg(\frac{t_{k_2+1} - s}{t_{k_2+1} - t_{k_2}}\bigg)\drm s  \\
		&-2 \int_{t_{k_2 - 1} + t_{k_1}}^{t_{k_2} + t_{k_1}} \hat w_M(s)\bigg(\frac{s - t_{k_2-1} - t_{k_1}}{t_{k_2} - t_{k_2-1}}\bigg)\drm s -2 \int_{t_{k_2} + t_{k_1}}^{t_{k_2+1} + t_{k_1}} \hat w_M(s)\bigg(\frac{t_{k_2+1} + t_{k_1} - s}{t_{k_2+1} - t_{k_2}}\bigg)\drm s
	\end{split}
\end{equation} 

\underline{$k_1 > k_2$}: 
\begin{equation}
	\begin{split}
		DG^M(\hat w_M)_{k_1,k_2} &= \phi_{1}(t_{k_1},t_{k_2}) \\
		&= - \int_{t_{k_2-1}}^{t_{k_2}}\hat w_M(t_{k_1} - s)\bigg(\frac{s - t_{k_2-1}}{t_{k_2} - t_{k_2-1}}\bigg)\drm s - \int_{t_{k_2}}^{t_{k_2+1}}\hat w_M(t_{k_1} - s)\bigg(\frac{t_{k_2 +1} - s}{t_{k_2+1} - t_{k_2}}\bigg)\drm s \\
		&- \int_{t_{k_1} - t_{k_2}}^{t_{k_1} - t_{k_2 - 1}} \hat w_M(s) \bigg(\frac{t_{k_1} - t_{k_2 - 1} - s}{t_{k_2} - t_{k_2 -1}}\bigg)\drm s - \int_{t_{k_1} - t_{k_2 + 1}}^{t_{k_1} - t_{k_2}} \hat w_M(s) \bigg(\frac{s - (t_{k_1} - t_{k_2 + 1})}{t_{k_2+1} - t_{k_2}}\bigg)\drm s
	\end{split}
\end{equation} 

Now, let $\mathcal{M}(\phi)$ be the matrix of intervals given by
\begin{equation}
	\mathcal{M}(\phi)_{k_1,k_2} ={\mathcal{E}} (\phi_{\mathrm{sign}(k_1 - k_2)}(t_{k_1},t_{k_2})).
\end{equation}
Combining this with the fact that $|(\Pi_M \tilde{w}_2)_k|\leq 1$, and using the notation ${\bf 1} = (1,1,\dots,1)\in\R^{M+1}$, the term $(\mathbbm{1}_{M+1} - A_M^\dagger DG^M(\hat w_M))\Pi_M \tilde w_2$ can then be estimated by 
\begin{equation}
	\bigg|((\mathbbm{1}_{M+1} - A_M^\dagger DG^M(\hat w_M))\Pi_M \tilde w_2)_k\bigg| \leq \sup \bigg(\bigg|\mathbbm{1}_{M+1} - A_M^\dagger\mathcal{M}(\phi)\bigg|{\bf 1}\bigg),
\end{equation}
where the absolute value on the right-hand side is taken element-wise for the matrix of intervals $\mathbbm{1}_{M+1} - A_M^\dagger\mathcal{M}(\phi)$.

Now, we can express $\eta(\tau) \in S_M$ as 
\begin{equation}
	\begin{split}
		\eta(\tau)_k &= r\tilde w_1 - 2\int_{t_k}^1 \zeta_{1,s}(\tau) \drm s - \int_0^{t_k} \zeta_{2,s}(\tau) \drm s
	\end{split}
\end{equation}
where $\zeta_s^j(\tau)$, $j = 1,2$ are given by
\begin{equation}
	\begin{split}
		\zeta_{1,s}(\tau) &= (\hat w_M(s) + r\tilde w_1(s) + \tau \tilde w_2(s))(\hat w_M(s - t_k) + r\tilde w_1(s - t_k) + \tau \tilde w_2(s - t_k)) \\ & \quad \quad- (\hat w_M(s) + \tau \Pi_M \tilde w_2(s))(\hat w_M(s - t_k) + \tau \Pi_M \tilde w_2(s - t_k)) \\
		\zeta_{2,s}(\tau) &= (\hat w_m(s) + r\tilde w_1(s) + \tau \tilde w_2(s))(\hat w_M(t_k - s) + r\tilde w_1(t_k - s) + \tau \tilde w_2(t_k - s)) \\ & \quad \quad- (\hat w_M(s) + \tau \Pi_M \tilde w_2(s))(\hat w_M(t_k - s) + \tau \Pi_M \tilde w_2(t_k - s)).
	\end{split}
\end{equation}
So, for any $\tilde w_1,\tilde w_2 \in B_\omega(1)$ we have that 
\begin{equation}
	\begin{split}
		|\zeta_{1,s}'(0)| &= |\tilde w_2(s)(\hat w_M(s - t_k) + r\tilde w(s - t_k)) + (\hat w_M(s) + r\tilde w_1(s))\tilde w_2(s - t_k) \\ &\quad \quad - \Pi_M \tilde w_2(s) \hat w_M(s - t_k) - \hat w_M(s)\Pi_M\tilde w_2(s - t_k)| \\
		&\leq \omega|\hat w_M(s-t_k)| + 2r(1 + \omega)^2 + \omega |\hat w_M(s)|,
	\end{split}
\end{equation}
and similarly 
\begin{equation}
	|\zeta_{2,s}'(0)| \leq \omega|\hat w_M(t_k - s)| + 2r(1 + \omega)^2 + \omega |\hat w_M(s)|.	
\end{equation}
Therefore, 
\begin{equation}
	|\eta'(0)_k| \leq V^1_k + V^2_kr
\end{equation}
where the $V^i_k > 0$ are given by
\begin{equation}
	\begin{split}
		V^1_k &= \sup \mathcal{E}_k \bigg(\omega \int_t^1 [\hat w_M(s - t) + \hat w_M(s)]\drm s + \omega \int_0^t [\hat w_M(t - s) + \hat w_M(s)]\drm s\bigg) \\
		V^2_k &= \sup \mathcal{E}_k(2(1 + \omega)^2).
	\end{split}
\end{equation}
Notice that $V^2_k$ is independent of all variables except $\omega > 0$. We have written it as above for easy comparison with the bounds in \cite{vandenberg2015Rigorous}. We summarise the above with the following lemma.

\begin{lem} 
	Define the vector-valued $Z(r) \in \R^{m+1}$ by  
	\begin{equation}
		Z(r) := \bigg(\sup\bigg|\mathbbm{1}_{M+1} - A_M^\dagger\mathcal{M}(\phi)\bigg|{\bf 1} + |A_M^\dagger|V^1\bigg)r + |A_M^\dagger|V^2r^2.
	\end{equation}
	Then, for all $w_1,w_2 \in B_\omega(r)$ we have
	\begin{equation}
		|\Pi_M z(w_1,w_2)_k| = |(DT(\hat w_M + w_1)w_2)_k| \leq Z_k(r)
	\end{equation}
	for each $k = 0,\dots,M$.
\end{lem}

Lastly, we derive the $Z_\infty$ bounds which bound $\|\Pi_\infty DT(\hat w_M + w_1)w_2\|_\infty$. Similar to above, we have
\begin{equation}
	\begin{split}
		\|\Pi_\infty z(w_1,w_2)\|_\infty &= r\|\Pi_\infty D(I - G)(\hat w_M + r\tilde w_1)\tilde w_2\|_\infty \\
		&= r\|(I - \Pi_M) D(I - G)(\hat w_M + r\tilde w_1)\tilde w_2\|_\infty \\
		&= r\sup_{k = 0,\dots, m, \quad \tau \in [0,t_{k+1} - t_k]} \bigg|\xi(t_k + \tau) - \xi(t_k) - \bigg(\frac{\xi(t_{k+1}) - \xi(t_k)}{t_{k+1} - t_k}\bigg)\tau \bigg| \\
	\end{split}
\end{equation} 
where we have used the definition of $\Pi_M$ and introduced the shorthand
\begin{equation}
	\begin{split}
		\xi(t) &:= (D(I - G)(\hat w_M + r\tilde w_1)\tilde w_2)(t) \\
		&= 2\int_0^{1-t} (\hat w_M(s+t) + r\tilde w_1(s+t))\tilde w_2(s) \drm s + 2\int_0^{1-t} \tilde w_2(s+t)(\hat w_M(s) + r\tilde w_1(s)) \drm s \\
		& \quad \quad +\int_0^t (\hat w_M(s) + r\tilde w_1(s))\tilde w_2(t-s) \drm s + \int_0^t \tilde w_2(s)(\hat w_M(t-s) + r\tilde w_1(t-s)) \drm s. 
	\end{split}
\end{equation}
We will proceed by bounding each integral term in $|\xi(t_k + \tau) - \xi(t_k)|$ individually. 

First, we have  
\begin{equation}
\begin{split}
	\bigg|&\int_{t_k + \tau}^1 (\hat w_M(s) + r\tilde w_1(s))\tilde w_2(s - t_k - \tau) \drm s - \int_{t_k}^1 (\hat w_M(s) + r\tilde w_1(s))\tilde w_2(s - t_k) \drm s\\ &\quad \quad - \frac{\tau}{t_{k+1} - t_k}\bigg(\int_{t_{k+1}}^1 (\hat w_M(s) + r\tilde w_1(s))\tilde w_2(s - t_{k+1}) \drm s - \int_{t_k}^1 (\hat w_M(s) + r\tilde w_1(s))\tilde w_2(s - t_k) \drm s\bigg)\bigg| \\
	&=\bigg|\int_{0}^{1 - t_k - \tau} (\hat w_M(s + t_k + \tau) + r\tilde w_1(s + t_k + \tau))\tilde w_2(s) \drm s - \int_{0}^{1-t_k} (\hat w_M(s + t_k) + r\tilde w_1(s+t_k))\tilde w_2(s) \drm s \\ 
	& \quad \quad - \frac{\tau}{t_{k+1} - t_k}\bigg(\int_{0}^{1 - t_{k+1}} (\hat w_M(s + t_{k+1}) + r\tilde w_1(s + t_{k+1}))\tilde w_2(s) \drm s - \int_{0}^{1-t_k} (\hat w_M(s + t_k) + r\tilde w_1(s+t_k))\tilde w_2(s) \drm s \bigg) \bigg|  \\
	&\leq \int_{1 - t_k - \tau}^{1 - t_k} |\hat w_M(s + t_k) + r\tilde w_1(s+t_k)||\tilde w_2(s)|\drm s + \frac{\tau}{t_{k+1} - t_k}\int_{1 - t_{k+1}}^{1 - t_k} |\hat w_M(s + t_k) + r\tilde w_1(s+t_k)||\tilde w_2(s)|\drm s \\ & \quad \quad + \int_{0}^{1 - t_k - \tau} \underbrace{\bigg|\hat w_M(s + t_k + \tau) - \hat w_M(s + t_k) - \frac{\tau}{t_{k+1} - t_k}[\hat w_M(s + t_{k+1}) - \hat w_M(s + t_k)] \bigg|}_{= 0 } |\tilde{w}_2(s)| \drm s \\
	 &\quad \quad + r\int_{0}^{1 - t_k - \tau} \underbrace{\bigg|\tilde w_1(s + t_k + \tau) - \tilde w_1(s + t_k) - \frac{\tau}{t_{k+1} - t_k}[\tilde w_1(s + t_{k+1}) - \tilde w_1(s + t_k)] \bigg|}_{= \Pi_\infty \tilde w_1(s)} |\tilde{w}_2(s)| \drm s \\
	&\leq 2(1+\omega)\bigg(\int_{1 - t_{k+1}}^{1 - t_k} |\hat w_M(s + t_k)|\drm s\bigg) + 2r\delta t(1+\omega)^2
	 + r\omega(1 + \omega) \\
	  \\
	 &\leq 2(1+\omega)\delta t\max\{\hat w_M(1-\delta t),\hat w_M(1)\} + 2r\delta t(1+\omega)^2
	 + r\omega(1 + \omega)
\end{split}
\end{equation}
Next, we have
\begin{equation}
\begin{split}
	\bigg|&\int^{1 - t_k - \tau}_0 \tilde w_2(s+t_k+\tau)(\hat w_M(s) + r\tilde w_1(s)) \drm s -\int_{0}^{1-t_k} \tilde w_2(s+t_k)(\hat w_M(s) + r\tilde w_1(s)) \drm s \\ 
	& \quad \quad - \frac{\tau}{\delta t}\bigg(\int_{0}^{1-t_{k+1}} \tilde w_2(s+t_{k+1})(\hat w_M(s) + r\tilde w_1(s)) \drm s -\int_{0}^{1-t_k} \tilde w_2(s+t_k)(\hat w_M(s) + r\tilde w_1(s)) \drm s\bigg)\bigg| \\
	= \bigg|&\int_{t_k - \tau}^1 \tilde w_2(s)(\hat w_M(s-t_k-\tau) + r\tilde w_1(s-t_k-\tau)) \drm s -\int_{t_k}^{1} \tilde w_2(s)(\hat w_M(s-t_k) + r\tilde w_1(s-t_k)) \drm s \\ 
	& \quad \quad - \frac{\tau}{\delta t}\bigg(\int^{1}_{t_{k+1}} \tilde w_2(s)(\hat w_M(s-t_{k+1}) + r\tilde w_1(s-t_{k+1})) \drm s -\int_{t_k}^{1} \tilde w_2(s)(\hat w_M(s-t_k) + r\tilde w_1(s-t_k)) \drm s\bigg)\bigg|
	\\
\leq& 2(1+\omega)\delta t\max\{\hat w_M(0),\hat w_M(\delta t) \}+ 2r\delta t(1+\omega)^2 + r\omega(1+\omega)
\end{split}
\end{equation}
upon following the similar bounding techniques to the first integral differences. The final two integrals are bounded as
\begin{equation}
\begin{split}
	\bigg|&\int_0^{t_k + \tau} (\hat w_M(s) + r\tilde w_1(s))\tilde w_2(t_k + \tau-s) \drm s - \int_0^{t_k} (\hat w_M(s) + r\tilde w_1(s))\tilde w_2(t_k-s) \drm s \\ 
	& \quad \quad - \frac{\tau}{t_{k+1} - t_k}\bigg(\int_0^{t_{k+1}} (\hat w_M(s) + r\tilde w_1(s))\tilde w_2(t_{k+1}-s) \drm s - \int_0^{t_k} (\hat w_M(s) + r\tilde w_1(s))\tilde w_2(t_k-s) \drm s\bigg)\bigg| \\
&\leq 2(1+\omega)\delta t\max\{\hat w_M(0),\hat w_M(\delta t) \} + 2r\delta t (1+\omega)^2  +  r\omega(1+\omega) 
\end{split}
\end{equation}
and
\begin{equation}
\begin{split}
	\bigg|&\int_0^{t_k + \tau} \tilde w_2(s)(\hat w_M(t_k + \tau-s) + r\tilde w_1(t_k+\tau-s)) \drm s - \int_0^{t_k} \tilde w_2(s)(\hat w_M(t_k-s) + r\tilde w_1(t_k-s)) \drm s \\ 
	&\quad \quad - \frac{\tau}{t_{k+1} - t_k}\bigg(\int_0^{t_{k+1}} \tilde w_2(s)(\hat w_M(t_{k+1}-s) + r\tilde w_1(t_{k+1}-s)) \drm s - \int_0^{t_k} \tilde w_2(s)(\hat w_M(t_k-s) + r\tilde w_1(t_k-s)) \drm s\bigg)\bigg| \\
	&\leq 2(1+\omega)\delta t\max\{\hat w_M(0),\hat w_M(\delta t) \} + 2r\delta t(1+\omega)^2 + r\omega(1+\omega)
\end{split}
\end{equation}
by following similar bounding techniques to the first integral differences. Putting this all together gives
\begin{equation}
	\|\Pi_\infty z(w_1,w_2)\|_\infty \leq W^1_kr + W^2_kr^2 
\end{equation}
where
\begin{equation}\label{W12kr}
	\begin{split}
		W^1_k &:= 4(1+\omega)\delta t \sup \mathcal{E}_k\bigg(\max\{\hat w_M(1),\hat w_M(1-\delta t) \}+  2\max\{\hat w_M(0),\hat w_M(\delta t) \}\bigg)\\
		W^2_k &:= 6(1+\omega)\sup \bigg(2\delta t(1+\omega) + \omega\bigg).
	\end{split}
\end{equation}
We summarise these results with the following lemma.

\begin{lem}
	We have $Z_\infty(r) = W^1_kr + W^2_kr^2$, as given in \eqref{W12kr}, to bound $\|\Pi_\infty z(w_1,w_2)\|_\infty$ as a function of $r$.
\end{lem}

\begin{rmk}
	Notice that the $\mathcal{O}(r^2)$ terms in the $Z_\infty(r)$ definition can be made small by taking $\omega > 0$ small. The linear terms in $r$ scale with the step size $\delta t$, hence meaning that the coefficients of $Z_\infty(r)$ can be made small by taking the step size and $\omega$ small. In particular, the $W_k^1$ bound can be made to be less than $\omega$ by taking $\delta t$ small enough so that the $p_\infty(r)$ radii polynomial can be negative for some choice of $r$.
\end{rmk}

We now conclude this section by briefly discussing the computational aspects of the proof. We are required to verify that the radii polynomials are all negative for a choice of $r$ and $\omega$. We first compute the numerical approximation $\hat w_M$ with $M=1000$ in MATLAB using the Newton-trust routine \verb1fsolve1. The remainder of the proof is carried out using IntLab \cite{Rump1999} to carry out the interval arithmetic calculations. We fix $\omega=0.02$ (though other values are also possible) and we find that for $r\in[1.652\times10^{-5},0.0892]$ the radii polynomials are all negative. The code for performing this verification can be found in the GitHub repository associated to this paper.

\bibliographystyle{abbrv}
\bibliography{Paper.bbl}
\end{document}